\documentclass[11pt]{article}

\usepackage{fullpage}
\usepackage[T1]{fontenc} 
\usepackage{amsmath,amssymb}
\usepackage{amsopn,amsthm}

\usepackage{graphicx} 
\usepackage{color}
\usepackage{psfrag}

\usepackage{mathrsfs} 
\usepackage{bm}

\usepackage{algpseudocode} 
\usepackage{algorithm}
\floatstyle{plain} \newfloat{myalgo}{tbhp}{mya}

{\begin{myalgo}[#1]
    \centering
    \begin{minipage}{0.9\textwidth}
      \begin{algorithm}[H]}%
      {\end{algorithm}
    \end{minipage}
  \end{myalgo}}

\usepackage[colorlinks=true]{hyperref}
\hypersetup{urlcolor=blue, citecolor=blue, linkcolor=blue}

\newtheorem{theorem}{Theorem}[section]

\newtheorem{lemma}[theorem]{Lemma}
\newtheorem{corollary}[theorem]{Corollary}

\theoremstyle{definition}

\newtheorem{example}[theorem]{Example}

\theoremstyle{remark} \newtheorem{remark}[theorem]{Remark}

\numberwithin{equation}{section}
\numberwithin{figure}{section}
\numberwithin{algorithm}{section}

\newcommand{\wh}{\widehat}

\newcommand{\Sum}{\ensuremath{\mathop{\sum}\limits}}

\newcommand{\Tc}{\ensuremath{T_{c}}}
\newcommand{\Lt}{\ensuremath{{L^{2}}}}
\newcommand{\lt}{\ensuremath{{l^{2}}}}

\renewcommand{\l}[1]{\ensuremath{{l^{#1}}}}
\renewcommand{\L}[1]{\ensuremath{{L^{#1}}}}

\newcommand{\third}{\ensuremath{\frac{1}{3}}}
\newcommand{\half}{\ensuremath{\frac{1}{2}}}
\newcommand{\Argmin}{\mathop\mathrm{argmin}\limits}

\newcommand{\Linfty}{\ensuremath{{L^{\infty}}}}

\newcommand{\field}[1]{{\mathbb{#1}}}
 
\newcommand{\N}{\field{N}}
\newcommand{\R}{\field{R}} 
\newcommand{\Z}{\field{Z}}

\newcommand{\Ccal}{\mathscr{C}}

\newcommand{\Fcal}{\mathcal{F}}

\newcommand{\Ncal}{\mathcal{N}}

\newcommand{\loc}{{\mathrm{loc}}}
\newcommand{\NE}{{\mathrm{NE}}}
\newcommand{\nr}{{\mathrm{NR}}}

\newcommand{\tm}{\subseteq}

\newcommand{\dr}{\, \dif r}
 
\newcommand{\dt}{\, \dif t}

\newcommand{\dy}{\, \dif y}

\newcommand{\dtheta}{\, \dif \theta} 

\newcommand{\dxi}{\, \dif\xi}

\newcommand{\alphatilde}{{\widetilde\alpha}}
\newcommand{\betatilde}{{\widetilde\beta}}

\newcommand{\rmi}{\mathrm{i}}

\newcommand{\eps}{\varepsilon}

\DeclareMathAlphabet{\mathbi}{\encodingdefault}{\rmdefault}{\bfdefault}{\itdefault}
\DeclareRobustCommand{\vec}[1]{\ifmmode\mathbi{#1}\else\textbf{\textit{#1}}\fi}

\DeclareMathOperator{\dif}{d\!}

\DeclareMathOperator{\supp}{supp}

\begin{document}

\title{Uncertainty principles for inverse source problems, far field splitting and data completion}
\author{Roland
  Griesmaier\footnote{Institut f\"ur Mathematik,
    Universit\"at W\"urzburg, 97074 W\"urzburg, Germany
    ({\tt roland.griesmaier@uni-wuerzburg.de})}\;
  and John Sylvester\footnote{Department of Mathematics,
    University of Washington, Seattle, Washington 98195,
    U.S.A. ({\tt sylvest@uw.edu})}
}
\date{\today}

\maketitle

\begin{abstract}
Starting with far field data of time-harmonic acoustic or
  electromagnetic waves radiated by a collection of compactly
  supported sources in two-dimensional free space, we develop criteria
  and algorithms for the recovery of the far field components radiated
  by each of the individual sources, and the simultaneous
  restoration of missing data segments. Although both parts of this
  inverse problem are severely ill-conditioned in general, we give
  precise conditions relating the wavelength, the diameters of the
  supports of the individual source components and the distances
  between them, and the size of the missing data segments, which
  guarantee that stable recovery in presence of noise is possible.
  The only additional requirement is that a priori information on the
  approximate location of the individual sources is available.  We
  give analytic and numerical examples to confirm the sharpness of our
  results and to illustrate the performance of corresponding
  reconstruction algorithms, and we discuss consequences for stability
  and resolution in inverse source and inverse scattering
  problems.
\end{abstract}

{\small\noindent
  Mathematics subject classifications (MSC2010): 35R30, (65N21)
  \\\noindent
  Short title: Uncertainty principles for inverse source problems
}

\section{Introduction}
\label{sec:Introduction}
In signal processing, a classical uncertainty principle limits the
time-bandwidth product $|T||W|$ of a signal, where $|T|$ is the
measure of the support of the signal $\phi(t)$, and $|W|$ is the
measure of the support of its Fourier transform
$\wh{\phi}(\omega)$ (cf., e.g., \cite{DonSta89}). 
A very  elementary  formulation of that principle is 
\begin{equation}
  \label{eq:intro1}
  |\langle \phi,\psi\rangle| 
  \,\le\, \sqrt{|T||W|} \|\phi\|_2 \|\psi\|_2
\end{equation}
whenever $\supp \phi \tm T$ and $\supp \wh{\psi}\tm W$ . 

In the inverse source problem, the far field radiated by a source $f$
is its \emph{restricted} (to the unit sphere) \emph{Fourier
  transform}, and the operator that maps the restricted Fourier
transform of $f(x)$ to the restricted Fourier transform of its
translate $f(x+c)$ is called the \emph{far field translation
  operator}. 
We will prove an uncertainty principle analogous to \eqref{eq:intro1},
where the role of the Fourier transform is replaced by the far field
translation operator. 
Combining this principle with a \emph{regularized Picard criterion},
which characterizes the non-evanescent (i.e., detectable) far fields
radiated by a (limited power) source supported in a ball provides
simple proofs and extensions of several results about locating the
support of a source and about splitting a far field radiated by
well-separated sources into the far fields radiated by each source
component. 

We also combine the regularized Picard criterion with a more
conventional uncertainty principle for the map from  a far field  in
$L^{2}(S^{1})$ to its Fourier coefficients. 
This leads to a data completion algorithm which tells us that we can
deduce missing data (i.e.\ on part of  $S^{1}$) if we know \textit{a
  priori}  that the source has small support. 
All of these results can be combined so that we can simultaneously
complete the data and split the far fields into the components
radiated by well-separated sources. 
We discuss both  $l^{2}$ (least squares) and $l^{1}$ (basis pursuit)
algorithms to accomplish this. 

Perhaps the most significant point is that all of these algorithms
come with bounds on their condition numbers (both the splitting and
data completion problems are linear) which we show are sharp in their
dependence on geometry and wavenumber.  
These results highlight an important difference between the inverse
source problem and the inverse scattering problem. 
The conditioning of the linearized inverse scattering problem does not
depend on wavenumber, which means that the conditioning does not
deteriorate as we increase the wavenumber in order to increase
resolution. 
The conditioning for splitting and data completion for the inverse
source problem does, however, deteriorate with increased wavenumber,
which means the dynamic range of the sensors must increase with
wavenumber to obtain higher resolution. 

We note that applications of classical uncertainty principles for the
one-dimensional Fourier transform to data completion for band-limited
signals have been developed in \cite{DonSta89}.
In this classical setting a problem that is somewhat similar to far
field splitting is the representation of highly sparse signals in
overcomplete dictionaries. 
Corresponding stability results for basis pursuit reconstruction
algorithms have been established in \cite{DonElaTem061}. 

The numerical algorithms for far field splitting that we are going to
discuss have been developed and analyzed in
\cite{GriHanSyl14,GriSyl16}. 
The novel mathematical contribution of the present work is the
stability analysis for these algorithms based on new uncertainty
principles, and their application to data completion. 
For alternate approaches to far field splitting that however, so far,
lack a rigorous stability analysis we refer to
\cite{HasLiuPot07,PotFazNel10} (see also \cite{GroKraNatAss15} for a
method to separate time-dependent wave fields due to multiple
sources).

This paper is organized as follows.
In the next section we provide the theoretical background for the
direct and inverse source problem for the two-dimensional Helmholtz
equation with compactly supported sources.
In section~\ref{sec:RegPicard} we discuss the singular value
decomposition of the restricted far field operator mapping sources
supported in a ball to their radiated far fields, and we formulate the 
regularized Picard criterion to characterize non-evanescent far
fields. 
In section~\ref{sec:Uncertainty} we discuss uncertainty principles
for the far field translation operator and for the Fourier expansion
of far fields, and in section~\ref{sec:L2Corollaries} we utilize those
to analyze the stability of least squares algorithms for far field
splitting and data completion.
Section~\ref{sec:L1Corollaries} focuses on corresponding results for
$\l1$ algorithms. 
Consequences of these stability estimates related to conditioning and
resolution of reconstruction algorithms for inverse source and inverse
scattering problems are considered in section~\ref{sec:Conditioning},
and in section~\ref{sec:AnalyticExample}--\ref{sec:NumericalExamples}
we provide some analytic and numerical examples.

\section{Far fields radiated by compactly supported sources}
\label{sec:Fourier}
Suppose that $f \in L^2_0(\R^2)$ represents a compactly
supported acoustic or electromagnetic source in the plane.
Then the time-harmonic wave $v \in H^1_\loc(\R^2)$
radiated by $f$ at \emph{wave number} $k>0$ solves
the \emph{source problem} for the Helmholtz equation
\begin{equation*}
  - \Delta v -k^{2} v
  \,=\, k^{2}g \qquad \text{in $\R^2$}\,,
\end{equation*}
and satisfies the \emph{Sommerfeld radiation condition}
\begin{equation*}
  \lim_{r\to\infty} \sqrt{r}\Bigl(\frac{\partial v}{\partial r} 
  - \rmi k v \Bigr) 
  \,=\, 0 \,, \qquad r=|x| \,.
\end{equation*}
We include the extra factor of  $k^{2}$ on the right hand side so that
both  $v$ and  $g$ scale (under dilations) as functions; i.e., if
$u(x)=v(kx)$ and  $f(x)=g(kx)$, then 
\begin{equation}
  \label{eq:SourceProblem}
  - \Delta u - u
  \,=\, f \qquad \text{in $\R^2$} 
  \qquad \text{and} \qquad 
  \lim_{r\to\infty} \sqrt{r}\Bigl(\frac{\partial u}{\partial r} 
  - \rmi u \Bigr) 
  \,=\, 0 \,.
\end{equation}
With this scaling, distances are measured in wavelengths\footnote{One
  unit represents $2\pi$ wavelengths.}, and this allows us to set
$k=1$ in our calculations, and then easily restore the dependence on
wavelength when we are done.

The \emph{fundamental solution} of the Helmholtz equation (with
$k=1$)  in two dimensions is 
\begin{equation*}
  \Phi(x) \,:=\, \frac \rmi{4} H^{(1)}_0(|x|) \,, \qquad 
  x\in\R^2\setminus\{0\}\,,
\end{equation*}
so the solution to \eqref{eq:SourceProblem} can be written as a volume
potential 
\begin{equation*}
  u(x) \,=\, \int_{\R^2} \Phi(x-y) f(y) \dy  \,, \qquad 
  x\in\R^2\,.
\end{equation*}
The asymptotics of the Hankel function tell us that 
\begin{equation*}
  u(x) 
  \,=\, \frac{e^{\frac{\rmi\pi}4}}{\sqrt{8\pi}}
  \frac{e^{\rmi r}}{\sqrt{r}} \alpha(\theta_x) 
  + O\left(r^{-\frac32}\right) 
  \qquad \text{as $r\to\infty$\,,}
\end{equation*}
where  $x=r\theta_x$ with  $\theta_x\in S^1$, and 
\begin{equation}
  \label{eq:Farfield} 
  \alpha(\theta_x) 
  \,=\, \int_{\R^2} e^{-\rmi\theta_x\cdot y} f(y) \dy \,.
\end{equation}
The function $\alpha$ is called the \emph{far field} radiated by the
source $f$, and equation \eqref{eq:Farfield} shows that the 
\emph{far field operator} $\Fcal$, which maps  $f$ to $\alpha$ is a
\emph{restricted Fourier transform}, i.e. 
\begin{equation}
  \label{eq:DefFarfieldOperator}
  \Fcal:\, L^2_0(\R^2) \to L^2(S^1) \,, \quad 
  \Fcal f \,:=\, \wh{f}\big|_{S^1} \,.
\end{equation}

The goal of the inverse source problem is to deduce properties of an
unknown source $f \in L^2_0(\R^2)$ from observations of the  far
field.  
Clearly,  any compactly supported source with Fourier transform that
vanishes on the unit circle is in the nullspace  $\Ncal(\Fcal)$ of the
far field operator. 
We call $f\in \Ncal(\Fcal)$ a \emph{non-radiating source} because a
corollary of Rellich's lemma and unique continuation is that, if the
far field vanishes, then the wave $u$ vanishes on the unbounded
connected component of the complement of the support of  $f$. 
The nullspace of  $\Fcal$ is exactly
\begin{equation*}
  \Ncal(\Fcal) \,=\, \{ g = -\Delta v - v \;|\; v
  \in H^2_0(\R^2) \} \,.
\end{equation*}

Neither the source $f$ nor its support is uniquely determined by the
far field, and, as  non-radiating sources can have arbitrarily large
supports, no upper bound on the support is possible.  
There are, however, well defined notions of lower bounds. 
We say that a compact set $\Omega \tm \R^2$ \emph{carries}
$\alpha$, if every open neighborhood of $\Omega$ supports a source $f
\in L^2_0(\R^2)$ that radiates $\alpha$. 
The \emph{convex scattering support} $\Ccal(\alpha)$ of $\alpha$, as
defined in \cite{KusSyl03} (see also \cite{KusSyl05,Syl06}), is the
intersection of all compact convex sets that carry $\alpha$. 
The set $\Ccal(\alpha)$ itself carries $\alpha$, so that
$\Ccal(\alpha)$ is the smallest convex set which carries the far field
$\alpha$, and  the convex hull of the support of the ``true'' source
$f$ must contain $\Ccal(\alpha)$. 
Because two disjoint compact sets with connected complements cannot
carry the same far field pattern (cf.~\cite[lemma~6]{Syl06}), it
follows that $\Ccal(\alpha)$ intersects any connected component of
$\supp(f)$, as long as the corresponding source component is not
non-radiating. 

In \cite{Syl06}, an analogous notion, the \emph{UWSCS support}, was
defined, showing that any far field with a compactly supported source
is carried by  a smallest union of well-separated convex sets
(well-separated means that the distance between any two connected
convex components is strictly greater than the diameter of any
component). 
A corollary is that it makes theoretical sense to look for the support
of a source with components that are small compared to the distance
between them. 

Here, as in previous investigations \cite{GriHanSyl14,GriSyl16}, we
study the well-posedness issues surrounding numerical algorithms to
compute that support.

\section{A regularized Picard criterion}
\label{sec:RegPicard}
If we consider the restriction of the source to far field map  $\Fcal$
from \eqref{eq:DefFarfieldOperator} to sources supported in the ball
$B_R(0)$ of radius  $R$ centered at the origin, i.e.,
\begin{equation}
  \label{eq:DefRestFarfieldOperator}
  \Fcal_{B_R(0)}:\, L^2(B_R(0)) \to L^2(S^1) \,, \quad 
  \Fcal_{B_R(0)} f \,:=\, \wh{f}\big|_{S^1} \,,
\end{equation}
we can write out a full singular value decomposition. 
We decompose  $f\in\L2(B_R(0))$ as  
\begin{equation*}
  f(x) \,=\,
  \biggl( \sum_{n=-\infty}^{\infty} 
  f_n\, \rmi^n J_n(|x|) e^{\rmi n \varphi_x} \biggr)
  \;\oplus f_\nr(x) \,, \qquad 
  x = |x| (\cos\varphi_x,\sin\varphi_x) \in B_R(0) \,,
\end{equation*}
where $\rmi^n J_n(|x|) e^{\rmi n \varphi_x}$, $n\in\Z$, span the
closed subspace of \textit{free sources}, which satisfy 
\begin{equation*}
  -\Delta u - u \,=\, 0\qquad \mathrm{in\ } B_{R}(0) \,,
\end{equation*}
and $f_\nr$ belongs to the orthogonal complement of that
subspace; i.e., $f_\nr$ is a non-radiating
source.\footnote{Throughout, we identify $f \in L^2(B_R(0))$ with its
  continuation to $\R^2$ by zero whenever appropriate.}
The restricted far field operator $\Fcal_{B_{R}(0)}$ maps
\begin{equation}
  \label{eq:SVD}
  \Fcal_{B_{R(0)}}:\, 
  \rmi^n J_n(|x|) e^{\rmi n \varphi_x} 
  \mapsto s_{n}^{2}(R) e^{\rmi n \theta} \,,
\end{equation}
where 
\begin{equation}
  \label{eq:Defsn}
  s_{n}^{2}(R) 
  \,=\, 2\pi \int_{0}^{R}J_n^2(r) r \dr \,.
\end{equation}

Denoting the Fourier coefficients of a far field $\alpha \in L^2(S^1)$ by
\begin{equation}
  \label{eq:DefFourierCoeffs}
  \alpha_{n} \,:=\,
  \frac{1}{\sqrt{2\pi}}
  \int_{S^1} \alpha(\theta) e^{\rmi n \theta} \dtheta \,, 
  \qquad n\in\Z \,,
\end{equation}
so that
\begin{equation*}
  \alpha(\theta) \,=\,  \sum_{n=-\infty}^\infty \alpha_{n}\frac{e^{\rmi
      n\theta}}{\sqrt{2\pi}} \,, \qquad \theta \in S^1 \,,
\end{equation*}
and
\begin{equation}
  \label{eq:Plancherel}
  \|\alpha\|_{\L2(S^1)}^{2} 
  \,=\, \sum_{n=-\infty}^\infty |\alpha_{n}|^{2} 
\end{equation}
by Parseval's identity, an immediate consequence of \eqref{eq:SVD} is
that 
\begin{equation}
  \label{eq:falpha*}
  f_\alpha^*(x) \,=\, \frac1{\sqrt{2\pi}} \sum_{n=-\infty}^{\infty} 
  \frac{\alpha_{n}}{s_{n}(R)^{2}} \, \rmi^n J_{n}(|x|) 
  e^{\rmi n \varphi_x} \,, \qquad x \in B_R(0) \,,
\end{equation}
which has $\L2$-norm
\begin{equation*}
  \|f_\alpha^*\|_{L^2(B_R(0))}^2 
  \,=\,  \frac1{2\pi} \sum_{n=-\infty}^{\infty} 
  \frac{|\alpha_{n}|^2}{s_{n}^2(R)} \,,
\end{equation*}
is the source with smallest  \L2-norm  that is supported in
$B_{R}(0)$ and radiates the far field $\alpha$. 
We refer to  $f_\alpha^*$ as the \textit{minimal power source}
because, in electromagnetic applications,  $f_\alpha^*$ is
proportional to current density, so that, in a system with a constant
internal resistance, $\|f_\alpha^*\|_{L^2(B_R(0))}^2$ is proportional
to the input power required to radiate a far field. 
Similarly,  $\|\alpha\|^2_{L^2(S^1)}$ measures the radiated
power of the far field. 

The squared singular values $\{s_n^2(R)\}$ of the restricted Fourier
transform $\Fcal_{B_R(0)}$ have a number of interesting properties
with immediate consequences for the inverse source problem;
full proofs of the results discussed in the following can be found in
appendix~\ref{app:sn2}.
The squared singular values satisfy
\begin{equation}
  \label{eq:Sumsn}
  \sum_{n=-\infty}^\infty s_n^2(R)
  \,=\, \pi  R^2 \,,
\end{equation}
and $s_n^2(R)$ decays rapidly as a function of $n$ as soon as 
$|n| \geq R$, 
\begin{equation}
  \label{eq:Decaysn2}
  s_n^2(R) 
  \,\leq\, \frac{\pi
    2^{\frac23}n^{\frac23}}{3^{\frac43}\bigl(\Gamma(\frac23)\bigr)^2}
  \Bigl( \frac{n+\frac12}{n} \Bigr)^{n+1}
  \Bigl( \frac{R^2}{n^2} e^{1-\frac{R^2}{n^2}} \Bigr)^n 
  \frac{R^2}{n^2} \qquad \text{if } |n| \geq R \,.
\end{equation}
Moreover, the odd and even squared singular values, $s_n^2(R)$, are
decreasing (increasing) as functions of $n\geq0$ ($n\leq0$), and
asymptotically
\begin{equation}
  \label{eq:Qualitativesn2}
  \lim_{R\rightarrow\infty} \frac{s_{\lceil\nu R\rceil}^{2}(R)}{2R}
  \,=\,
  \begin{cases}
    \sqrt{1-\nu^{2}} & \nu \le 1 \,,\\
    0 & \nu \ge 1 \,,
  \end{cases}
\end{equation}
where $\lceil \nu R\rceil$ denotes the smallest integer that is
greater than or equal to $\nu R$. 
This can also be seen in figure~\ref{fig:sn2}, where we include plots
of $s_n^2(R)$ (solid line) together with plots of the asymptote
$2\sqrt{R^2-n^2}$ (dashed line) for $R=10$ (left) and $R=100$ (right).
\begin{figure}
  \centering
  \includegraphics[height=5.5cm]{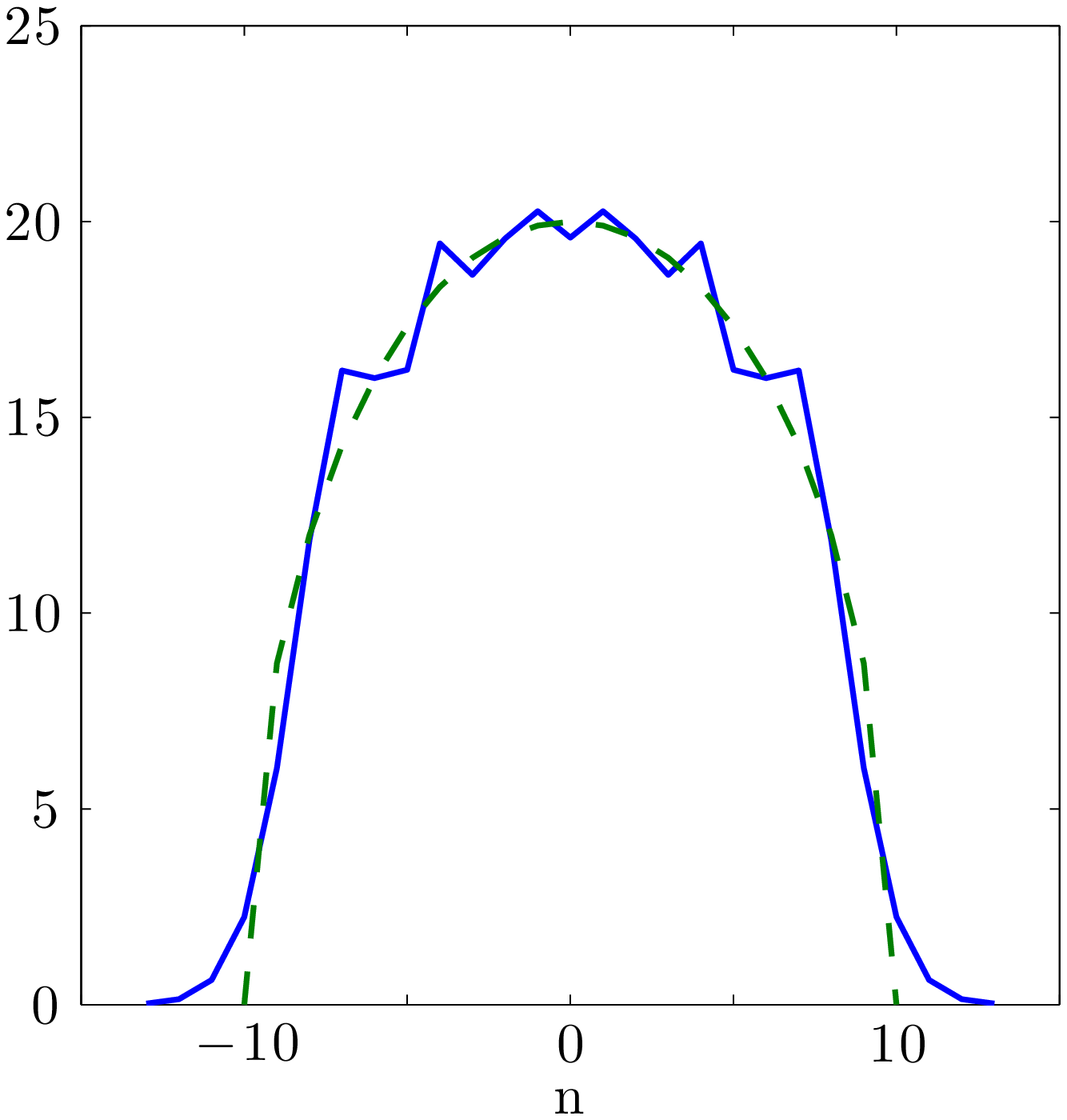}
  \qquad
  \includegraphics[height=5.5cm]{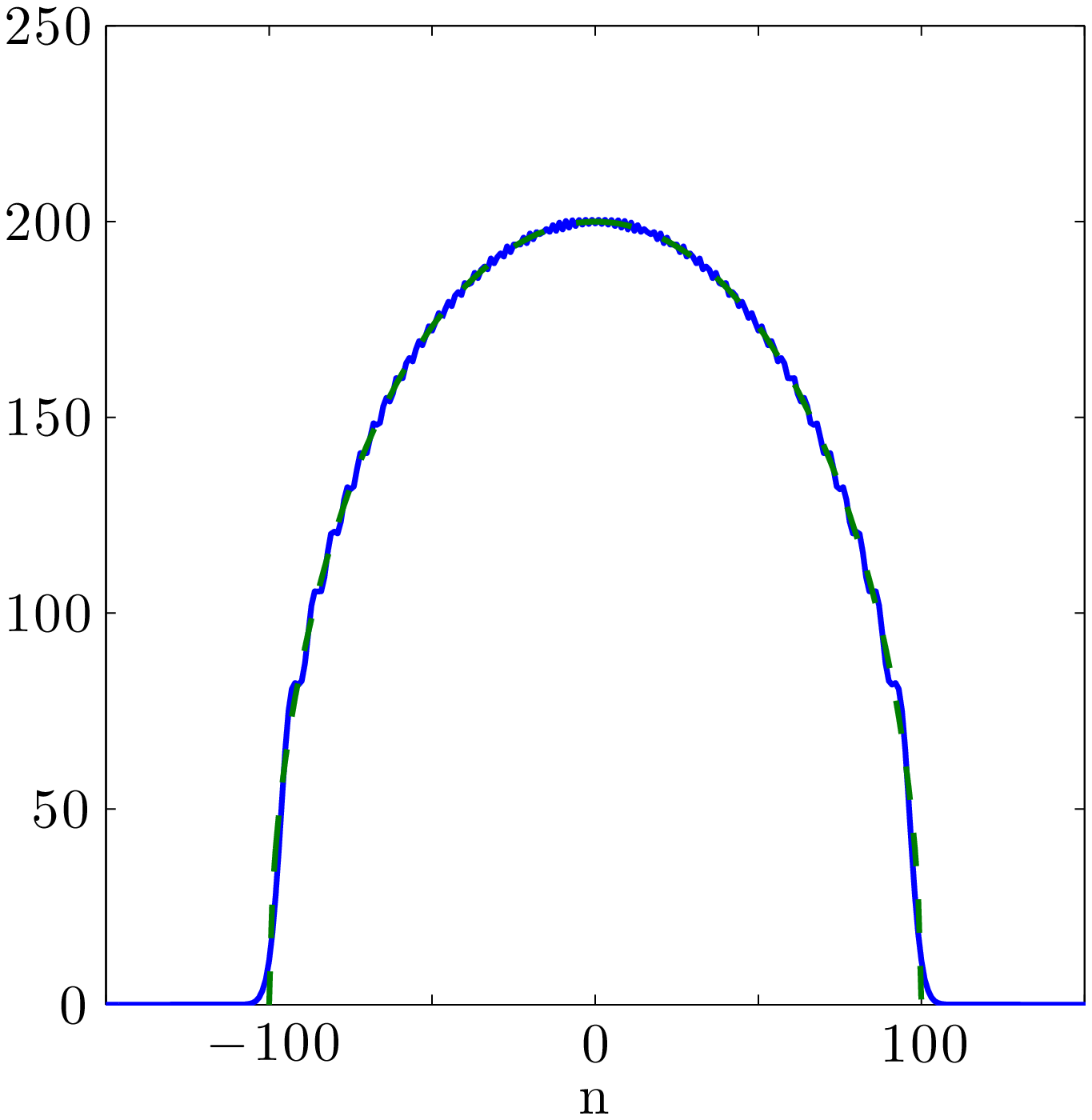}
  \caption{\small Squared singular values $s_n^2(R)$ (solid
    line) and asymptote $2\sqrt{R^2-n^2}$ (dashed line) for
    $R=10$ (left) and $R=100$ (right).}
  \label{fig:sn2}
\end{figure}
The asymptotic regime in \eqref{eq:Qualitativesn2} is already reached 
for moderate values of $R$. 

The forgoing yields a very explicit understanding of the restricted
Fourier transform $\Fcal_{B_R(0)}$.
For $|n|\lesssim R$ the singular values $s_n(R)$ are uniformly large,
while for $|n|\gtrsim R$ the $s_n(R)$ are close to zero, and it is
seen from \eqref{eq:Sumsn}--\eqref{eq:Qualitativesn2} as well as from
figure~\ref{fig:sn2} that as $R$ gets large the width of the
$n$-interval in which $s_n(R)$ falls from uniformly large to zero
decreases.
Similar properties are known for the singular values of more
classical restricted Fourier transforms (see \cite{Sle83}).

A physical source  has \emph{limited power}, which we denote by $P>0$,
and a receiver has a \emph{power threshold}, which we denote by $p>0$. 
If the radiated far field has power less than  $p$, the receiver 
cannot detect it. 
Because $s_{-n}^2(R)=s_n^2(R)$ and the odd and even squared singular
values, $s_{n}^{2}(R)$, are decreasing as functions of $n\geq0$, we
may define: 
\begin{equation}
  \label{eq:DefNrpp}
  N(R,P,p) \,:=\, \sup_{s_{n}^{2}(R)\ge{2\pi}\frac{p}{P}}n \,.
\end{equation}
So, if $\alpha \in L^2(S^1)$ is a far field radiated by a limited
power source supported in $B_R(0)$ with
${\|f_\alpha^*\|_{L^2(B_R(0))}^2 \leq P}$, then, for $N = N(R,P,p)$
\begin{equation*}
  P
  \,\geq\, \frac1{2\pi} 
  \sum_{|n|>N} \frac{|\alpha_{n}|^2}{s_{n}^2(R)} 
  \,\geq\, \frac1{2\pi} \frac{1}{s_{N+1}^2(R)} 
  \sum_{|n|>N} |\alpha_{n}|^2
  \,>\, \frac{P}{p} \sum_{|n|>N} |\alpha_{n}|^2 \,.
\end{equation*}
Accordingly, $\sum_{|n|\geq N} |\alpha_{n}|^2 < p$ is below the
power threshold.
So the subspace of detectable far fields, that can be radiated by a
power limited source supported in $B_{R}(0)$ is:
\begin{equation*}
  V_\NE \,:=\, \Bigl\{ \alpha \in L^2(S^1) \;\Big|\; 
  \alpha(\theta) = \sum_{n=-N}^{N} \alpha_{n}e^{in\theta} \Bigr\} \,. 
\end{equation*}

We refer to $V_\NE$ as the subspace of \textit{non-evanescent far
  fields}, and to the orthogonal projection of a far field onto this subspace as
the \textit{non-evanescent} part of the far field.  
We use the term \textit{non-evanescent} because it is the phenomenon
of evanescence that explains why the the singular values
$s_{n}^{2}(R)$ decrease rapidly for $|n|\gtrsim R$, resulting in the fact that,
for a wide range of $p$ and $P$, $R<N(R,p,P)<1.5R$, if $R$ is
sufficiently large.
\begin{figure}
  \centering
  \includegraphics[height=5.5cm]{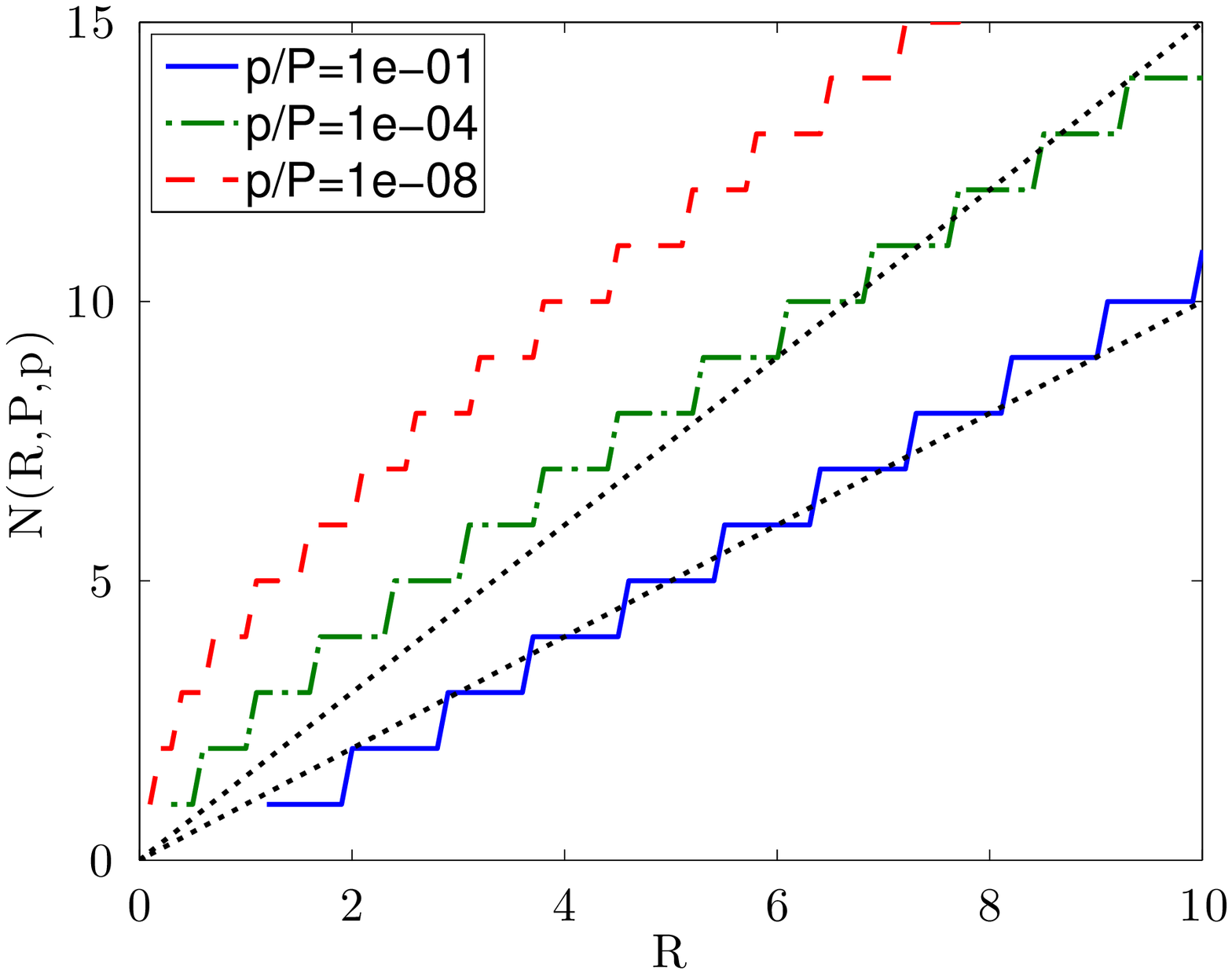}
  \qquad \quad
  \includegraphics[height=5.5cm]{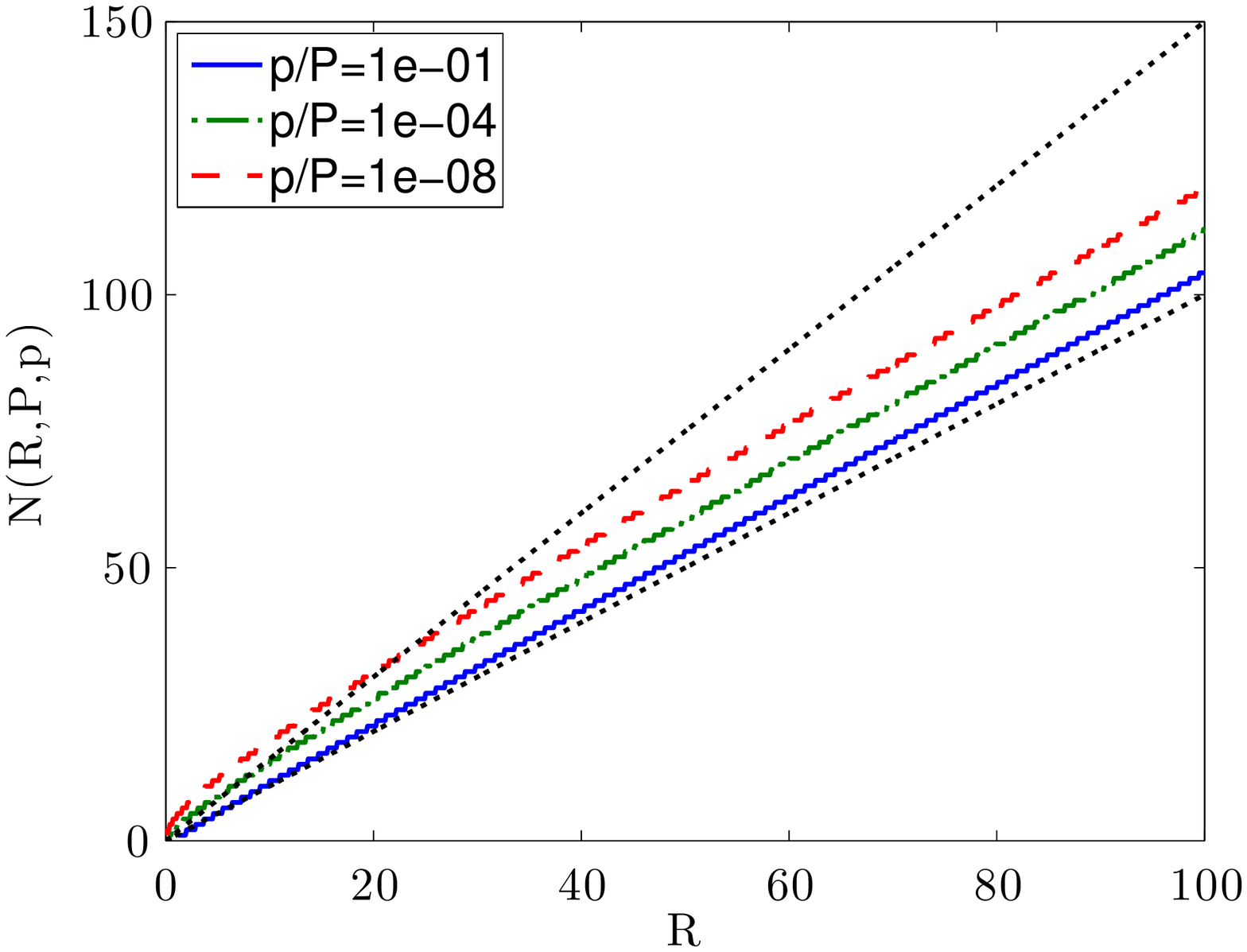}
  \caption{\small Threshold $N(R,P,p)$ as function of $R$ for
    different values of $p/P$. Dotted lines correspond to $g_1(R)=R$
    and $g_{1.5}(R)=1.5R$.}
  \label{fig:N-RPp}
\end{figure}
This is also illustrated in figure~\ref{fig:N-RPp}, where we include
plots of $N(R,P,p)$ from \eqref{eq:DefNrpp} for $p/P=10^{-1}$,
$p/P=10^{-4}$, and $p/P = 10^{-8}$ and for varying $R$.
The dotted lines in these plots correspond to $g_1(R)=R$ and
$g_{1.5}(R)=1.5R$, respectively.

\section{Uncertainty principles for far field translation}
\label{sec:Uncertainty}
In the inverse source problem, we seek to recover information about
the size and location of the support of a source from observations of
its far field. 
Because the far field is a restricted Fourier transform, the formula 
for the Fourier transform of the translation of a function:
\begin{equation*}
  \wh{f(\cdot+c)} (\theta) 
  \,=\, e^{\rmi c\cdot\theta}\wh{f}(\theta) \,,
  \qquad \theta \in S^1 \,,\; c\in\R^2 \,,
\end{equation*}
plays an important role. 
We use  $T_c$ to denote the map from  $\Lt(S^1)$ to itself given by 
\begin{equation}
  \label{eq:DefTc}
  \Tc: \alpha \mapsto  e^{\rmi c\cdot\theta}\alpha \,.
\end{equation}
The mapping  $\Tc$ acts on the Fourier coefficients $\{\alpha_{n}\}$ of
$\alpha$ as a convolution operator, i.e., the Fourier coefficients
$\{\alpha_{m}^{c}\}$ of $\Tc\alpha$ satisfy 
\begin{equation}
  \label{eq:TcConvolution}
  \alpha^{c}_{m} 
  \,=\,
  \sum_{n=-\infty}^{\infty} \alpha_{m-n} 
  \bigl( \rmi^{n} J_{n}(|c|) e^{\rmi n \varphi_c} \bigr) \,,
  \qquad m\in\Z \,,
\end{equation}
where  $|c|$ and  $\varphi_c$ are the polar coordinates of $c$.
Employing a slight abuse of notation, we also use $\Tc$ to denote the
corresponding operator from $\lt$ to itself that maps 
\begin{equation}
  \label{eq:DefTcAbuse}
  \Tc: \{\alpha_n\} \mapsto \{\alpha^{c}_{m}\} \,.
\end{equation}
Note that $\Tc$ is a unitary operator on \l2, i.e.\ $T_c^* = T_{-c}$.

The following theorem, which we call an \textit{uncertainty principle 
  for the translation operator}, will be the main 
ingredient in our analysis of far field splitting.

\begin{theorem}[Uncertainty principle for far field translation]
  \label{th:Uncertainty1}
  Let $\alpha,\beta\in\L2(S^1)$ such that the corresponding Fourier
  coefficients $\{\alpha_n\}$ and $\{\beta_n\}$ satisfy
  $\supp\{\alpha_{n}\}\tm W_1$ and $\supp\{\beta_{n}\}\tm W_2$
  with $W_1,W_2 \tm \Z$, and let $c\in\R^2$.
  Then, 
  \begin{equation*}
    |\langle\alpha,\Tc\beta\rangle_{L^2(S^1)}| 
    \,\le\, \frac{\sqrt{|W_1||W_2|}}{|c|^{1/3}} \|\alpha\|_{L^2(S^1)}
    \|\beta\|_{L^2(S^1)} \,.
  \end{equation*}
\end{theorem}

We will frequently be discussing properties of a far field $\alpha$
and those of its Fourier coefficients. 
The following notation will be a useful shorthand:
\begin{align}
  \label{eq:LpNorm}
  \|\alpha\|_{L^{p}} 
  &\,=\, \Bigl( \int_{S^1} |\alpha(\theta)|^{p} \dtheta \Bigr)^{1/p} \,,
  && 1\leq p\leq \infty\,,\\
  \label{eq:lpNorm}
  \|\alpha\|_{l^{p}} 
  &\,=\, \Bigl( \sum_{n=-\infty}^{\infty}|\alpha_{n}|^{p} \Bigr)^{1/p} \,,
  && 1\leq p\leq \infty\,.
\end{align}
The notation emphasizes that we treat the representation of the
function  $\alpha$ by its values, or by the sequence of its Fourier
coefficients as simply a way of inducing different norms. 
That is, both \eqref{eq:LpNorm} and \eqref{eq:lpNorm} describe
different norms of the same function on $S^{1}$. 
Note that, because of the Plancherel equality \eqref{eq:Plancherel},
$\|\alpha\|_{\Lt} = \|\alpha\|_{\lt}$, so we may just write
$\|\alpha\|_2$, and we write $\langle\cdot,\cdot\rangle$ for the
corresponding inner product.

\begin{remark}
  We will extend the notation a little more and refer to the
  support of $\alpha$ in $S^{1}$ as its $\L0$-support and denote by
  $\|\alpha\|_{\L0}$ the measure of $\supp(\alpha)\tm S^1$.
  We will call the indices of the nonzero Fourier coefficients in its
  Fourier series expansion the $\l0$-support of $\alpha$, and use
  $\|\alpha\|_{l^{0}}$ to denote the number of non-zero coefficients. 
  \hfill$\lozenge$
\end{remark}

With this notation, theorem~\ref{th:Uncertainty1} becomes

\begin{theorem}[Uncertainty principle for far field translation]
  \label{th:Uncertainty2}
  Let $\alpha, \beta \in \L2(S^1)$, and let $c\in\R^2$.
  Then,
  \begin{equation}
    \label{eq:Uncertainty2}
    |\langle\alpha,\Tc\beta\rangle|
    \,\le\,
    \frac{\sqrt{\|\alpha\|_{l_{0}}\|\beta\|_{l_{0}}}}{|c|^{1/3}}
    \|\alpha\|_2 \|\beta\|_2 \,.
  \end{equation}
\end{theorem}

We refer to theorem~\ref{th:Uncertainty2} as an uncertainty principle,
because, if we could take  $\beta=T_{c}^*\alpha$ in
\eqref{eq:Uncertainty2}, it would yield 
\begin{equation}
  \label{eq:Uncertainty2a}
  1 \,\le\, \frac{\|\alpha\|_{\l0}
    \|T_{c}^*\alpha\|_{\l0}}{|c|^{2/3}} \,.
\end{equation}
As stated, \eqref{eq:Uncertainty2a} is is true but not useful, 
because $\|\alpha\|_{\l0}$ and $\|T_{c}^*\alpha\|_{\l0}$ cannot
simultaneously be finite.\footnote{This would imply, using
  \eqref{eq:falpha*}, that  $\alpha$ could have been radiated by a source
  supported in an arbitrarily small ball centered at the origin, or 
  centered at  $c$, but Rellich's lemma and unique continuation show
  that no nonzero far field can have two sources with disjoint
  supports.}
We present the corollary only to illustrate the close analogy to the
theorem 1 in \cite{DonSta89}, which treats the discrete Fourier
transform (DFT) on sequences of length $N$:

\begin{theorem}[Uncertainty principle for the Fourier transform
  (Donoho, Stark~\cite{DonSta89})]  
  If $x$ represents the sequence $\{x_{n}\}$ for $n=0,\ldots,N-1$ and
  $\wh{x}$ its DFT, then 
  \begin{equation*}
    1 \,\le\, \frac{\|x\|_{\l0}\|\wh{x}\|_{\l0}}{N} \,.
  \end{equation*}
\end{theorem}

This is a lower bound on the \textit{time-bandwidth product}. 
In \cite{DonSta89} Donoho and Stark present two important corollaries
of uncertainty principles for the Fourier transform. 
One is the uniqueness of sparse representations of a signal $x$ as
a superposition of vectors taken from both the standard basis and
the basis of Fourier modes, and the second is the recovery of this
representation by \l1 minimization. 

The main observation we make here is that, if we phrase
our uncertainty principle as in theorem~\ref{th:Uncertainty2}, then
the far field translation operator, as well as the map from $\alpha$
to its Fourier coefficients, satisfy an uncertainty principle. 
Combining the uncertainty principle with the regularized Picard
criterion from section~\ref{sec:RegPicard} yields analogs of both
results in the context of the 
inverse source problem. 
These include previous results about the splitting of far fields
from \cite{GriHanSyl14} and \cite{GriSyl16}, which can be simplified
and extended by viewing them as consequences of the uncertainty
principle and the regularized Picard criterion.

The proof of theorem~\ref{th:Uncertainty2} is a simple corollary of
the lemma below:

\begin{lemma}
  \label{lmm:TcEstimate}
  Let $c\in\R^2$ and let $T_c$ be the operator introduced in
  \eqref{eq:DefTc} and \eqref{eq:DefTcAbuse}. 
  Then, the operator norm of 
  $\Tc:\, L^{p}(S^{1}) \longrightarrow L^{p}(S^{1})$, 
  $1\leq p\leq \infty$, satisfies
  \begin{equation}
    \label{eq:EstTcLL}
    \|\Tc\|_{\L{p},\L{p}} \,=\, 1 \,,
  \end{equation}
  whereas $\Tc: \l1 \longrightarrow \l\infty$ fulfills
  \begin{equation}
    \label{eq:EstTcll}
    \|\Tc\|_{\l1,\l\infty} \,\le\, \frac{1}{|c|^{\third}} \,.
  \end{equation}
\end{lemma}

\begin{proof}
  Recalling \eqref{eq:DefTc}, we see that $\Tc$ is multiplication by a
  function of modulus one, so \eqref{eq:EstTcLL} is immediate.
  On the other hand, combining \eqref{eq:TcConvolution} with the last
  inequality from page 199 of \cite{Lan00}; more precisely, 
  \begin{equation*}
    |J_n(x)| \,<\, \frac{b}{|x|^{\third}} \qquad 
    \text{with } b \approx 0.6749 \,,
  \end{equation*}
  shows that 
  \begin{equation*}
    \|\Tc\|_{\l1,\l\infty} 
    \,\le\, \sup_{n\in\Z}|J_{n}(|c|)| 
    \,\le\, \frac{1}{|c|^{\third}} \,.
  \end{equation*}
\end{proof}

\begin{proof}[Proof of theorem~\ref{th:Uncertainty2}]
  Using H\"older's inequality and \eqref{eq:EstTcll} we obtain that 
  \begin{equation*}
    |\langle\alpha,\Tc\beta\rangle|
    \,\le\, \|\alpha\|_{\l1} \|\Tc\beta\|_{\l\infty}
    \,\le\, \frac{1}{|c|^{\third}}\|\alpha\|_{\l1}\;\|\beta\|_{\l1}
    \,\le\,
    \frac{\sqrt{\|\alpha\|_{\l0}\|\beta\|_{\l0}}}{|c|^{\third}}
    \|\alpha\|_{\l2} \|\beta\|_{\l2} \,.
  \end{equation*}
\end{proof}

We can improve the dependence on  $|c|$ in \eqref{eq:Uncertainty2}
under  hypotheses on  $\alpha$ and $\beta$ that are more restrictive,
but well suited to the inverse source problem.

\begin{theorem}
  \label{th:Uncertainty3}
  Suppose that  $\alpha\in\l2(-M,M)$, $\beta\in\l2(-N,N)$ with
  $M,N\geq1$, and let $c\in\R^2$ such that $|c|>2(M+N+1)$. 
  Then
  \begin{equation}
    \label{eq:Uncertainty3}
    |\langle\alpha,\Tc\beta\rangle| \,\le\,
    \frac{\sqrt{(2N+1)(2M+1)}}{|c|^{\half}}
    \|\alpha\|_2 \|\beta\|_2 \,.
  \end{equation}
\end{theorem}

\begin{proof}
  Because the \l0-support of  $\beta$ is contained in  $[-N,N]$ 
  \begin{equation*}
    \beta^{c}_{m} \,=\,
    \sum_{n=-N}^{N} \beta_{n} \Bigl( \rmi^{m-n} J_{m-n}(|c|) 
    e^{\rmi(m-n)\varphi_{c}}\Bigr)
  \end{equation*}
  so
  \begin{equation*}
    \sup_{-M<m<M}| \beta^{c}_{m}| 
    \,\le\, \|\beta\|_{\l1} \sup_{-(M+N)<n<(M+N)} |J_{n}(|c|)|
  \end{equation*}
  and it follows from theorem 2 of \cite{Kra06}, using the fact that
  $M,N\ge 1$, together with our hypothesis, which implies that
  $|c|>6$, that 
  \begin{equation}
    \label{eq:ProofUncertainty3-1}
    \sup_{-(M+N)<n<(M+N)} J_n^2(|c|)
    \,\le\, \frac{b}{|c|} \qquad \text{with $b \approx 0.7595$}
  \end{equation}
  (see appendix~\ref{app:ProofKrasikov} for details).
  We now simply repeat the proof of theorem~\ref{th:Uncertainty2},
  replacing  the estimate for $\|\Tc\beta\|_{\l\infty}$ from
  \eqref{eq:EstTcll} with the estimate we have just established in
  \eqref{eq:ProofUncertainty3-1}, i.e. 
  \begin{equation}
    \label{eq:ProofUncertainty3-2}
    \|\Tc\|_{\l1[-N,N],\l\infty[-M,M]} 
    \,\le\, \frac{1}{|c|^{\half}} \,.
  \end{equation}
\end{proof}

We will also make use of another uncertainty principle. 
A glance at \eqref{eq:DefFourierCoeffs}--\eqref{eq:Plancherel} reveals
that the operator which maps $\alpha$ to its Fourier coefficients maps
\L2 to \l2 with norm 1, \L1 to \l\infty \,with norm
$1/\sqrt{2\pi}$, and its inverse maps \l1 to \L\infty, also with norm
$1/\sqrt{2\pi}$. 
An immediate corollary of this observation is

\begin{theorem}
  \label{th:FourierUncertainty}
  Let $\alpha, \beta \in L^2(S^1)$ and let $c\in\R^2$.
  Then,
  \begin{equation}
    \label{eq:FourierUncertainty}
    |\langle T_c\alpha,\beta\rangle| 
    \,\le\, \sqrt{\frac{\|\alpha\|_{\l0} \|\beta\|_{\L0}}{2\pi}}
    \|\alpha\|_2 \|\beta\|_2 \,.
  \end{equation}
\end{theorem}

\begin{proof}
  Combining H\"older's inequality with \eqref{eq:EstTcLL} and using the
  mapping properties of the operator which maps $\alpha$ to its Fourier
  coefficients we find that 
  \begin{equation*}
    \begin{split}
      |\langle T_c\alpha,\beta\rangle|
      &\,\le\, \|T_c \alpha\|_{\L\infty} \|\beta\|_{\L1} 
      \,\le\, \|\alpha\|_{\L\infty} \|\beta\|_{\L1} 
      \,\le\, \frac{1}{\sqrt{2\pi}}\|\alpha\|_{\l1} \|\beta\|_{\L1}\\
      &\,\le\, \frac{1}{\sqrt{2\pi}} \sqrt{\|\alpha\|_{\l0}} \|\alpha\|_2 
      \sqrt{\|\beta\|_{\L0}} \|\beta\|_2 \,.
    \end{split}
  \end{equation*}
\end{proof}

\section{\l2 corollaries of the uncertainty principles}
\label{sec:L2Corollaries}
The regularized Picard criterion tells us that, up to an \L2-small
error, a far field radiated by a limited power source in $B_R(0)$ is
\L2-close to an $\alpha$ that belongs to the subspace of
non-evanescent far fields, the span of $\{e^{\rmi n\theta}\}$ with
$|n|\le N$, where $N = N(R,P,p)$ is a little bigger than the radius
$R$. 
This non-evanescent $\alpha$ satisfies $\|\alpha\|_{\l0}\le 2N+1$.
The uncertainty principle will show that the angle between translates
of these subspaces is bounded below when the  translation parameter is
large enough, so that we can split the sum of the two non-evanescent
far fields into the original two summands. 

\begin{lemma}
  \label{lm:L2Split}
  Suppose that $\gamma,\alpha_1,\alpha_2\in\L2(S^1)$ and
  $c_1,c_2\in\R^2$ with
  \begin{equation}
    \label{eq:Split2}
    \gamma 
    \,=\, T^*_{c_1}\alpha_1 + T^*_{c_2}\alpha_2
  \end{equation}
  and that
  $\frac{\|\alpha_1\|_{\l0}\|\alpha_2\|_{\l0}}{|c_1-c_2|^{\frac{2}{3}}}<1$.
  Then, for  $i=1,2$
  \begin{equation}
    \label{eq:L2Split}
    \|\alpha_{i}\|_2^{2}
    \,\le\, \biggl( 1 - \frac{\|\alpha_1\|_{\l0}\|\alpha_2\|_{\l0}}
    {|c_1-c_2|^{\frac{2}{3}}} \biggr)^{-1} \|\gamma\|_2^{2} \,.
  \end{equation}
\end{lemma}

\begin{proof}
  We first note that \eqref{eq:Split2} and \eqref{eq:DefTc} imply
  \begin{equation}
    \label{eq:proofSplit2-1}
    \begin{split}
      \|\gamma\|_2^{2} 
      &\,\geq\, \|\alpha_1\|_2^{2} + \|\alpha_2\|_2^{2} 
      - 2|\langle T^*_{c_1}\alpha_1,T^*_{c_2}\alpha_2\rangle|\\
      &\,=\, \|\alpha_1\|_2^{2} + \|\alpha_2\|_2^{2} 
      - 2|\langle \alpha_1,T^*_{c_2-c_1}\alpha_2\rangle| \,.
    \end{split}
  \end{equation}
  We now use \eqref{eq:Uncertainty2},
  \begin{equation}
    \label{eq:proofSplit2-2}
    \begin{split}
      \|\gamma\|_2^{2} 
      &\,\ge\, \|\alpha_1\|_2^{2} + \|\alpha_2\|_2^{2}
      -2\frac{\sqrt{\|\alpha_1\|_{\l0}\|\alpha_2\|_{\l0}}}{|c_2-c_1|^{\frac{1}{3}}}
      \|\alpha_1\|_2 \|\alpha_2\|_2\\
      &\,=\, \biggl(1-\frac{\|\alpha_1\|_{\l0}\|\alpha_2\|_{\l0}}
      {|c_2-c_1|^{\frac{2}{3}}}\|\biggr)\|\alpha_1\|_2^{2}
      +\biggl(\|\alpha_2\|_2-\frac{\sqrt{\|\alpha_1\|_{\l0}\|\alpha_2\|_{\l0}}}
      {|c_2-c_1|^{\frac{1}{3}}}\|\alpha_1\|_2\biggr)^{2} \,.
    \end{split}
  \end{equation}
  Dropping the second term now gives \eqref{eq:L2Split} for
  $\alpha_1$, and we may interchange the roles $\alpha_1$ and
  $\alpha_2$ in the proof to obtain the estimate for $\alpha_2$ .
\end{proof}

The analogous consequence of theorem~\ref{th:Uncertainty3} is
\begin{lemma}
  \label{lm:L2SplitWS}
  Suppose that $\gamma \in \L2(S^1)$,
  $\alpha_{i}\in\l2(-N_{i},N_{i})$ for some $N_i\in\N$, $i=1,2$, and 
  $c_1,c_2\in\R^2$ with $|c_1-c_2|>2(N_1+N_2+1)$ and
  \begin{equation*}
    \gamma 
    \,=\, T^*_{c_1}\alpha_1 + T^*_{c_2}\alpha_2 \,,
  \end{equation*}
  and that $\frac{(2N_1+1)(2N_2+1)}{|c_1-c_2|}<1$.
  Then, for  $i=1,2$
  \begin{equation}
    \label{eq:L2SplitWS}
    \|\alpha_{i}\|_2^{2}
    \,\le\,
    \biggl(1-\frac{(2N_1+1)(2N_2+1)}{|c_1-c_2|}\biggr)^{-1}\|\gamma\|_2^{2} \,.
  \end{equation}
\end{lemma}

In our application to the inverse source problem, we will know that
each far field is the translation of a far field  $\alpha_{i}$,
radiated by a limited power source supported in a ball centered at the
origin, and therefore that all but a very small amount of the radiated
power is contained in the non-evanescent part, the translation of the
Fourier modes $e^{\rmi n\theta}$ for $|n|<N(R,p,P)$. 
The estimate in the theorem below says that, if the distances between
the balls is large enough, we may uniquely solve for the
non-evanescent parts of the individual far fields, and that this split
is stable with respect to perturbations in the data. 

\begin{theorem}
  \label{th:L2SplitLS}
  Suppose that $\gamma^0,\gamma^1\in\L2(S^1)$, $c_1,c_2\in\R^2$ and
  $N_1,N_2\in\N$ such that $|c_1-c_2|>2(N_1+N_2+1)$ and
  \begin{equation}
    \label{eq:L2SplitLSPositivity}
    \frac{(2N_1+1)(2N_2+1)}{|c_1-c_2|} \,<\, 1 \,,
  \end{equation}
  and let
  \begin{subequations}
    \label{eq:split2LS}
    \begin{align}
      \gamma^{0} 
      &\,\mathop{=}\limits^{\mathrm{LS}}\,
      T^*_{c_1}\alpha_1^{0}+T^*_{c_2}\alpha_2^{0} \,,
      &&\alpha_{i}^{0}\in\l2(-N_{i},N_{i}) \,,\\
      \gamma^{1}
      &\,\mathop{=}\limits^{\mathrm{LS}}\,
      T^*_{c_1}\alpha_1^{1}+T^*_{c_2}\alpha_2^{1} \,,
      &&\alpha_{i}^{1}\in\l2(-N_{i},N_{i}) \,.
    \end{align}
  \end{subequations}
  Then, for  $i=1,2$ 
  \begin{equation}
    \label{eq:L2SplitLS}
    \|\alpha^{1}_{i}-\alpha_{i}^{0}\|_2^{2} 
    \,\le\,
    \biggl(1-\frac{(2N_1+1)(2N_2+1)}{|c_1-c_2|}\biggr)^{-1} 
    \|\gamma^{1}-\gamma^{0}\|_2^{2} \,.
  \end{equation}
\end{theorem}

The notation in \eqref{eq:split2LS} above means that the
$\alpha_{i}^{j}$ are the (necessarily unique) least squares
solutions to the equations
$\gamma^{j}=T^*_{c_1}\alpha_1^{j}+T^*_{c_2}\alpha_2^{j}$.
Recall that the far fields radiated by a limited power source from a
ball have almost all, but not all, of their power (\L2-norm)
concentrated in the Fourier modes with $n\le N(R,P,p)$. 
Therefore the $\gamma^{i}$ will typically not belong to the subspace
that is the direct sum of 
$T^*_{c_1}\l2(-N_1,N_1) \oplus T^*_{c_2}\l2(-N_2,N_2)$, and
therefore $\alpha_1^{j}$ and $\alpha_2^{j}$ will usually not solve
equations \eqref{eq:split2LS} exactly. 
The estimate in \eqref{eq:L2SplitLS} is nevertheless always true, and 
guarantees that the pair $(\alpha_1^{j},\alpha_2^{j})$ is unique
and that the absolute condition number of the splitting operator which maps
$\gamma$ to $(\alpha_1^{j},\alpha_2^{j})$ is no larger than
$\left(1-\frac{(2N_1+1)(2N_2+1)}{|c_1-c_2|}\right)^{-\half}$.

\begin{proof}[Proof of theorem~\ref{th:L2SplitLS}]
  Each  $\gamma^{j}$ can be uniquely decomposed as
  \begin{equation}
    \label{eq:proofL2SplitLS-1}
    \gamma^{j} \,=\, w^{j}+w^{j}_{\perp} \,,
  \end{equation}
  where each $w^{j}$ belongs to the  $2N_1+2N_2+2$-dimensional
  subspace 
  \begin{equation*}
    W 
    \,=\, T^*_{c_1}\l2(-N_1,N_1) \oplus T^*_{c_2}\l2(-N_2,N_2) 
  \end{equation*}
  and each $w^{j}_{\perp}$ is orthogonal to  $W$. 
  The definition of least squares solutions means that
  \begin{equation*}
    w^{j}
    \,=\, T^*_{c_1}\alpha^{j}_1 + T^*_{c_2}\alpha^{j}_2 \,.
  \end{equation*}
  Subtracting gives
  \begin{equation}
    \label{eq:proofL2SplitLS-2}
    w^{1}-w^{0}
    \,=\, T^*_{c_1}(\alpha^{1}_1-\alpha^{0}_1) 
    + T^*_{c_2}(\alpha^{1}_2 -\alpha^{0}_2)
  \end{equation}
  and applying the estimate \eqref{eq:L2SplitWS} yields
  \begin{equation}
    \label{eq:proofL2SplitLS-3}
    \|\alpha^{1}_{i}-\alpha^{0}_{i}\|_2^{2}
    \,\le\, \biggl(1-\frac{(2N_1+1)(2N_2+1)}{|c_1-c_2|}\biggr)^{-1} 
    \|w^{1}-w^{0}\|_2^{2} \,.
  \end{equation}
  Finally, we note that
  \begin{equation}
    \label{eq:proofL2SplitLS-4}
    \|\gamma^{1}-\gamma^{0}\|_2^{2} 
    \,=\, \|w^{1}-w^{0}\|_2^{2}+\|w^{1}_{\perp}-w^{0}_{\perp}\|_2^{2}
    \,\ge\, \|w^{1}-w^{0}\|_2^{2} \,,
  \end{equation}
  which finishes the proof.
\end{proof}

We also have corresponding corollaries of theorem
\ref{th:FourierUncertainty}, which tell us that, if a far field is
radiated from a small ball, and measured on most of the circle, then
it is possible to recover its non-evanescent part on the entire
circle. 
Theorem~\ref{th:L2CompleteLS} below, describes the case where we
cannot measure the far field $\alpha=\Tc^*\alpha^{0}$ on a subset
$\Omega\tm S^{1}$. 
We measure $\gamma=\alpha+\beta$, where
$\beta=-\alpha\big|_{\Omega}$. 
The estimates \eqref{eq:L2CompleteLS} imply that we can stably recover
the non-evanescent part of the far field on $\Omega$. 

Before we state the theorem, we give the corresponding analogue of
lemma~\ref{lm:L2Split} and lemma~\ref{lm:L2SplitWS}.
\begin{lemma}
  \label{lm:L2Complete}
  Suppose that $\gamma,\alpha,\beta\in\L2(S^1)$ and $c\in\R^2$ with
  \begin{equation*}
    \gamma 
    \,=\, \beta + T^*_{c}\alpha 
  \end{equation*}
  and that
  $\frac{\|\alpha\|_{\l0}\|\beta\|_{\L0}}{2\pi}<1$.
  Then
  \begin{subequations}
    \label{eq:L2Complete}
    \begin{align}
      \|\alpha\|_2^{2}
      &\,\le\, \biggl( 1 - \frac{\|\alpha\|_{\l0}\|\beta\|_{\L0}}{2\pi} 
      \biggr)^{-1} \|\gamma\|_2^{2} \\\noalign{and}
      \|\beta\|_2^{2}
      &\,\le\, \biggl( 1 - \frac{\|\alpha\|_{\l0}\|\beta\|_{\L0}}{2\pi} 
      \biggr)^{-1} \|\gamma\|_2^{2} \,.
    \end{align}
  \end{subequations}
\end{lemma}

\begin{proof}
  Proceeding as in \eqref{eq:proofSplit2-1}--\eqref{eq:proofSplit2-2},
  but replacing \eqref{eq:Uncertainty2} by
  \eqref{eq:FourierUncertainty} yields the result. 
\end{proof}

\begin{theorem}
  \label{th:L2CompleteLS}
  Suppose that $\gamma^0,\gamma^1\in\L2(S^1)$, $c\in\R^2$, $N\in\N$
  and $\Omega\tm S^1$ such that $\frac{(2N+1)|\Omega|}{2\pi}<1$, and
  let
  \begin{align*}
    \gamma^{0} 
    &\,\mathop{=}\limits^{\mathrm{LS}}\,
    \beta^{0} + \Tc\alpha^{0} \,,
    &&\alpha^{0}\in\l2(-N,N) \text{ and }
    \beta^{0}\in\L2(\Omega)\,,\\
    \gamma^{1}
    &\,\mathop{=}\limits^{\mathrm{LS}}\, 
    \beta^{1} + \Tc\alpha^{1} \,,
    &&\alpha^{1}\in\l2(-N,N) \text{ and } 
    \beta^{1}\in\L2(\Omega) \,.
  \end{align*}
  Then
  \begin{subequations}
    \label{eq:L2CompleteLS}
    \begin{align}
      \|\alpha^{1}-\alpha^{0}\|_2^{2}
      &\,\le\, \biggl(1-\frac{(2N+1)|\Omega|}{2\pi}\biggr)^{-1}
      \|\gamma^{1}-\gamma^{0}\|_2^{2}\\\noalign{and}
      \|\beta^{1}-\beta^{0}\|_2^{2}
      &\,\le\, \biggl(1-\frac{(2N+1)|\Omega|}{2\pi}\biggr)^{-1}
      \|\gamma^{1}-\gamma^{0}\|_2^{2} \,.
    \end{align}
  \end{subequations}
\end{theorem}

\begin{proof}
  Just as in \eqref{eq:proofL2SplitLS-1}, we decompose each $\gamma^{j}$ 
  \begin{equation*}
    \gamma^{j} \,=\, w^{j} + w^{j}_{\perp} \,,
  \end{equation*}
  where each $w^{j}$ belongs to the subspace
  \begin{equation*}
    W \,=\, \L2(\Omega) \oplus \Tc\l2(-N,N)
  \end{equation*}
  and each $w^{j}_{\perp}$ is orthogonal to $W$. 
  Proceeding as in
  \eqref{eq:proofL2SplitLS-2}--\eqref{eq:proofL2SplitLS-3}, but using
  the estimates from \eqref{eq:L2Complete}, we find 
  \begin{align*}
    \|\alpha^{1}-\alpha^{0}\|_2^{2}
    \,\le\, \biggl(1-\frac{(2N+1)|\Omega|}{2\pi}\biggr)^{-1} 
    \|w^{1}-w^{0}\|_2^{2}\\\noalign{and}
    \|\beta^{1}-\beta^{0}\|_2^{2}
    \,\le\, \biggl(1-\frac{(2N+1)|\Omega|}{2\pi}\biggr)^{-1} 
    \|w^{1}-w^{0}\|_2^{2}
  \end{align*}
  and then note that \eqref{eq:proofL2SplitLS-4} is true here as well
  to finish the proof.
\end{proof}

A version of theorem~\ref{th:L2SplitLS} with multiple well-separated
components is also true. 

\begin{theorem}
  \label{th:L2SplitMultipleLS}
  Suppose that $\gamma^0,\gamma^1\in\L2(S^1)$, $c_i\in\R^2$ and
  $N_i\in\N$, $i=1,\ldots,I$, such that
  $|c_{i}-c_{j}|>2(N_{i}+N_{j}+1)$ for every $i\ne j$ and 
  \begin{equation*}
    \biggl(\sqrt{2N_{i}+1}\sum_{j\not=i}
    \sqrt{\frac{2N_{j}+1}{|c_{i}-c_{j}|}}\biggr) 
    \,<\,1 \qquad \text{for each } i \,,
  \end{equation*}
  and let
  \begin{align*}
    \gamma^{0} &\,\mathop{=}\limits^{LS}\,
    \sum_{i=1}^{I} T^*_{c_{i}}\alpha_{i}^{0} \,,
    &&\alpha_{i}^{0}\in\l2(-N_{i},N_{i}) \,,\\
    \gamma^{1} &\,\mathop{=}\limits^{LS}\,
    \sum_{i=1}^{I}T^*_{c_{i}}\alpha_{i}^{1} \,,
    &&\alpha_{i}^{1}\in\l2(-N_{i},N_{i}) \,.
  \end{align*}
  Then, for $i=1,\ldots,I$
  \begin{equation*}
    \|\alpha_{i}^{1}-\alpha_{i}^{0}\|_2^{2}
    \,\le\, \biggl( 1-\sqrt{2N_{i}+1} 
    \sum_{j\not=i} \sqrt{\frac{2N_{j}+1}{|c_{j}-c_{i}|}} \biggr)^{-1}
    \|\gamma^{1}-\gamma^{0}\|_2^{2} \,.
  \end{equation*}
\end{theorem}

\begin{proof}
  As in \eqref{eq:proofL2SplitLS-1}, we decompose each $\gamma^{j}$ 
  \begin{equation*}
    \gamma^{j} \,=\, w^{j} + w^{j}_{\perp} \,,
  \end{equation*}
  where each $w^{j}$ belongs to the subspace
  \begin{equation*}
    W \,=\, \bigoplus_{i=1}^I T^*_{c_i} \l2(-N_i,N_i)
  \end{equation*}
  and each $w^{j}_{\perp}$ is orthogonal to $W$. 
  Subtracting gives 
  \begin{equation*}
    w^1-w^0 
    \,=\, \sum_{i=1}^I T^*_{c_i} (\alpha_i^1-\alpha_i^0) 
  \end{equation*}
  and applying \eqref{eq:Uncertainty3} shows that
  \begin{equation}
    \label{eq:proofL2SplitMultipleLS}
    \begin{split}
      \|w^1-w^0\|_2^2
      &\,\geq\, \sum_{i=1}^I \|\alpha_i^1-\alpha_i^0\|_2^2 
      - \sum_{i=1}^I\sum_{j\not=i}
      |\langle\alpha_i^1-\alpha_i^0,T^*_{c_j-c_i}(\alpha_j^1-\alpha_j^0)\rangle|\\
      &\,\geq\, \sum_{i=1}^I \|\alpha_i^1-\alpha_i^0\|_2^2 
      - \sum_{i=1}^I\sum_{j\not=i}
      \biggl(\Bigl(\frac{r_i r_j}{|c_{ij}|}\Bigr)^{\frac14}\|\alpha_i^1-\alpha_i^0\|_2
      \Bigl(\frac{r_i r_j}{|c_{ij}|}\Bigr)^{\frac14}\|\alpha_j^1-\alpha_j^0\|_2
      \biggr)\\
      &\,\geq\, \sum_{i=1}^I \|\alpha_i^1-\alpha_i^0\|_2^2 
      - \frac12 \biggl( \sum_{i=1}^I \sum_{j\not=i} 
      \biggl( \Bigl(\frac{r_ir_j}{|c_{ij}|}\Bigr)^{\frac12}\|\alpha_i^1-\alpha_i^0\|_2^2
      + \Bigl(\frac{r_ir_j}{|c_{ij}|}\Bigr)^{\frac12}\|\alpha_j^1-\alpha_j^0\|_2^2 \biggr)\\
      &\,=\, \sum_{i=1}^I \|\alpha_i^1-\alpha_i^0\|_2^2 
      \biggl( 1 - 
      \sum_{j\not=i}\Bigl(\frac{r_i r_j}{|c_{ij}|}\Bigr)^{\frac12}
      \biggr) \,,
    \end{split}
  \end{equation}
  where $r_i = 2N_i+1$, $r_j = 2N_j+1$ and $c_{ij} = c_i-c_j$. 
  Since \eqref{eq:proofL2SplitLS-4} is true here as well this ends the proof.
\end{proof}

We may include a missing data component as well.

\begin{theorem}
  \label{th:L2CompleteAndSplitMultipleLS}
  Suppose that $\gamma^0,\gamma^1\in\L2(S^1)$, $c_i\in\R^2$,
  $N_i\in\N$, $i=1,\ldots,I$, and $\Omega\tm\L2(S^1)$ such that 
  $|c_{i}-c_{j}|>2(N_{i}+N_{j}+1)$ for every  $i\ne j$ and
  \begin{align*}
    \sqrt{\frac{|\Omega|}{2\pi}} \sum_{i=1}^{I}\sqrt{2N_{i}+1} 
    &\,<\, 1 \,,\\
    \sqrt{2N_{i}+1} \biggl( \sqrt{\frac{|\Omega|}{2\pi}} 
    + \sum_{j\ne i}\sqrt{\frac{2N_{i}+1}{|c_i-c_j|}} \biggr) 
    &\,<\, 1 \qquad \text{for each } i \,,
  \end{align*}
  and let
  \begin{subequations}
    \label{eq:L2CompleteAndSplitMultipleLS-LSP}
    \begin{align}
      \gamma^{0} &\,\mathop{=}\limits^{LS}\, \beta^{0} +
      \sum_{i=1}^{I}T^*_{c_{i}}\alpha_{i}^{0} \,,
      &\alpha_{i}^{0}\in\l2(-N_{i},N_{i}) \text{ and }
      \beta^{0}\in\L2(\Omega) \,,\\
      \gamma^{1} &\,\mathop{=}\limits^{LS}\, \beta^{1} +
      \sum_{i=1}^{I}T^*_{c_{i}}\alpha_{i}^{1} \,,
      &\alpha_{i}^{1}\in\l2(-N_{i},N_{i}) \text{ and }
      \beta^{0}\in\L2(\Omega) \,.
    \end{align}
  \end{subequations}
  Then
  \begin{align*}
    \|\beta^{1}-\beta^{0}\|_2^{2} &\,\le\,
    \biggl(1-\sqrt{\frac{|\Omega|}{2\pi}}
    \sum_{i}\sqrt{2N_{i}+1}\biggr)^{-1}
    \|\gamma^{1}-\gamma^{0}\|_2^2\\\noalign{and, for $i=1,\ldots,I$}
    \|\alpha_{i}^{1}-\alpha_{i}^{0}\|_2^{2} &\,\le\,
    \biggl(1-\sqrt{2N_{i}+1} \biggl(\sqrt{\frac{|\Omega|}{2\pi}} +
    \sum_{j\ne i}\sqrt{\frac{2N_{i}+1}{|c_i-c_j|}} \biggr) \biggr)^{-1}
    \|\gamma^{1}-\gamma^{0}\|_2^2 \,.
  \end{align*}
\end{theorem}

\begin{proof}
  As in \eqref{eq:proofL2SplitLS-1}, we decompose each $\gamma^{j}$ 
  \begin{equation*}
    \gamma^{j} \,=\, w^{j} + w^{j}_{\perp} \,,
  \end{equation*}
  where each $w^{j}$ belongs to the subspace
  \begin{equation*}
    W \,=\, L^2(\Omega) \oplus \bigoplus_{i=1}^J T^*_{c_i} \l2(-N_i,N_i)
  \end{equation*}
  and each $w^{j}_{\perp}$ is orthogonal to $W$. 
  Subtracting gives 
  \begin{equation*}
    w^1-w^0 
    \,=\, \beta^1-\beta^0 + \sum_{i=1}^I T^*_{c_i} (\alpha_i^1-\alpha_i^0) 
  \end{equation*}
  and thus
  \begin{equation*}
    \begin{split}
      \|w^1-w^0\|_2^2
      &\,\geq\, \|\beta^1-\beta^0\|_2^2
      -2\sum_{i=1}^I|\langle T^*_{c_i}(\alpha_i^1-\alpha_i^0),\beta^1-\beta^0\rangle|\\
      &\phantom{\,\geq\,}
      +\sum_{i=1}^I\|\alpha_i^1-\alpha_i^0\|_2^2 
      -\sum_{i=1}^I\sum_{j\not= i}
      |\langle\alpha_i^1-\alpha_i^0,T^*_{c_j-c_i}(\alpha_j^1-\alpha_j^0)\rangle| \,.
    \end{split}
  \end{equation*}
  Proceeding as in \eqref{eq:proofL2SplitMultipleLS}, using
  \eqref{eq:Uncertainty3} and \eqref{eq:FourierUncertainty}, and applying
  \eqref{eq:proofL2SplitLS-4} finishes the proof.
\end{proof}

\section{\l1 corollaries of the uncertainty principle}
\label{sec:L1Corollaries}
The results below are analogous to those in the previous section. 
The main difference is that they do not require the \emph{a priori}
knowledge of the size of the non-evanescent subspaces (the  $N_{i}$ in
theorems~\ref{th:L2SplitLS} through
\ref{th:L2CompleteAndSplitMultipleLS}).

In theorem \ref{th:L1Split} below, $\gamma^{0}$
represents the (measured) approximate far field; the
$\alpha_{i}^{0}$ are the non-evanescent parts of the
true (unknown) far fields radiated by each of the two
components, which we assume are well-separated
\eqref{eq:L1SplitPositivity}.  The constant $\delta_{0}$
in \eqref{eq:L1SplitNoisy} accounts for both the noise and
the evanescent components of the true far
fields. Condition \eqref{eq:L1SplitBoundDelta} requires
that the optimization problem \eqref{eq:L1SplitMin} be
formulated with a constraint that is weak enough so that the
$\alpha_{i}^{0}$ are feasible.

\begin{theorem}
  \label{th:L1Split}
  Suppose that $\gamma^0,\alpha_1^0,\alpha_2^0\in\L2(S^1)$ and
  $c_1,c_2\in\R^2$ such that
  \begin{equation}
    \label{eq:L1SplitPositivity}
    \frac{4\|\alpha_{i}^{0}\|_{\l0}}{|c_1-c_2|^{\third}} \,<\, 1
    \qquad \text{for each } i
  \end{equation}
  and
  \begin{equation}
    \label{eq:L1SplitNoisy}
    \|\gamma^{0}-T^*_{c_1}\alpha_1^{0}-T^*_{c_2}\alpha_2^{0}\|_2
    \,\le\, \delta_0 \qquad \text{for some } \delta_0\geq0 \,.
  \end{equation}
  If $\delta\geq0$ and $\gamma\in\L2(S^1)$ with 
  \begin{equation}
    \label{eq:L1SplitBoundDelta}
    \delta \,\ge\, \delta_0+\|\gamma-\gamma^{0}\|_2
  \end{equation}
  and
  \begin{equation}
    \label{eq:L1SplitMin}
    (\alpha_1,\alpha_2) 
    =\Argmin \|\alpha_1\|_{\l1}+\|\alpha_2\|_{\l1}
    \quad \text{s.t.} \quad 
    {\|\gamma-T^*_{c_1}\alpha_1-T^*_{c_2}\alpha_2\|_2\le\delta}\,,\;
    \alpha_1,\alpha_2 \in \L2(S^1) \,,
  \end{equation}
  then, for $i=1,2$
  \begin{equation}
    \label{eq:L1Split}
    \|\alpha^{0}_{i}-\alpha_{i}\|_2^{2} 
    \,\le\,
    \Bigl(1-\frac{4\|\alpha_{i}^{0}\|_{\l0}}{|c_1-c_2|^{\third}}\Bigr)^{-1} 
    4\delta^{2} \,.
  \end{equation}
\end{theorem}

\begin{proof}
  A consequence of \eqref{eq:L1SplitBoundDelta} is that the pair 
  $(\alpha_1^{0},\alpha_2^{0})$ satisfies the constraint in
  \eqref{eq:L1SplitMin}, which implies that 
  \begin{equation}
    \label{eq:ProofL1Split-1}
    \|\alpha_1\|_{\l1} +  \|\alpha_2\|_{\l1}
    \,\le\, \|\alpha_1^{0}\|_{\l1} +  \|\alpha_2^{0}\|_{\l1} 
  \end{equation}
  because $(\alpha_1,\alpha_2)$ is a minimizer.  
  Additionally, with $W_{i}$ representing the $\l0$-support of
  $\alpha^0_{i}$ and $W_{i}^{c}$ its complement,
  \begin{equation}
    \label{eq:ProofL1Split-2}
    \begin{split}
      \|\alpha_i\|_{\l1} 
      &\,=\,\|\alpha_i^0+(\alpha_i-\alpha_i^0)\|_{\l1}\\
      &\,=\,\|\alpha_i^0+(\alpha_i-\alpha_i^0)\|_{\l1(W_i)}
      +\|\alpha_i-\alpha_i^0\|_{\l1(W_i^{c})}\\
      &\,=\,\|\alpha_i^0+(\alpha_i-\alpha_i^0)\|_{\l1(W_i)}
      +\|\alpha_i-\alpha_i^0\|_{\l1} 
      -\|\alpha_i-\alpha_i^0\|_{\l1(W_i)}\\
      &\,\ge\,\|\alpha_i^0\|_{\l1} 
      +\|\alpha_i-\alpha_i^0\|_{\l1}
      -2\|\alpha_i-\alpha_i^0\|_{\l1(W_i)} \,.
    \end{split}
  \end{equation}
  Inserting \eqref{eq:ProofL1Split-2} into \eqref{eq:ProofL1Split-1} yields
  \begin{equation}
    \label{eq:ProofL1Split-3}
    \|\alpha_1-\alpha_1^{0}\|_{\l1}
    +\|\alpha_2-\alpha_2^{0}\|_{\l1}
    \,\le\,2(\|\alpha_1-\alpha_1^{0}\|_{\l1(W_1)}
    +\|\alpha_2-\alpha_2^{0}\|_{\l1(W_2)}) \,.
  \end{equation}
  We now use \eqref{eq:L1SplitBoundDelta} together with
  \eqref{eq:L1SplitNoisy}, the constraint in \eqref{eq:L1SplitMin} and
  the fact that $T^*_{c_1-c_2}$ is an \L2-isometry to obtain
  \begin{equation}
    \label{eq:ProofL1Split-4}
    \begin{split}
      4\delta^{2}
      &\,\ge\,\bigl(\|\gamma-\gamma^{0}\|_2+\delta_0+\delta\bigl)^2\\
      &\,\ge\,\bigl(\|\gamma-\gamma^0\|_2
      +\|\gamma^0-T^*_{c_1}\alpha_1^0-T^*_{c_2}\alpha_2^0\|_2
      +\|\gamma-T^*_{c_1}\alpha_1-T^*_{c_2}\alpha_2\|_2\bigl)^2\\
      &\,\ge\,\|T^*_{c_1}(\alpha_1-\alpha_1^0)+T^*_{c_2}(\alpha_2-\alpha_2^0)\|_2^2\\
      &\,=\,\|\alpha_1-\alpha_1^0+T^*_{c_2-c_1}(\alpha_2-\alpha_2^0)\|_2^2\\
      &\,\geq\,\|\alpha_1-\alpha_1^0\|_2^2+\|\alpha_2-\alpha_2^0\|_2^2
      -2|\langle\alpha_1-\alpha_1^0,T^*_{c_2-c_1}(\alpha_2-\alpha_2^0)\rangle| \,.
    \end{split}
  \end{equation}
  H\"older's inequality, \eqref{eq:EstTcll}, and
  \eqref{eq:ProofL1Split-3} show
  \begin{equation}
    \label{eq:ProofL1Split-5}
    \begin{split}
      4\delta^2
      &\,\ge\,\|\alpha_1-\alpha_1^0\|_2^2+\|\alpha_2-\alpha_2^0\|_2^2
      -\frac{2}{|c_1-c_2|^{\third}}\|\alpha_1-\alpha_1^0\|_{\l1}
      \|\alpha_2-\alpha_2^0\|_{\l1}\\
      &\,\ge\,\|\alpha_1-\alpha_1^0\|_2^2+\|\alpha_2-\alpha_2^0\|_2^2
      -\frac{1}{2|c_1-c_2|^{\third}}\bigl(\|\alpha_1-\alpha_1^0\|_{\l1}
      +\|\alpha_2-\alpha_2^0\|_{\l1}\bigr)^2\\
      &\,\ge\,\|\alpha_1-\alpha_1^0\|_2^2+\|\alpha_2-\alpha_2^0\|_2^2
      -\frac{2}{|c_1-c_2|^{\third}}\bigl(\|\alpha_1-\alpha_1^0\|_{\l1(W_1)}
      +\|\alpha_2-\alpha_2^0\|_{\l1(W_2)}\bigr)^2 \,.
    \end{split}
  \end{equation}
  Using H\"older's inequality once more yields
  \begin{equation}
    \label{eq:ProofL1Split-6}
    \begin{split}
      4\delta^2
      &\,\ge\,\|\alpha_1-\alpha_1^0\|_2^2+\|\alpha_2-\alpha_2^0\|_2^2
      -\frac{2}{|c_1-c_2|^{\third}}\bigl(|W_1|^{\half}\|\alpha_1-\alpha_1^0\|_2
      +|W_2|^{\half}\|\alpha_2-\alpha_2^0\|_2\bigr)^2\\
      &\,\ge\,\|\alpha_1-\alpha_1^0\|_2^2+\|\alpha_2-\alpha_2^0\|_2^2
      -\frac{4}{|c_1-c_2|^{\third}}\bigl(|W_1|\|\alpha_1-\alpha_1^0\|_2^2
      +|W_2|\|\alpha_2-\alpha_2^0\|_2^2\bigr) \,,
    \end{split}
  \end{equation}
  which implies \eqref{eq:L1Split} because  $|W_{i}|=\|\alpha_{i}^{0}\|_{\l0}$. 
\end{proof}

Assuming that some a priori information on the size of the
non-evanescent subspaces is available and that the distances between
the source components is large relative to their dimensions, we can
improve the dependence of the stability estimates on the distances.

\begin{corollary}
  \label{th:L1SplitWS}
  If we add to the hypothesis of theorem~\ref{th:L1Split}:
  \begin{equation*}
    \alpha_i^0,\alpha_{i}\in\l2(-N_{i},N_{i}) \qquad \text{and} \qquad
    |c_1-c_2|\,>\,2(N_1+N_2+1)
  \end{equation*}
  for some $N_1,N_2\in\N$ and replace \eqref{eq:L1SplitPositivity}
  with 
  \begin{equation}
    \label{eq:L1SplitPositivityWS}
    \frac{4\|\alpha_{i}^{0}\|_{\l0}}{|c_1-c_2|^{\half}} \,<\, 1
    \qquad \text{for each } i
  \end{equation}
  then, for $i=1,2$
  \begin{equation}
    \label{eq:L1SplitWS}
    \|\alpha^{0}_{i}-\alpha_{i}\|_2^{2} 
    \,\le\,
    \Bigl(1-\frac{4\|\alpha_{i}^{0}\|_{\l0}}{|c_1-c_2|^{\half}}\Bigr)^{-1} 
    4\delta^{2} \,.
  \end{equation}
\end{corollary}

\begin{proof}
  Replace \eqref{eq:EstTcll} by \eqref{eq:ProofUncertainty3-2} in
  \eqref{eq:ProofL1Split-4}--\eqref{eq:ProofL1Split-5}. 
\end{proof}

The constraint \eqref{eq:L1SplitPositivityWS} in
corollary~\ref{th:L1SplitWS} is much more restrictive than the
corresponding assumption \eqref{eq:L2SplitLSPositivity} in
theorem~\ref{th:L2SplitLS}, and also the estimate \eqref{eq:L1SplitWS}
is weaker than \eqref{eq:L2SplitLS}.
However, if we add to the hypothesis of theorem~\ref{th:L1Split} all 
\emph{a priori} assumptions on the non-evanescent subspaces used in
theorem~\ref{th:L2SplitLS}, the result improves as follows. 

\begin{corollary}
  \label{th:L1SplitWS-apriori}
  If we add to the hypothesis of theorem~\ref{th:L1Split}:
  \begin{equation*}
    \alpha_i^0,\alpha_{i}\in\l2(-N_{i},N_{i}) \qquad \text{and} \qquad
    |c_1-c_2|\,>\,2(N_1+N_2+1)
  \end{equation*}
  for some $N_1,N_2\in\N$ and replace \eqref{eq:L1SplitPositivity} with
  \begin{equation}
    \label{eq:L1SplitPositivityWS-apriori}
    \frac{(2N_1+1)(2N_2+1)}{|c_1-c_2|} \,<\, 1
  \end{equation}
  the conclusion becomes
  \begin{equation}
    \label{eq:L1SplitWS-apriori}
    \|\alpha^{0}_{i}-\alpha_{i}^{1}\|_2^{2} 
    \,\le\,\Bigl(1-\frac{(2N_1+1)(2N_2+1)}{|c_1-c_2|}\Bigr)^{-1} 
    4\delta^{2} \,.
  \end{equation}
\end{corollary}

The constraint \eqref{eq:L1SplitPositivityWS-apriori} in 
corollary~\ref{th:L1SplitWS-apriori} is now the same as the 
corresponding assumption \eqref{eq:L2SplitLSPositivity} in
theorem~\ref{th:L2SplitLS}, but the estimate
\eqref{eq:L1SplitWS-apriori} still differs from \eqref{eq:L2SplitLS}
(after taking the square root on both sides of these inequalities) by
a factor of two. 
However, the main advantage of the \l1 approach is clearly that
no a priori knowledge of the size of the non-evanescent subspaces is
required. 
If such a priori information is available, we recommend using least
squares. 

\begin{proof}[Proof of corollary~\ref{th:L1SplitWS-apriori}]
  Proceeding as in \eqref{eq:ProofL1Split-4} and applying
  \eqref{eq:Uncertainty3}, we find that 
  \begin{equation*}
    \begin{split}
      4\delta^{2}
      &\,\ge\,\|\alpha_1-\alpha_1^0\|_2^2+\|\alpha_2-\alpha_2^0\|_2^2
      -2|\langle\alpha_1-\alpha_1^0,T^*_{c_2-c_1}(\alpha_2-\alpha_2^0)\rangle|\\
      &\,\ge\,\|\alpha_1-\alpha_1^0\|_2^2+\|\alpha_2-\alpha_2^0\|_2^2
      -2\Bigl(\frac{r_1r_2}{|c_{12}|}\Bigr)^{\frac{1}{2}} 
      \|\alpha_1-\alpha_1^0\|_2 \|\alpha_2-\alpha_2^0\|_2\\
      &\,\ge\,\biggl( 1 - \frac{r_1r_2}{|c_{12}|} \biggr) 
      \|\alpha_1-\alpha_1^0\|_2^2 
      + \Biggl( \|\alpha_2-\alpha_2^0\|_2
      -\Bigl(\frac{r_1r_2}{|c_{12}|}\Bigr)^{\frac{1}{2}} 
      \|\alpha_1-\alpha_1^0\|_2 \Biggr)^2 \,,
    \end{split}
  \end{equation*}
  where $r_1 = 2N_1+1$, $r_2 = 2N_2+1$ and $c_{12} = c_1-c_2$.
  Dropping the second term gives \eqref{eq:L1SplitWS} for $\alpha_1$,
  and we may interchange the roles of $\alpha_1$ and $\alpha_2$ in the
  proof to obtain the estimate for $\alpha_2$.
\end{proof}

The analogue of theorem~\ref{th:L2CompleteLS} for data completion but
without \emph{a priori} knowledge on the size of the non-evanescent
subspaces is

\begin{theorem}
  \label{th:L1Complete}
  Suppose that $\gamma^0,\alpha^0\in\L2(S^1)$, $\Omega\tm S^1$, 
  $\beta^0\in\L2(\Omega)$ and $c\in\R^2$ such that 
  \begin{equation*}
    \frac{2\|\alpha^0\|_{\l0}|\Omega|}{\pi} \,<\, 1
  \end{equation*}
  and
  \begin{equation*}
    \|\gamma^{0}-T^*_{c}\alpha^{0}-\beta^{0}\|_2 
    \,\le\, \delta_0 \qquad \text{for some } \delta^0\geq 0 \,.
  \end{equation*}
  If $\delta\geq0$ and $\gamma\in\L2(S^1)$ with 
  \begin{equation*}
    \delta \,\ge\, \delta_0+\|\gamma-\gamma^{0}\|_2 
  \end{equation*}
  and
  \begin{equation*}
    \alpha 
    =\Argmin \|\alpha\|_{\l1}
    \quad \text{s.t.} \quad 
    {\|\gamma-\beta-T^*_{c}\alpha\|_2\le\delta}\,,\;
    \alpha \in \L2(S^1)\,,\; \beta \in L^2(\Omega) \,,
  \end{equation*}
  then
  \begin{subequations}
    \label{eq:L1Complete}
    \begin{align}
      \|\alpha^{0}-\alpha\|_2^{2} 
      \,\le\, \Bigl(1-\frac{2\|\alpha^0\|_{\l0}|\Omega|}
      {\pi}\Bigr)^{-1} 4\delta^{2}\\\noalign{and}
      \|\beta^{0}-\beta\|_2^{2} 
      \,\le\, \Bigl(1-\frac{2\|\alpha^0\|_{\l0}|\Omega|}
      {\pi}\Bigr)^{-1}4\delta^{2} \,.
    \end{align}
  \end{subequations}
\end{theorem}

\begin{proof}
  Proceeding as in \eqref{eq:ProofL1Split-1}--\eqref{eq:ProofL1Split-3}
  we find that
  \begin{equation}
    \label{eq:ProofL1Complete-1}
    \|\alpha-\alpha^{0}\|_{\l1}
    \,\le\, 2\|\alpha-\alpha^{0}\|_{\l1(W)}
  \end{equation}
  with $W$ representing the $\l0$-support of $\alpha^0$.
  Applying similar arguments as in \eqref{eq:ProofL1Split-4} yields
  \begin{equation*}
    \begin{split}
      4\delta^2 
      \,\geq\, \|\alpha-\alpha^0\|_2^2 + \|\beta-\beta^0\|_2^2
      - 2|\langle T^*_c(\alpha-\alpha^0),\beta-\beta^0\rangle| \,.
    \end{split}
  \end{equation*}
  We now use H\"older's inequality, \eqref{eq:DefTc}, the mapping
  properties of the operator which maps $\alpha$ to its Fourier
  coefficients and \eqref{eq:ProofL1Complete-1} to obtain 
  \begin{equation}
    \label{eq:ProofL1Complete-3}
    \begin{split}
      4\delta^2 
      &\,\geq\, \|\alpha-\alpha^0\|_2^2 + \|\beta-\beta^0\|_2^2
      - 2\|T^*_c(\alpha-\alpha^0)\|_{\Linfty}\|\beta-\beta^0\|_{\L1}\\
      &\,=\, \|\alpha-\alpha^0\|_2^2 + \|\beta-\beta^0\|_2^2
      - 2\|\alpha-\alpha^0\|_{\Linfty}\|\beta-\beta^0\|_{\L1}\\
      &\,\geq\, \|\alpha-\alpha^0\|_2^2 + \|\beta-\beta^0\|_2^2
      - \frac{2}{\sqrt{2\pi}} \|\alpha-\alpha^0\|_{\l1}
      \|\beta-\beta^0\|_{\L1}\\
      &\,\geq\, \|\alpha-\alpha^0\|_2^2 + \|\beta-\beta^0\|_2^2
      - \frac{4}{\sqrt{2\pi}} \|\alpha-\alpha^0\|_{\l1(W)}
      \|\beta-\beta^0\|_{\L1}\\
      &\,\geq\, \|\alpha-\alpha^0\|_2^2 + \|\beta-\beta^0\|_2^2
      - \frac{4}{\sqrt{2\pi}} \sqrt{|W|} \|\alpha-\alpha^0\|_2
      \sqrt{|\Omega|} \|\beta-\beta^0\|_2\\
      &\,\geq\, \Bigl( 1 - \frac{2}{\pi} |W||\Omega| \Bigr) 
      \|\alpha-\alpha^0\|_2^2 + \Bigl( \|\beta-\beta^0\|_2 
      - \frac{2}{\sqrt{2\pi}} \sqrt{|W||\Omega|} \|\alpha-\alpha^0\|_2
      \Bigr)^2 \,.
    \end{split}
  \end{equation}
  Dropping the second term gives \eqref{eq:L1Complete} for $\alpha$
  because $|W|=\|\alpha^{0}\|_{\l0}$, and we may interchange the roles
  of $\alpha$ and $\beta$ when completing the square in the last line of
  \eqref{eq:ProofL1Complete-3} to obtain the estimate for~$\beta$. 
\end{proof}

If $\Omega$ is unknown as well, then theorem~\ref{th:L1Complete} can
be adapted as follows.

\begin{corollary}
  \label{th:L1CompleteNO}
  Suppose that $\gamma^0,\alpha^0\in\L2(S^1)$, $\Omega\tm S^1$,
  $\beta^0\in\L2(\Omega)$ and $c\in\R^2$ such that
  \begin{equation*}
    \frac4{\sqrt{2\pi}}\frac1{\tau^2} \|\alpha^0\|_{\l0} \,<\, 1 
    \quad \text{and} \quad
    \frac4{\sqrt{2\pi}}\tau^2|\Omega| \,<\, 1 
    \qquad \text{for some } \tau>0
  \end{equation*}
  and
  \begin{equation*}
    \|\gamma^{0}-T^*_{c}\alpha^{0}-\beta^{0}\|_2 \,\le\, \delta_0 
    \qquad \text{for some } \delta_0\geq0 \,.
  \end{equation*}
  If $\delta\geq0$ and $\gamma\in\L2(S^1)$ with 
  \begin{equation*}
    \delta \,\ge\, \delta_0+\|\gamma-\gamma^{0}\|_2
  \end{equation*}
  and
  \begin{equation*}
    (\alpha, \beta)
    = \Argmin \frac1\tau \|\alpha\|_{\l1} + \tau \|\beta\|_{\L1}
    \quad \text{s.t.} \quad 
    {\|\gamma-T^*_{c}\alpha-\beta\|_2\le\delta}\,,\;
    \alpha, \beta \in \L2(S^1) \,,
  \end{equation*}
  then
  \begin{align*}
    \|\alpha^{0}-\alpha\|_2^{2} 
    \,\le\, \Bigl(1-\frac4{\sqrt{2\pi}}\frac1{\tau^2}
    \|\alpha^0\|_{\l0}\Bigr)^{-1} 4\delta^{2}\\\noalign{and}
    \|\beta^{0}-\beta\|_2^{2} 
    \,\le\, \Bigl(1-\frac4{\sqrt{2\pi}}\tau^2|\Omega|\Bigr)^{-1}
    4\delta^{2} \,.
  \end{align*}
\end{corollary}

\begin{proof}
  Proceeding as in \eqref{eq:ProofL1Split-1}--\eqref{eq:ProofL1Split-3}
  we find that
  \begin{equation}
    \label{eq:ProofL1CompleteNO-1}
    \frac1\tau\|\alpha-\alpha^{0}\|_{\l1} + \tau\|\beta-\beta^0\|_{\L1}
    \,\le\, 2 \bigl(\frac1\tau\|\alpha-\alpha^{0}\|_{\l1(W)} 
    + \tau\|\beta-\beta^0\|_{\L1(\Omega)} \bigr)
  \end{equation}
  with $W$ representing the $\l0$-support of $\alpha^0$.
  Applying similar arguments as in
  \eqref{eq:ProofL1Split-4}--\eqref{eq:ProofL1Split-6} and
  \eqref{eq:ProofL1Complete-3} together with
  \eqref{eq:ProofL1CompleteNO-1} yields 
  \begin{equation*}
    \begin{split}
      4\delta^2 
      &\,\geq\, \|\alpha-\alpha^0\|_2^2 + \|\beta-\beta^0\|_2^2
      - \frac{2}{\sqrt{2\pi}} \|\alpha-\alpha^0\|_{\l1}
      \|\beta-\beta^0\|_{\L1}\\
      &\,\geq\, \|\alpha-\alpha^0\|_2^2 + \|\beta-\beta^0\|_2^2
      - \frac1{2\sqrt{2\pi}} \Bigl( \frac1\tau \|\alpha-\alpha^0\|_{\l1}
      + \tau \|\beta-\beta^0\|_{\L1} \Bigr)^2\\
      &\,\geq\, \|\alpha-\alpha^0\|_2^2 + \|\beta-\beta^0\|_2^2
      - \frac2{\sqrt{2\pi}} \Bigl( \frac1\tau \|\alpha-\alpha^0\|_{\l1(W)}
      + \tau \|\beta-\beta^0\|_{\L1(\Omega)} \Bigr)^2\\
      &\,\geq\, \|\alpha-\alpha^0\|_2^2 + \|\beta-\beta^0\|_2^2
      - \frac2{\sqrt{2\pi}} \Bigl( \frac1\tau \sqrt{|W|}
      \|\alpha-\alpha^0\|_2 + \tau \sqrt{|\Omega|}
      \|\beta-\beta^0\|_2 \Bigr)^2\\ 
      &\,\geq\, \|\alpha-\alpha^0\|_2^2 + \|\beta-\beta^0\|_2^2 -
      \frac4{\sqrt{2\pi}} \Bigl( \frac1{\tau^2} |W| \|\alpha-\alpha^0\|_2^2
      + \tau^2 |\Omega| \|\beta-\beta^0\|_2^2 \Bigr) \,.
    \end{split}
  \end{equation*}
  This ends the proof because $|W|=\|\alpha^{0}\|_{\l0}$.
\end{proof}

A possible application of corollary~\ref{th:L1CompleteNO} is the
problem of removing (high-amplitude) strongly localized noise from
measured far field data. 
Next we consider sources supported on sets with multiple disjoint
components. 

\begin{theorem}
  \label{th:L1SplitMultiple}
  Suppose that $\gamma^0,\alpha_i^0\in\L2(S^1)$ and $c_i\in\R^2$,
  $i=1,\ldots,I$ such that
  \begin{equation}
    \label{eq:L1SplitPositivityMultiple}
    \max_{j\not=k}\frac1{|c_k-c_j|^{\third}}4(I-1)\|\alpha_{i}^{0}\|_{\l0} 
    \,<\, 1 \qquad \text{for each } i
  \end{equation}
  and
  \begin{equation*}
    \|\gamma^{0}-\sum_{i=1}^I T^*_{c_i}\alpha_i^{0}\|_2
    \,\le\, \delta_0 \qquad \text{for some } \delta^0\geq 0 \,.
  \end{equation*}
  If $\delta\geq0$ and $\gamma\in\L2(S^1)$ with
  \begin{equation*}
    \delta \,\ge\, \delta_0+\|\gamma-\gamma^{0}\|_2
  \end{equation*}
  and
  \begin{equation}
    \label{eq:L1SplitMinMultiple}
    (\alpha_1,\ldots,\alpha_I) 
    =\Argmin \sum_{i=1}^I \|\alpha_i\|_{\l1}
    \quad \text{s.t.} \quad 
    {\|\gamma-\sum_{i=1}^IT^*_{c_i}\alpha_i\|_2\le\delta}\,,\;
    \alpha_i\in\L2(S^1) \,,
  \end{equation}
  then, for $i=1,\ldots,I$
  \begin{equation*}
    \|\alpha^{0}_{i}-\alpha_{i}\|_2^{2} 
    \,\le\,
    \Bigl(1-\max_{j\not=k}\frac1{|c_k-c_j|^{\third}}4(I-1)\|\alpha_{i}^{0}\|_{\l0}\Bigr)^{-1} 
    4\delta^{2} \,.
  \end{equation*}
\end{theorem}

\begin{proof}
  Proceeding as in \eqref{eq:ProofL1Split-1}--\eqref{eq:ProofL1Split-3}
  we find that
  \begin{equation}
    \label{eq:ProofL1SplitMultiple-1}
    \sum_{i=1}^I \|\alpha_i-\alpha_i^0\|_{l^1}
    \,\leq\, 2 \sum_{i=1}^I \|\alpha_i-\alpha_i^0\|_{l^1(W_i)}
  \end{equation}
  with $W_i$ representing the \l0-support of $\alpha_i^0$.
  Applying similar arguments as in
  \eqref{eq:ProofL1Split-4}--\eqref{eq:ProofL1Split-5} and using the
  inequality \eqref{eq:BasicIneq3} from
  appendix~\ref{app:SumEstimates} and
  \eqref{eq:ProofL1SplitMultiple-1} we obtain 
  \begin{equation}
    \label{eq:ProofL1SplitMultiple-2}
    \begin{split}
      4\delta^2 
      &\,\geq\, \sum_{i=1}^I\|\alpha_i-\alpha_i^0\|^2_2 -
      \sum_{i=1}^I \sum_{j\not=i}|\langle \alpha_i-\alpha_i^0,
      T^*_{c_j-c_i}(\alpha_j-\alpha_j^0)\rangle|\\
      &\,\geq\, \sum_{i=1}^I\|\alpha_i-\alpha_i^0\|^2_2 -
      \sum_{i=1}^I \sum_{j\not=i}\frac1{|c_i-c_j|^{\third}} 
      \|\alpha_i-\alpha_i^0\|_\l1 \|\alpha_j-\alpha_j^0\|_\l1\\
      &\,\geq\, \sum_{i=1}^I\|\alpha_i-\alpha_i^0\|^2_2 -
      \max_{j\not=k}\frac1{|c_j-c_k|^{\third}} \sum_{i=1}^I \sum_{j\not=i}
      \|\alpha_i-\alpha_i^0\|_\l1 \|\alpha_j-\alpha_j^0\|_\l1\\
      &\,\geq\, \sum_{i=1}^I\|\alpha_i-\alpha_i^0\|^2_2 -
      \max_{j\not=k}\frac1{|c_j-c_k|^{\third}} \frac{I-1}{I} 
      \Bigl( \sum_{i=1}^I \|\alpha_i-\alpha_i^0\|_\l1\Bigr)^2\\
      &\,\geq\, \sum_{i=1}^I\|\alpha_i-\alpha_i^0\|^2_2 -
      \max_{j\not=k}\frac1{|c_j-c_k|^{\third}} \frac{I-1}{I} 4
      \Bigl( \sum_{i=1}^I \|\alpha_i-\alpha_i^0\|_{\l1(W_i)}\Bigr)^2 \,.
    \end{split}
  \end{equation}
  Applying H\"older's inequality and \eqref{eq:BasicIneq2} yields
  \begin{equation}
    \label{eq:ProofL1SplitMultiple-3}
    \begin{split}
      4\delta^2 
      &\,\geq\, \sum_{i=1}^I\|\alpha_i-\alpha_i^0\|^2_2 -
      \max_{j\not=k}\frac1{|c_j-c_k|^{\third}} \frac{I-1}{I} 4
      \Bigl( \sum_{i=1}^I |W_i|^{\half}\|\alpha_i-\alpha_i^0\|_2\Bigr)^2\\
      &\,\geq\, \sum_{i=1}^I\|\alpha_i-\alpha_i^0\|^2_2 -
      \max_{j\not=k}\frac1{|c_j-c_k|^{\third}} 4 (I-1)
      \sum_{i=1}^I |W_i|\|\alpha_i-\alpha_i^0\|_2^2 \,,
    \end{split}
  \end{equation}
  where $|W_i|=\|\alpha_i^{0}\|_{\l0}$.
\end{proof}

As in corollary~\ref{th:L1SplitWS} we can improve these estimates, under
the assumption that some a priori knowledge of the size of the
non-evanescent subspaces is available and that the individual source components
are sufficiently far apart from each other. 

\begin{corollary}
  \label{th:L1SplitMultipleWS}
  If we add to the hypothesis of theorem~\ref{th:L1SplitMultiple}:
  \begin{equation*}
    \alpha_i^0,\alpha_{i}\in\l2(-N_{i},N_{i}) \quad \text{for each } i 
    \qquad \text{and} \qquad
    |c_i-c_j|\,>\,2(N_i+N_j+1) \quad \text{for every } i\not=j
  \end{equation*}
  for some $N_1,\ldots,N_I\in\N$,
  and replace \eqref{eq:L1SplitPositivityMultiple} with
  \begin{equation*}
    \max_{j\not=k}\frac1{|c_k-c_j|^{\half}}4(I-1)\|\alpha_i^0\|_\l0
    \,<\, 1 \qquad \text{for each } i \,,
  \end{equation*}
  the conclusion becomes, for $i=1,\ldots,I$
  \begin{equation*}
    \|\alpha^{0}_{i}-\alpha_{i}\|_2^{2} 
    \,\le\,
    \Bigl(1-\max_{j\not=k}\frac1{|c_k-c_j|^{\half}}4(I-1)\|\alpha_i^0\|_\l0\Bigr)^{-1} 
    4\delta^{2} \,.
  \end{equation*}
\end{corollary}

\begin{proof}
  Replace \eqref{eq:EstTcll} by \eqref{eq:ProofUncertainty3-2} in
  \eqref{eq:ProofL1SplitMultiple-2}. 
\end{proof}

Next we consider multiple source components together with a missing
data component. 

\begin{theorem}
  \label{th:L1CompleteAndSplitMultiple}
  Suppose that $\gamma^0,\alpha_i^0 \in \L2(S^1)$, $c_i\in\R^2$,
  $i=1,\ldots,I$, $\Omega\tm S^1$ and $\beta^0\in\L2(\Omega)$ such
  that
  \begin{subequations}
    \label{eq:L1CompleteAndSplitPositivityMultiple}
    \begin{align}
      \label{eq:L1CompleteAndSplitPositivityMultiple_a}
      \frac{2}{\sqrt{2\pi}} \sum_{i=1}^I \sqrt{|\Omega|\|\alpha_i^0\|_\l0}
      &\,<\, 1 \,,\\
      \label{eq:L1CompleteAndSplitPositivityMultiple_b}
      \max_{j\not=k}\frac1{|c_k-c_j|^{\third}}4(I-1)\|\alpha_{i}^{0}\|_{\l0} 
      + \frac{2}{\sqrt{2\pi}} \sqrt{|\Omega|\|\alpha^0_i\|_\l0}
      &\,<\, 1 \qquad \text{for each } i \,,
    \end{align}
  \end{subequations}
  and
  \begin{equation*}
    \|\gamma^{0}-\beta^0-\sum_{i=1}^I T^*_{c_i}\alpha_i^{0}\|_2
    \,\le\, \delta_0 \qquad \text{for some } \delta_0\geq0 \,.
  \end{equation*}
  If $\delta\geq0$ and $\gamma\in\L2(S^1)$ with 
  \begin{equation*}
    \delta \,\ge\, \delta_0+\|\gamma-\gamma^{0}\|_2
  \end{equation*}
  and
  \begin{equation}
    \label{eq:L1CompleteAndSplitMinMultiple}
    (\alpha_1,\ldots,\alpha_I) 
    =\Argmin \sum_{i=1}^I \|\alpha_i\|_{\l1}
    \quad \text{s.t.} \quad 
    {\|\gamma-\beta-\sum_{i=1}^IT^*_{c_i}\alpha_i\|_2\le\delta}\,,\;
    \alpha_i\in\L2(S^1) \,,\; \beta\in\L2(\Omega) \,,
  \end{equation}
  then
  \begin{subequations}
    \label{eq:L1CompleteAndSplitMultiple}
    \begin{align}
      \label{eq:L1CompleteAndSplitMultiple_a}
      \|\beta^0-\beta\|_2^2 
      &\,\le\, \biggl(1-\frac{2}{\sqrt{2\pi}}
      \sum_{i=1}^I \sqrt{|\Omega|\|\alpha_i^0\|_\l0}\biggr)^{-1} 
      4\delta^{2}\\\noalign{and, for $i=1,\ldots,I$}
      \label{eq:L1CompleteAndSplitMultiple_b}
      \|\alpha^{0}_{i}-\alpha_{i}\|_2^{2} 
      &\,\le\,
      \Bigl(1-\max_{j\not=k}\frac1{|c_k-c_j|^{\third}}
      4(I-1)\|\alpha_{i}^{0}\|_{\l0} 
      - \frac{2}{\sqrt{2\pi}} 
      \sqrt{|\Omega|\|\alpha^0_i\|_\l0}\Bigr)^{-1} 4\delta^{2} \,.
    \end{align}
  \end{subequations}
\end{theorem}

\begin{proof}
  Proceeding as in \eqref{eq:ProofL1Split-1}--\eqref{eq:ProofL1Split-3}
  we find that
  \begin{equation}
    \label{eq:ProofL1CompleteAndSplitMultiple-1}
    \sum_{i=1}^I \|\alpha_i-\alpha_i^0\|_{l^1}
    \,\leq\, 2 \sum_{i=1}^I \|\alpha_i-\alpha_i^0\|_{l^1(W_i)}
  \end{equation}
  with $W_i$ representing the \l0-support of $\alpha_i^0$.
  Applying similar arguments as in \eqref{eq:ProofL1Split-4} we obtain
  \begin{equation}
    \label{eq:ProofL1CompleteAndSplitMultiple-2}
    \begin{split}
      4\delta^2 
      &\,\geq\, \sum_{i=1}^I\|\alpha_i-\alpha_i^0\|^2_2 
      + \|\beta-\beta^0\|_2^2
      -\sum_{i=1}^I \sum_{j\not=i}|\langle \alpha_i-\alpha_i^0,
      T^*_{c_j-c_i}(\alpha_j-\alpha_j^0)\rangle|\\
      &\phantom{\,\geq\,}
      -2\sum_{i=1}^I|\langle T^*_{c_i}(\alpha_i-\alpha_i^0),
      \beta-\beta^0\rangle|\,.
    \end{split}
  \end{equation}
  The third term on the right hand side of 
  \eqref{eq:ProofL1CompleteAndSplitMultiple-2} can be estimated as in
  \eqref{eq:ProofL1SplitMultiple-2}--\eqref{eq:ProofL1SplitMultiple-3},
  while for the last term we find using H\"older's inequality, the
  mapping properties of the operator which maps $\alpha$ to its
  Fourier coefficients and
  \eqref{eq:ProofL1CompleteAndSplitMultiple-1} that 
  \begin{equation*}
    \begin{split}
      2\sum_{i=1}^I|\langle T^*_{c_i}(\alpha_i-\alpha_i^0),
      &\beta-\beta^0\rangle|
      \,\le\, 2\Bigl(\sum_{i=1}^I\|\alpha_i-\alpha_i^0\|_\Linfty\Bigr)
      \|\beta-\beta^0\|_\L1\\
      &\,\le\, \frac{2}{\sqrt{2\pi}}
      \Bigl(\sum_{i=1}^I\|\alpha_i-\alpha_i^0\|_\l1\Bigr)
      \|\beta-\beta^0\|_\L1\\
      &\,\le\, \frac{4}{\sqrt{2\pi}}
      \Bigl(\sum_{i=1}^I\|\alpha_i-\alpha_i^0\|_{\l1(W_i)}\Bigr)
      \|\beta-\beta^0\|_\L1\\
      &\,\le\, \frac{4}{\sqrt{2\pi}}
      \Bigl(\sum_{i=1}^I\sqrt{|W_i|}\|\alpha_i-\alpha_i^0\|_2\Bigr)
      \sqrt{|\Omega|}\|\beta-\beta^0\|_2\\
      &\,\le\, \frac{2}{\sqrt{2\pi}} \Bigl(
      \sum_{i=1}^I \sqrt{|\Omega||W_i|}\|\alpha_i-\alpha_i^0\|_2^2
      + \sum_{i=1}^I \sqrt{|\Omega||W_i|}\|\beta-\beta^0\|_2^2 \Bigr) \,,
    \end{split}
  \end{equation*}
  where $|W_i|=\|\alpha_i^{0}\|_{\l0}$.
  Combining these estimates ends the proof.
\end{proof}

Again, including a priori information of the size of the
non-evanescent subspaces and assuming that the individual source
components are well separated, the result can be improved:

\begin{corollary}
  \label{th:L1CompleteAndSplitMultipleWS}
  If we add to the hypothesis of theorem~\ref{th:L1CompleteAndSplitMultiple}:
  \begin{equation*}
    \alpha_i^0,\alpha_{i}\in\l2(-N_{i},N_{i}) \quad \text{for each } i
    \qquad \text{and} \qquad
    |c_i-c_j|\,>\,2(N_i+N_j+1) \quad \text{for every } i\not=j
  \end{equation*}
  for some $N_1,\ldots,N_I\in\N$, and replace
  \eqref{eq:L1CompleteAndSplitPositivityMultiple_b} with 
  \begin{equation*}
    \max_{j\not=k}\frac1{|c_k-c_j|^{\half}}4(I-1)\|\alpha_{i}^{0}\|_{\l0} 
    + \frac{2}{\sqrt{2\pi}} \sqrt{|\Omega|\|\alpha_i^0\|_\l0}
    \,<\, 1 \qquad \text{for each } i \,,
  \end{equation*}
  the conclusion \eqref{eq:L1CompleteAndSplitMultiple_b} becomes, for
  $i=1,\ldots,I$ 
  \begin{equation*}
    \|\alpha^{0}_{i}-\alpha_{i}\|_2^{2} 
    \,\le\,
    \Bigl(1-\max_{j\not=k}\frac1{|c_k-c_j|^{\half}}4(I-1)\|\alpha_{i}^{0}\|_{\l0} 
    + \frac{2}{\sqrt{2\pi}}\sqrt{|\Omega|\|\alpha_i^0\|_\l0}\Bigr)^{-1} 
    4\delta^{2} \,.
  \end{equation*}
\end{corollary}

Finally, we establish variants of theorem~\ref{th:L1SplitMultiple}
and theorem~\ref{th:L1CompleteAndSplitMultiple}, where we replace the
\l1 minimization in \eqref{eq:L1SplitMinMultiple} and
\eqref{eq:L1CompleteAndSplitMinMultiple} by a weighted \l1
minimization in order to obtain better estimates for certain geometric
configurations of the supports of the individual source components.
In contrast to theorem~\ref{th:L1SplitMultiple} the constant in the
stability estimate \eqref{eq:L1SplitMultipleN} in
theorem~\ref{th:L1SplitMultipleN} below only depends on the distances
of source components relative to the source component corresponding to
the far field component appearing on the left hand side of the
estimate. 

\begin{theorem}
  \label{th:L1SplitMultipleN}
  Suppose that $\gamma^0,\alpha_i^0\in\L2(S^1)$ and $c_i\in\R^2$, and
  set
  \begin{equation}
    \label{eq:L1SplitMultipleAssumptionN}
    a_i^2 
    \,=\, \max_{j\not=i} \Bigl(\frac{2}{|c_i-c_j|}\Bigr)^{\third} 
    \qquad\text{or}\qquad
    a_i^2 
    \,=\, \max_{\substack{j\not=i\\k\not=i,j}} 
    \Bigl( \frac{1}{|c_i-c_j|} + \frac{1}{|c_i-c_k|}\Bigr)^{\third} \,,
  \end{equation}
  $i=1,\ldots,I$.
  Assume that
  \begin{equation*}
    4(I-1)a_i^2\|\alpha_{i}^{0}\|_{\l0} 
    \,<\, 1 \qquad \text{for each } i \,,
  \end{equation*}
  and
  \begin{equation*}
    \|\gamma^{0}-\sum_{i=1}^I T^*_{c_i}\alpha_i^{0}\|_2
    \,\le\, \delta_0  \qquad \text{for some } \delta_0\geq0 \,.
  \end{equation*}
  If $\delta\geq0$ and $\gamma\in\L2(S^1)$ with 
  \begin{equation*}
    \delta \,\ge\, \delta_0+\|\gamma-\gamma^{0}\|_2
  \end{equation*}
  and
  \begin{equation*}
    (\alpha_1,\ldots,\alpha_I) 
    \,=\,\Argmin \sum_{i=1}^I a_i \|\alpha_i\|_{\l1}
    \quad \text{s.t.} \quad 
    {\|\gamma-\sum_{i=1}^IT^*_{c_i}\alpha_i\|_2\le\delta}\,,\;
    \alpha_i\in\L2(S^1) \,,
  \end{equation*}
  then, for $i=1,\ldots,I$
  \begin{equation}
    \label{eq:L1SplitMultipleN}
    \|\alpha^{0}_{i}-\alpha_{i}\|_2^{2} 
    \,\le\,
    \bigl(1-4(I-1)a_i^2\|\alpha_{i}^{0}\|_{\l0}\bigr)^{-1} 
    4\delta^{2} \,.
  \end{equation}
\end{theorem}

\begin{proof}
  Proceeding as in \eqref{eq:ProofL1Split-1}--\eqref{eq:ProofL1Split-3}
  we find
  \begin{equation}
    \label{eq:ProofL1SplitMultipleN-1}
    \sum_{i=1}^I a_i \|\alpha_i-\alpha_i^0\|_{l^1}
    \,\leq\, 2 \sum_{i=1}^I a_i \|\alpha_i-\alpha_i^0\|_{l^1(W_i)}
  \end{equation}
  with $W_i$ representing the \l0-support of $\alpha_i^0$.

  Abbreviating $c_{ij} = c_i-c_j$, the triangle inequality shows
  \begin{equation*}
    |c_{ij}| + |c_{ik}|
    \,=\, |c_i-c_j| + |c_i-c_k|
    \,\geq\, |c_j-c_k| 
    \,=\, |c_{jk}| \,,
  \end{equation*}
  i.e.,
  \begin{equation*}
    \frac{1}{|c_{ij}|} + \frac{1}{|c_{ik}|}
    \,\geq\, \frac{|c_{jk}|}{|c_{ij}||c_{ik}|} \,.
  \end{equation*}
  Thus our assumption \eqref{eq:L1SplitMultipleAssumptionN} implies
  \begin{equation*}
    a_i^2
    \,\geq\, \Bigl( \frac{|c_{jk}|}{|c_{ij}||c_{ik}|} \Bigr)^{\frac13}
    \qquad \text{for every } j\not=k \,,
  \end{equation*}
  and therefore, using \eqref{eq:BasicIneq1}
  \begin{equation}
    \label{eq:ProofL1SplitMultipleN-5}
    \begin{split}
      \Bigl(\sum_{i=1}^Ia_i\|\alpha_i\|_\l1\Bigr)^2 
      &\,=\, \sum_{i=1}^Ia_i^2\|\alpha_i\|_\l1^2 
      + \sum_{i=1}^I\sum_{j\not=i}a_ia_j\|\alpha_i\|_\l1\|\alpha_j\|_\l1\\
      &\,\geq\, \frac{I}{I-1} 
      \sum_{i=1}^I\sum_{j\not=i}a_ia_j\|\alpha_i\|_\l1\|\alpha_j\|_\l1\\
      &\,\geq\, \frac{I}{I-1} 
      \sum_{i=1}^I\sum_{j\not=i}
      \Bigl(\frac{|c_{jk}|}{|c_{ij}||c_{ik}|}\Bigr)^{\frac16}
      \Bigl(\frac{|c_{ik}|}{|c_{ij}||c_{jk}|}\Bigr)^{\frac16}
      \|\alpha_i\|_\l1\|\alpha_j\|_\l1\\
      &\,=\, \frac{I}{I-1} \sum_{i=1}^I\sum_{j\not=i} 
      \frac{1}{|c_{ij}|^{\frac13}} \|\alpha_i\|_\l1\|\alpha_j\|_\l1 \,.
    \end{split}
  \end{equation}
  Now, applying similar arguments as in
  \eqref{eq:ProofL1Split-4}--\eqref{eq:ProofL1Split-5}, using
  \eqref{eq:ProofL1SplitMultipleN-1} and
  \eqref{eq:ProofL1SplitMultipleN-5}, we obtain
  \begin{equation}
    \label{eq:ProofL1SplitMultipleN-6}
    \begin{split}
      4\delta^2 
      &\,\geq\, \sum_{i=1}^I\|\alpha_i-\alpha_i^0\|^2_2 -
      \sum_{i=1}^I \sum_{j\not=i}|\langle \alpha_i-\alpha_i^0,
      T^*_{c_j-c_i}(\alpha_j-\alpha_j^0)\rangle|\\
      &\,\geq\, \sum_{i=1}^I\|\alpha_i-\alpha_i^0\|^2_2 
      - \sum_{i=1}^I \sum_{j\not=i}\frac1{|c_{ij}|^{\third}} 
      \|\alpha_i-\alpha_i^0\|_\l1 \|\alpha_j-\alpha_j^0\|_\l1\\
      &\,\geq\, \sum_{i=1}^I\|\alpha_i-\alpha_i^0\|^2_2 
      - \frac{I-1}{I} 
      \Bigl( \sum_{i=1}^I a_i \|\alpha_i-\alpha_i^0\|_\l1\Bigr)^2
    \end{split}
  \end{equation}
  Using H\"older's inequality and \eqref{eq:BasicIneq2} we deduce
  \begin{equation}
    \label{eq:ProofL1SplitMultipleN-7}
    \begin{split}
      4\delta^2 
      &\,\geq\, \sum_{i=1}^I\|\alpha_i-\alpha_i^0\|^2_2 
      - 4 \frac{I-1}{I} 
      \Bigl( \sum_{i=1}^I a_i \|\alpha_i-\alpha_i^0\|_{\l1(W_i)}\Bigr)^2\\
      &\,\geq\, \sum_{i=1}^I\|\alpha_i-\alpha_i^0\|^2_2 
      - 4 \frac{I-1}{I}
      \Bigl( \sum_{i=1}^I a_i|W_i|^{\half}\|\alpha_i-\alpha_i^0\|_2\Bigr)^2\\
      &\,\geq\, \sum_{i=1}^I\|\alpha_i-\alpha_i^0\|^2_2 
      - 4 (I-1)
      \sum_{i=1}^I a_i^2|W_i|\|\alpha_i-\alpha_i^0\|_2^2 \,,
    \end{split}
  \end{equation}
  where $|W_i|=\|\alpha_i^{0}\|_{\l0}$.
  This ends the proof.
\end{proof}

\begin{corollary}
  \label{th:L1SplitMultipleNWS}
  If we add to the hypothesis of theorem~\ref{th:L1SplitMultipleN}:
  \begin{equation*}
    \alpha_i^0,\alpha_{i}\in\l2(-N_{i},N_{i}) \quad \text{for each } i
    \qquad \text{and} \qquad 
    |c_i-c_j|\,>\,2(N_i+N_j+1) \quad \text{for every } j\not=i
  \end{equation*}
  for some $N_1,\ldots,N_I\in\N$, and replace
  \eqref{eq:L1SplitMultipleAssumptionN} with 
  \begin{equation*}
    a_i^2 
    \,=\, \max_{j\not=i} \Bigl(\frac{2}{|c_i-c_j|}\Bigr)^{\half} 
    \qquad\text{or}\qquad
    a_i^2 
    \,=\, \max_{\substack{j\not=i\\k\not=i,j}} 
    \Bigl( \frac{1}{|c_i-c_j|} + \frac{1}{|c_i-c_k|}\Bigr)^{\half}
    \qquad \text{for each } i \,,
  \end{equation*}
  the conclusion remains true.
\end{corollary}

\begin{theorem}
  \label{th:L1CompleteAndSplitMultipleN}
  Suppose that $\gamma^0,\alpha_i^0\in\L2(S^1)$, $c_i\in\R^2$,
  $\Omega\tm S^1$ and $\beta^0\in\L2(\Omega)$, and set
  \begin{equation}
    \label{eq:L1CompleteAndSplitMultipleAssumptionN}
    a_i^2 
    \,=\, \max_{j\not=i} \Bigl(\frac{2}{|c_i-c_j|}\Bigr)^{\third} 
    \qquad\text{or}\qquad
    a_i^2 
    \,=\, \max_{\substack{j\not=i\\k\not=i,j}} 
    \Bigl( \frac{1}{|c_i-c_j|} + \frac{1}{|c_i-c_k|}\Bigr)^{\third} \,,
  \end{equation}
  $i=1,\ldots,I$. 
  Assume that
  \begin{align*}
    \frac{2}{\sqrt{2\pi}} \Bigl(\max_j\frac{1}{a_j}\Bigr) 
    \sum_{i=1}^I a_i \sqrt{|\Omega|\|\alpha_i^0\|_\l0}
    &\,<\, 1 \,,\\
    4(I-1)a_i^2\|\alpha_{i}^{0}\|_{\l0} 
    + \frac{2}{\sqrt{2\pi}} \Bigl(\max_j\frac{1}{a_j}\Bigr) 
    a_i\sqrt{|\Omega|\|\alpha^0_i\|_\l0}
    &\,<\, 1 \qquad \text{for each } i \,,
  \end{align*}
  and
  \begin{equation*}
    \|\gamma^{0}-\beta^0-\sum_{i=1}^I T^*_{c_i}\alpha_i^{0}\|_2
    \,\le\, \delta_0 \qquad \text{for some } \delta_0\geq0 \,.
  \end{equation*}
  If $\delta\geq0$ and $\gamma\in\L2(S^1)$ with 
  \begin{equation*}
    \delta \,\ge\, \delta_0+\|\gamma-\gamma^{0}\|_2
  \end{equation*}
  and
  \begin{equation*}
    (\alpha_1,\ldots,\alpha_I) 
    =\Argmin \sum_{i=1}^I a_i\|\alpha_i\|_{\l1}
    \quad \text{s.t.} \quad 
    {\|\gamma-\beta-\sum_{i=1}^IT^*_{c_i}\alpha_i\|_2\le\delta}\,,\;
    \alpha_i\in\L2(S^1) \,,\; \beta\in\L2(\Omega) \,,
  \end{equation*}
  then
  \begin{align*}
    \|\beta^0-\beta\|_2^2 
    &\,\le\, \biggl(1-\frac{2}{\sqrt{2\pi}} 
    \Bigl(\max_j\frac{1}{a_j}\Bigr) 
    \sum_{i=1}^I a_i \sqrt{|\Omega|\|\alpha_i^0\|_\l0}\biggr)^{-1} 
    4\delta^{2} \,,\\\noalign{and, for $i=1,\ldots,I$}
    \|\alpha^{0}_{i}-\alpha_{i}\|_2^{2} 
    &\,\le\,
    \biggl(1-4(I-1)a_i^2\|\alpha_{i}^{0}\|_{\l0} 
    + \frac{2}{\sqrt{2\pi}} \Bigl(\max_j\frac{1}{a_j}\Bigr) 
    a_i\sqrt{|\Omega|\|\alpha^0_i\|_\l0}\biggr)^{-1} 
    4\delta^{2} \,.
  \end{align*}
\end{theorem}

\begin{proof}
  Proceeding as in \eqref{eq:ProofL1Split-1}--\eqref{eq:ProofL1Split-3}
  we find that
  \begin{equation}
    \label{eq:ProofL1CompleteAndSplitMultipleN-1}
    \sum_{i=1}^I a_i \|\alpha_i-\alpha_i^0\|_{l^1}
    \,\leq\, 2 \sum_{i=1}^I a_i \|\alpha_i-\alpha_i^0\|_{l^1(W_i)}
  \end{equation}
  with $W_i$ representing the \l0-support of $\alpha_i^0$.
  Applying similar arguments as in \eqref{eq:ProofL1Split-4} we obtain
  \begin{equation}
    \label{eq:ProofL1CompleteAndSplitMultipleN-2}
    \begin{split}
      4\delta^2 
      &\,\geq\, \sum_{i=1}^I\|\alpha_i-\alpha_i^0\|^2_2 
      + \|\beta-\beta^0\|_2^2
      -\sum_{i=1}^I \sum_{j\not=i}|\langle \alpha_i-\alpha_i^0,
      T^*_{c_j-c_i}(\alpha_j-\alpha_j^0)\rangle|\\
      &\phantom{\,\geq\,}
      -2\sum_{i=1}^I|\langle T^*_{c_i}(\alpha_i-\alpha_i^0),
      \beta-\beta^0\rangle|\,.
    \end{split}
  \end{equation}
  The third term on the right hand side of 
  \eqref{eq:ProofL1CompleteAndSplitMultipleN-2} can be estimated as in
  \eqref{eq:ProofL1SplitMultipleN-6}--\eqref{eq:ProofL1SplitMultipleN-7},
  while for the last term we find using H\"older's inequality, the
  mapping properties of the operator which maps $\alpha$ to its
  Fourier coefficients, and \eqref{eq:ProofL1CompleteAndSplitMultipleN-1} 
  that
  \begin{equation*}
    \begin{split}
      2\sum_{i=1}^I|\langle T^*_{c_i}(\alpha_i-\alpha_i^0),
      &\beta-\beta^0\rangle|
      \,\le\, 2 \Bigl(\sum_{i=1}^I\|\alpha_i-\alpha_i^0\|_\Linfty\Bigr)
      \|\beta-\beta^0\|_\L1\\
      &\,\le\, \frac{2}{\sqrt{2\pi}} \max_{j}\frac{1}{a_j}
      \Bigl(\sum_{i=1}^Ia_i\|\alpha_i-\alpha_i^0\|_\l1\Bigr)
      \|\beta-\beta^0\|_\L1\\
      &\,\le\, \frac{4}{\sqrt{2\pi}} \max_{j}\frac{1}{a_j}
      \Bigl(\sum_{i=1}^Ia_i\|\alpha_i-\alpha_i^0\|_{\l1(W_i)}\Bigr)
      \|\beta-\beta^0\|_\L1 \,.
    \end{split}
  \end{equation*}
  Applying H\"older's inequality once more yields
  \begin{equation*}
    \begin{split}
      2\sum_{i=1}^I|\langle T^*_{c_i}(\alpha_i&-\alpha_i^0),
      \beta-\beta^0\rangle|
      \,\le\, \frac{4}{\sqrt{2\pi}} \max_{j}\frac{1}{a_j}
      \Bigl(\sum_{i=1}^Ia_i\sqrt{|W_i|}\|\alpha_i-\alpha_i^0\|_2\Bigr)
      \sqrt{|\Omega|}\|\beta-\beta^0\|_2\\
      &\,=\, \frac{4}{\sqrt{2\pi}} \max_{j}\frac{1}{a_j}
      \sum_{i=1}^I\Bigl(\sqrt{a_i}
      (|\Omega||W_i|)^{\frac14}\|\alpha_i-\alpha_i^0\|_2
      \sqrt{a_i} (|\Omega||W_i|)^{\frac14}\|\beta-\beta^0\|_2\Bigr)\\
      &\,\le\, \frac{2}{\sqrt{2\pi}} \max_{j}\frac{1}{a_j}
      \sum_{i=1}^I\Bigl(a_i(|\Omega||W_i|)^{\frac12}
      \|\alpha_i-\alpha_i^0\|_2^2
      + a_i(|\Omega||W_i|)^{\frac12}\|\beta-\beta^0\|_2^2\Bigr) \,,
    \end{split}
  \end{equation*}
  where $|W_i|=\|\alpha_i^{0}\|_{\l0}$.
  Combining these estimates ends the proof.
\end{proof}

\begin{corollary}
  \label{th:L1CompleteAndSplitMultipleNWS}
  If we add to the hypothesis of
  theorem~\ref{th:L1CompleteAndSplitMultipleN}: 
  \begin{equation*}
    \alpha_i^0,\alpha_{i}\in\l2(-N_{i},N_{i}) \quad \text{for each } i
    \qquad \text{and} \qquad 
    |c_i-c_j|\,>\,2(N_i+N_j+1) \quad \text{for every } j\not=i
  \end{equation*}
  for some $N_1,\ldots,N_I\in\N$, and replace
  \eqref{eq:L1CompleteAndSplitMultipleAssumptionN} with 
  \begin{equation*}
    a_i^2 
    \,=\, \max_{j\not=i} \Bigl(\frac{2}{|c_i-c_j|}\Bigr)^{\half} 
    \qquad\text{or}\qquad
    a_i^2 
    \,=\, \max_{\substack{j\not=i\\k\not=i,j}} 
    \Bigl( \frac{1}{|c_i-c_j|} + \frac{1}{|c_i-c_k|}\Bigr)^{\half}
    \qquad \text{for each } i \,,
  \end{equation*}
  the conclusion remains true.
\end{corollary}

\section{Conditioning, resolution, and  wavelength}
\label{sec:Conditioning}
So far, we have suppressed the dependence on the wavenumber $k$. 
We restore it here, and consider the consequences related to
conditioning and resolution. 
We confine our discussion to theorem~\ref{th:L2SplitLS}, assuming that
the $\gamma^j$, $j=1,2$, represent far fields that are radiated by
superpositions of limited power sources supported in balls
$B_{R_i}(c_i)$, $i=1,2$, and that accordingly, for $k=1$ (following
our discussion at the end of section~\ref{sec:RegPicard}), the numbers
$N_i \gtrsim R_i$ are just a little bigger than the radii of these
balls. 
This becomes $N_i \gtrsim kR_i$ when we return to conventional units,
and the estimate \eqref{eq:L2SplitLS} then depends on the quantity
\begin{equation}
  \label{eq:Conditioning1}
  \frac{(2N_1+1)(2N_2+1)}{k|c_1-c_2|} \,.
\end{equation}

Writing $V_i := T_{c_i}^* \l2(-N_i,N_i)$ and denoting by 
$P_i: \l2 \to \l2$ the orthogonal projection onto $V_i$, $i=1,2$, we
have $V_1\cap V_2 = \{0\}$ if $c_1\not= c_2$, and the angle
$\theta_{12}$ between these subspaces is given by
\begin{equation*}
  \cos \theta_{12}
  \,=\, \sup_{\substack{\alpha_1\in V_1\\\alpha_2\in V_2}}
  \frac{|\langle\alpha_1,\alpha_2\rangle|}{\|\alpha_1\|_2\|\alpha_2\|_2}
  \,=\, \sup_{\alpha_1,\alpha_2\in \l2} 
  \frac{|\langle P_1\alpha_1,P_2\alpha_2\rangle|}{\|\alpha_1\|_2\|\alpha_2\|_2}
  \,=\, \|P_1P_2\|_{\l2,\l2} \,.
\end{equation*}
A glance at the proof of lemma~\ref{lm:L2Split} reveals that the
square root of \eqref{eq:Conditioning1} is just  a lower bound for this
cosine. 
Furthermore, the least squares solutions to \eqref{eq:split2LS} can be
constructed from simple formulas
\begin{align*}
  \alpha_1^j 
  &\,=\, (I-P_1P_2)^{-1}P_1(I-P_2) \gamma^j
  \,=:\, P_{1|2} \gamma^j \,,\\
  \alpha_2^j 
  &\,=\, (I-P_2P_1)^{-1}P_2(I-P_1) \gamma^j
  \,=:\, P_{2|1} \gamma^j \,,
\end{align*}
where $P_{1|2}$ and $P_{2|1}$ denote the projection onto $V_1$ along
$V_2$ and vice versa.
These satisfy
\begin{equation*}
  \|P_{1|2}\|_{\l2,\l2}
  \,=\, \|P_{2|1}\|_{\l2,\l2}
  \,=\, \csc\theta_{12}
  \,=\, \Bigl(\frac1{1-\cos^2\theta_{12}}\Bigr)^{1/2} \,.
\end{equation*}
Consequently $\csc\theta_{12}$ is the absolute condition number for
the splitting problem \eqref{eq:split2LS}, and
Theorem~\ref{th:L2SplitLS} (with our choice of $N_1$ and $N_2$)
essentially says that 
\begin{equation}
  \label{eq:CscEstimate}
  \csc(\theta_{12}) 
  \,\le\, \frac{1}{\sqrt{1-\frac{(2N_1+1)(2N_2+1)}{k|c_1-c_2|}}}
  \,\lesssim\, \frac{1}{\sqrt{1-\frac{(2kR_1+1)(2kR_2+1)}{k|c_1-c_2|}}} \,.
\end{equation}

We will include an example below to show that, at least for large
distances, the dependence on $k$ in estimate in \eqref{eq:CscEstimate}
is sharp. 
This means that, for a fixed geometry $((c_1,R_1),(c_2,R_2))$, the
condition number increases with $k$. 
Because resolution is proportional to
wavelength, this means that we cannot increase resolution by simply
increasing the wavenumber without increasing the dynamic range of the
sensors (i.e.\ the number of significant figures in the measured
data). 
Note that as $k$ increases, the dimensions of the subspaces 
$V_{i} = T_{c_i}^*\l2(-N_i,N_i) \approx T_{c_i}^*\l2(-kR_{i},kR_{i})$
increase. 
The increase in the number of significant Fourier coefficients
(non-evanescent Fourier modes) is the way we see higher resolution in
this problem. 

The situation changes considerably if we replace the limited power
source radiated from $B_{R_1}(c_1)$ by a point source with singularity
in $c_1$. 
Then we can choose for $V_1$ a one-dimensional subspace of $\l2$
(spanned by the zeroth order Fourier mode translated by $T_{c_1}^*$),
and accordingly set $N_1=R_1=0$. 
Consequently, the estimate \eqref{eq:CscEstimate} reduces to
\begin{equation}
  \label{eq:LinSampEstimate}
  \csc(\theta_{12}) 
  \,\le\, \frac{1}{\sqrt{1-\frac{2N_2+1}{k|c_1-c_2|}}}
  \,\lesssim\, \frac{1}{\sqrt{1-\frac{2kR_2+1}{k|c_1-c_2|}}} \,.
\end{equation}
Since numerator and denominator have the same units, the conditioning
of the splitting operator does not depend on $k$ in this case. 

This has immediate consequences for the inverse scattering problem:
Qualitative reconstruction methods like the linear sampling method
\cite{CakCol14} or the factorization method \cite{KirGri08} determine
the support of an unknown scatterer by testing pointwise whether the
far field of a point source belongs to the range of a certain
restricted far field operator, mapping sources supported inside the
scatterer to their radiated far field.
The inequality \eqref{eq:LinSampEstimate} indeed shows that (using
these qualitative reconstruction algorithms for the inverse scattering
problem) one can increase resolution by simply increasing the wave
number.

Finally, if we replace both sources by point sources with
singularities in $c_1$ and $c_2$, respectively, then we can choose
both subspaces $V_1$ and $V_2$ to be one-dimensional, and accordingly
set $N_1=N_2=R_1=R_2=0$. 
The estimate  \eqref{eq:CscEstimate} reduces to
\begin{equation}
  \label{eq:MUSICEstimate}
  \csc(\theta_{12}) 
  \,\le\, \frac{1}{\sqrt{1-\frac{1}{k|c_1-c_2|}}} \,,
\end{equation}
i.e., in this case the conditioning of the splitting operator improves
with increasing wave number $k$. 
MUSIC-type reconstruction methods \cite{Dev12} for inverse scattering
problems with infinitesimally small scatterers recover the locations
of a collection of unknown small scatterers by testing pointwise
whether the far field of a point source belongs to the range of a
certain restricted far field operator, mapping point sources with
singularities at the positions of the small scatterers to their
radiated far field. 
From \eqref{eq:MUSICEstimate} we conclude that (using MUSIC-type
reconstruction algorithms for the inverse scattering problem with
infinitesimally small scatterers) on can increase resolution by simply
increasing the wave number and the reconstruction becomes more stable
for higher frequencies.

\section{An analytic example}
\label{sec:AnalyticExample}
The example below illustrates that the estimate of
the cosine of the angle between two far fields radiated by two sources
supported in balls $B_{R_1}(c_1)$ and $B_{R_2}(c_2)$, respectively,
cannot be better than proportional to the quantity
\begin{equation*}
  \sqrt{\frac{kR_1R_2}{|c_1-c_2|}} \,.
\end{equation*}

As pointed out in the previous section, we need only
construct the example for $k=1$. We will let $f$ be a
single layer source supported on a horizontal line segment
of width $W$, and $g$ be the same source, translated
vertically by a distance $d$ (i.e., $c_1=(0,0)$ and
$c_2=(0,d)$).  Specifically, with $H$ denoting the
Heavyside or indicator function, and $\delta$ the dirac
mass:
\begin{align*}
  f &\,=\, \frac{1}{\sqrt{W}}H_{|x|<W}\delta_{y=0}\\
  g &\,=\, \frac{1}{\sqrt{W}}H_{|x|<W}\delta_{y=d}
\end{align*}
The far fields radiated by $f$ and $g$ are:
\begin{align*}
  \alpha_{f}(\theta) 
  &\,=\, \Fcal f 
  \,=\, 2\frac{\sin(W\cos t)}{\sqrt{W}\cos t} 
  \\
  \alpha_{g}(\theta) 
  &\,=\, \Fcal g
  \,=\, e^{-\rmi d\sin t} \, 
  2\frac{\sin(W\cos t)}{\sqrt{W}\cos t} 
\end{align*}
for $\theta = (\cos t, \sin t) \in S^1$.
Accordingly
\begin{equation*}
  \begin{split}
    \|\alpha_{f}\|_2^{2} 
    \,=\, \|\alpha_{g}\|_2^{2}
    &\,=\, 4 \int_{0}^{2\pi} \frac{\sin^{2}(W\cos t)}
    {(W\cos t)^{2}} W \dt
    \,=\, 8 \int_{-W}^{W} \frac{\sin^{2}(\xi)}
    {\xi^{2}}\frac{1}{\sqrt{1-\xi^{2}}} \dxi\\
    &\,\ge\, 8 \int_{-W}^{W} \frac{\sin^{2}(\xi)}{\xi^{2}}\dxi
    \,=\, 8 \int_{-\infty}^{\infty} \frac{\sin^{2}(\xi)}{\xi^{2}}\dxi 
    - 16\int_W^{\infty} \frac{\sin^{2}(\xi)}{\xi^{2}}\dxi \,,
  \end{split}
\end{equation*}
and we can evaluate the first integral on the right hand side using
the Plancherel equality as $\frac{\sin\xi}{2\xi}$ is the Fourier
transform of the characteristic function of the interval $[-1,1]$, and
estimate the second, yielding 
\begin{equation*}
  \|\alpha_{f}\|_2^{2}
  \,\ge\, 8\Bigl(\pi- \frac{2}{W}\Bigr) \,.
\end{equation*}

On the other hand, for  $d \gg W$, according to the principle of
stationary phase (there are stationary points at  $\pm\frac{\pi}{2}$)
\begin{equation*}
  \begin{split}
    \langle \alpha_{f},\alpha_{g} \rangle 
    &\,=\, 4W \int_{0}^{2\pi} \frac{\sin^{2}(W\cos t)}
    {(W\cos t)^2} e^{-\rmi d\sin t} \dt
    \,=\, 8 \sqrt{2\pi} \frac{W}{\sqrt{d}} \cos\Bigl(d-\frac{\pi}{4}\Bigr)
    + O(d^{-\frac{3}{2}}) \,,
  \end{split}
\end{equation*}
which shows that for $d \gg W \gg 1$
\begin{equation*}
  \frac{\langle\alpha_{f},\alpha_{g}\rangle}
  {\|\alpha_{f}\|_2\|\alpha_{g}\|_2}
  \,\approx\, \sqrt{\frac{2}{\pi}} \frac{W}{\sqrt{d}}
  \cos\Bigl(d-\frac{\pi}{4}\Bigr) \,,
\end{equation*}
which decays no faster than that predicted by
theorem~\ref{th:L2SplitLS}.

\section{Numerical examples}
\label{sec:NumericalExamples}
Next we consider the numerical implementation of the \l2 approach from
section~\ref{sec:L2Corollaries} and the \l1 approach from
section~\ref{sec:L1Corollaries} for far field splitting and data
completion simultaneously
(cf.~theorem~\ref{th:L2CompleteAndSplitMultipleLS} and
theorem~\ref{th:L1CompleteAndSplitMultiple}). 
Since both schemes are extensions of corresponding algorithms for far
field splitting as described in \cite{GriHanSyl14} (least squares) and
\cite{GriSyl16} (basis pursuit), we just briefly comment on
modifications that have to be made to include data completion and
refer to \cite{GriHanSyl14,GriSyl16} for further details. 

Given a far field $\alpha = \sum_{i=1}^I T_{c_i}^*\alpha_i$ that is a
superposition of far fields $T_{c_i}^*\alpha_i$ radiated from balls
$B_{R_i}(c_i)$, for some $c_i\in\R^2$ and $R_i>0$, we assume in the
following that we are unable to observe all of $\alpha$ and that a
subset $\Omega\tm S^1$ is unobserved. 
The aim is to recover $\alpha|_\Omega$ from
$\alpha|_{S^1\setminus\Omega}$ and a priori information on the
location of the supports of the individual source components
$B_{R_i}(c_i)$, $i=1,\ldots,I$. 

We first consider the \l2 approach from
section~\ref{sec:L2Corollaries} and write 
$\gamma := \alpha|_{S^1\setminus\Omega}$ for the observed far field
data and $\beta := -\alpha|_\Omega$.
Accordingly, 
\begin{equation*}
  \gamma \,=\,  \beta + \sum_{i=1}^I T_{c_i}^*\alpha_i \,,
\end{equation*}
i.e., we are in the setting of 
theorem~\ref{th:L2CompleteAndSplitMultipleLS}.
Using the shorthand $V_\Omega:=L^2(\Omega)$ and
$V_i:=T_{c_i}^*\l2(-N_i,N_i)$, $i=1,\ldots,I$, the least squares
problem \eqref{eq:L2CompleteAndSplitMultipleLS-LSP} is equivalent to
seeking approximations $\betatilde\in V_\Omega$ and
$\alphatilde_i\in\l2(-N_i,N_i)$, $i=1,\ldots,I$, satisfying the
Galerkin condition 
\begin{equation}
  \label{eq:NumEx2}
  \langle\betatilde+T_{c_1}^*\alphatilde_1+\cdots
  +T_{c_I}^*\alphatilde_I,\phi\rangle 
  \,=\, \langle\gamma,\phi\rangle \qquad \text{for all } \phi\in
  V_\Omega \oplus V_1 \oplus \cdots \oplus V_I \,.
\end{equation}
The size of the individual subspaces depends on the a priori
information on $R_1,\ldots,R_I$.
Following our discussion at the end of section~\ref{sec:RegPicard} we
choose $N_j = \frac{e}{2}kR_j$ in our numerical example below.
Denoting by $P_\Omega$ and $P_1,\ldots,P_I$ the orthogonal projections
onto $V_\Omega$ and $V_1,\ldots,V_I$, respectively, \eqref{eq:NumEx2}
is equivalent to the linear system 
\begin{equation}
  \label{eq:NumEx3}
  \begin{split}
    \betatilde + P_{\Omega}P_1T_{c_1}^*\alphatilde_1 + \cdots 
    + P_{\Omega}P_IT_{c_I}^*\alphatilde_I
    &\,=\, 0 \,,\\
    P_1P_\Omega\betatilde + T_{c_1}^*\alphatilde_1 + \cdots 
    + P_1P_IT_{c_I}^*\alphatilde_I
    &\,=\, P_1\gamma \,,\\
    & \phantom{x} \,\vdots \\[1ex]
    P_IP_\Omega\betatilde + P_IP_1T_{c_1}^*\alphatilde_1 + \cdots 
    + T_{c_I}^*\alphatilde_I
    &\,=\, P_I\gamma \,.
  \end{split}
\end{equation}
Explicit matrix representations of the individual matrix blocks in
\eqref{eq:NumEx3} follow directly from
\eqref{eq:TcConvolution}--\eqref{eq:DefTcAbuse} (see
\cite[lemma~3.3]{GriHanSyl14} for details) for $P_1,\ldots,P_I$ and
by applying a discrete Fourier transform to the characteristic
function on $S^1\setminus\Omega$ for $P_\Omega$.
Accordingly, the block matrix corresponding to the entire linear system
can be assembled, and the linear system can be solved directly. 
The estimates from theorem~\ref{th:L2CompleteAndSplitMultipleLS} give
bounds on the absolute condition number of the system matrix.

The main advantage of the \l1 approach from
section~\ref{sec:L1Corollaries} is that no a priori information on the
radii $R_i$ of the balls $B_{R_i}(c_i)$, $i=1,\ldots,I$, containing the
individual source components is required. 
However, we still assume that a priori knowledge of the centers
$c_1,\ldots,c_I$ of such balls is available.
Using the orthogonal projection $P_\Omega$ onto $L^2(\Omega)$, the
basis pursuit formulation from
theorem~\ref{th:L1CompleteAndSplitMultiple} can be rewritten as 
\begin{equation}
  \label{eq:NumExa4}
  (\alphatilde_1,\ldots,\alphatilde_I) 
  =\Argmin \sum_{i=1}^I \|\alpha_i\|_{\l1}
  \quad \text{s.t.} \quad 
  {\|\gamma-P_\Omega(\sum_{i=1}^IT_{c_i}^*\alpha_i)\|_2\le\delta}\,,\;
  \alpha_i\in\L2(S^1) \,.
\end{equation}
Accordingly, $\betatilde :=
\sum_{i=1}^I(T_{c_i}^*\alphatilde_i)|_{\Omega}$ is an 
approximation of the missing data segment.
It is well known that the minimization problem from \eqref{eq:NumExa4}
is equivalent to minimizing the Tikhonov functional
\begin{equation}
  \label{eq:NumExa5} 
  \Psi_{\mu}(\alpha_1,\ldots,\alpha_I) 
  \,=\, \|\gamma-P_{\Omega}(\sum_{i=1}^IT_{c_i}^*\alpha_i)\|_{\ell^2}^2 
  + \mu \sum_{i=1}^I \|\alpha_i\|_{\ell^1} \,, \qquad 
  [\alpha_1,\ldots,\alpha_m] \in \ell^2\times\cdots\times\ell^2 \,,
\end{equation} 
for a suitably chosen regularization parameter $\mu>0$ 
(see, e.g., \cite[proposition~2.2]{GraHalSch11}).
The unique minimizer of this functional can be approximated using
(fast) iterative soft thresholding (cf.~\cite{BecTeb09,DauDefDeM04}). 
Apart from the projection $P_\Omega$, which can be implemented
straightforwardly, our numerical implementation analogously to the
implementation for the splitting problem described in \cite{GriSyl16},
and also the convergence analysis from \cite{GriSyl16} carries
over.\footnote{In \cite{GriSyl16} we used additional weights in the
  \l1 minimization problem to ensure that its solution indeed gives
  the exact far field split.  Here we don't use these weights, but our
  estimates from section~\ref{sec:L1Corollaries} imply that the solution
  of \eqref{eq:NumExa4} and \eqref{eq:NumExa5} is very close to the true
  split.}

\begin{example}
  We consider a scattering problem with three obstacles
  as shown in figure~\ref{fig:NumExa1} (left), which are illuminated by
  a plane wave $u^i(x) = e^{\rmi k x\cdot d}$, $x\in\R$, with incident
  direction $d=(1,0)$ and wave number $k=1$ (i.e., the wave length is
  $\lambda =2\pi \approx 6.28$).
  \begin{figure}
    \centering
    \includegraphics[height=5.0cm]{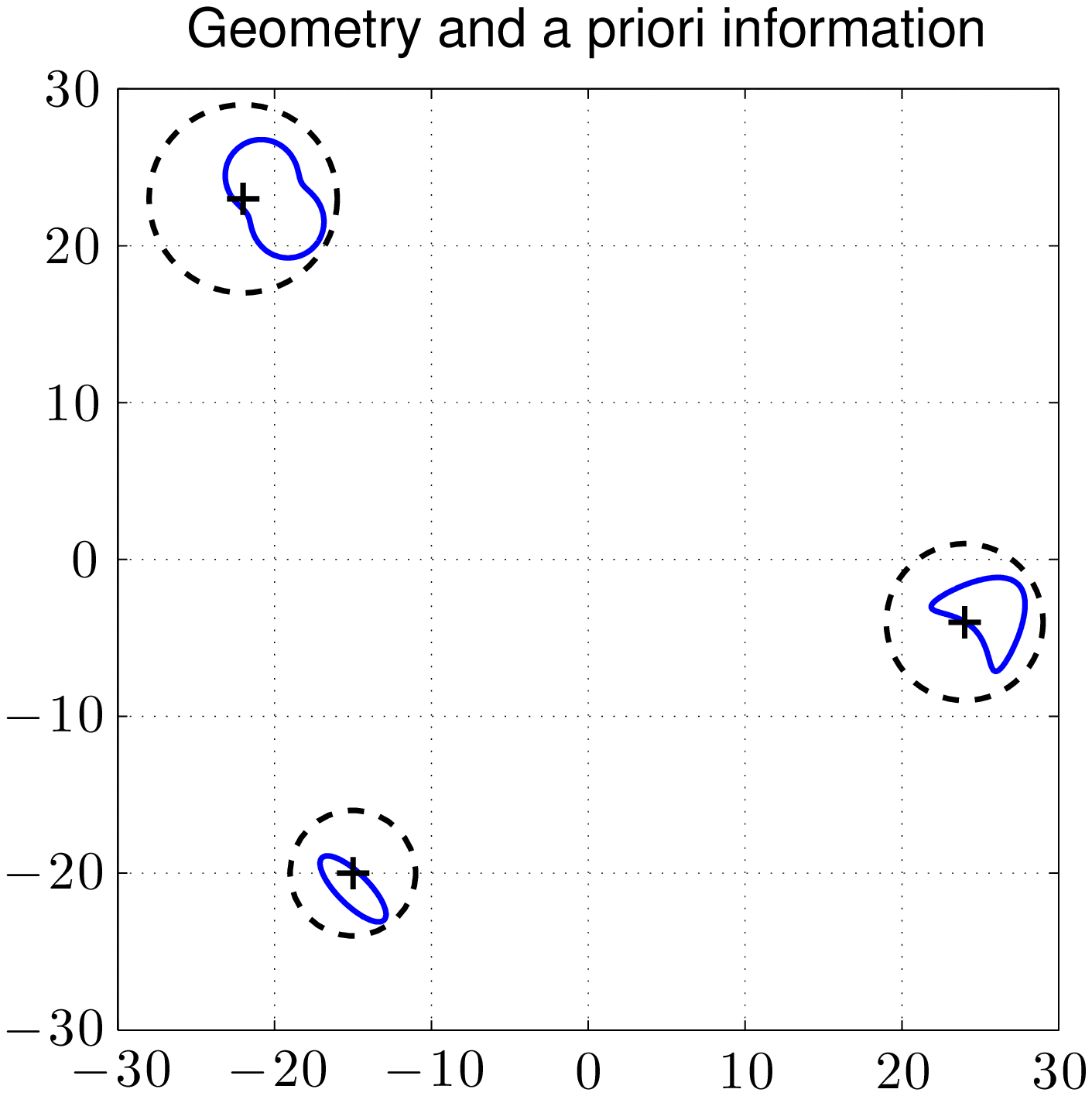}
    \qquad
    \includegraphics[height=5.0cm]{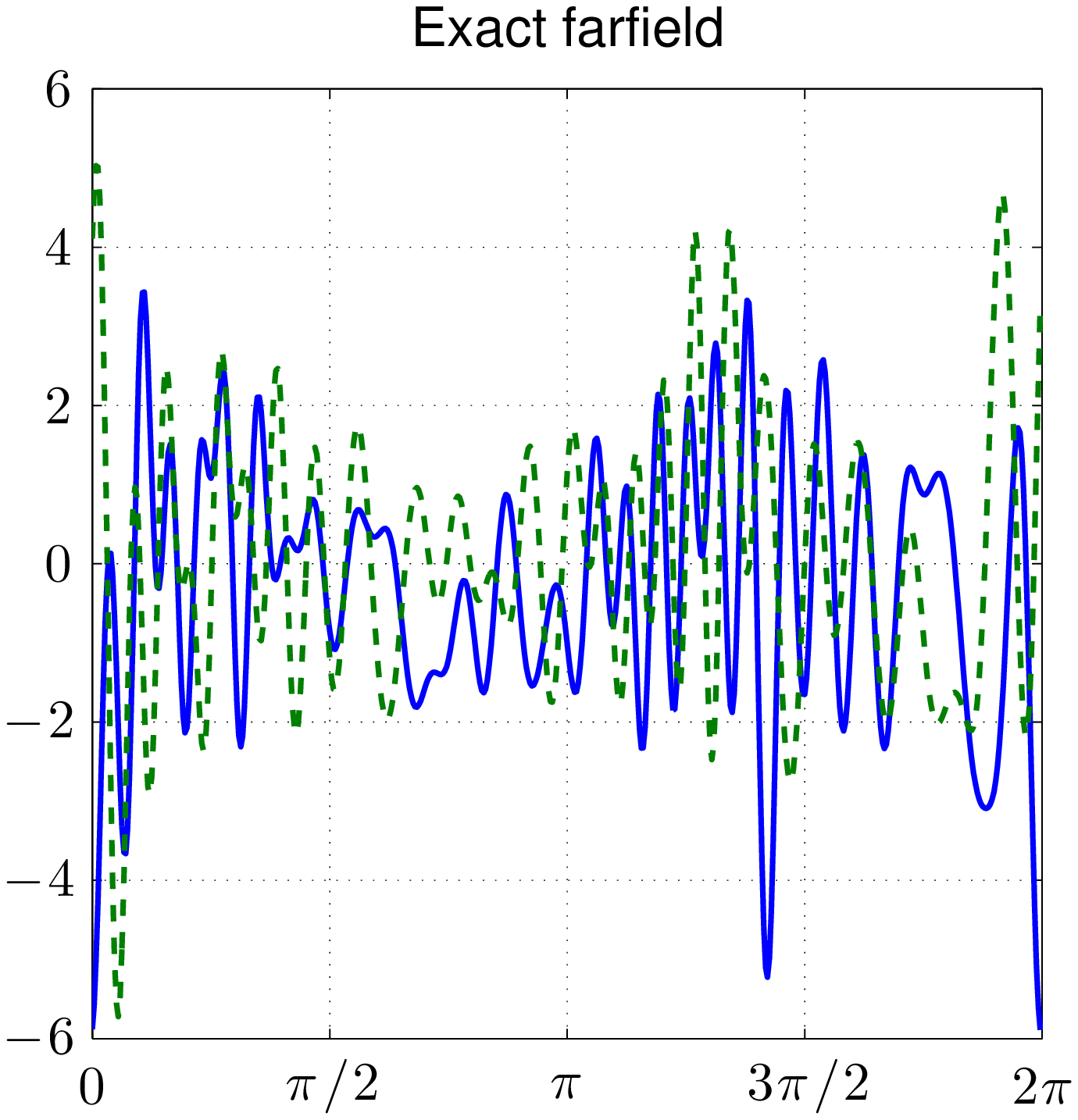}
    \caption{\small Left: Geometry of the scatterers (solid) and a
      priori information on the source locations (dashed).
      Right: Real part (solid) and imaginary part (dashed) of the far
      field $\alpha$.} 
    \label{fig:NumExa1}
  \end{figure}
  Assuming that the ellipse is sound soft whereas the kite and the nut
  are sound hard, the scattered field $u^s$ satisfies the homogeneous
  Helmholtz equation outside the obstacles, the Sommerfeld radiation
  condition at infinity, and Dirichlet (ellipse) or Neumann boundary
  conditions (kite and nut) on the boundaries of the obstacles. 
  We simulate the corresponding far field $\alpha$ of $u^s$ on an
  equidistant grid with $512$ points on the unit sphere $S^1$ using a
  Nystr\"om method (cf.\ \cite{ColKre98,Kre95}).
  Figure~\ref{fig:NumExa1} (middle) shows the real part (solid line) and
  the imaginary part (dashed line) of $\alpha$. 
  Since the far field $\alpha$ can be written as a superposition 
  of three far fields radiated by three individual smooth sources
  supported in arbitrarily small neighborhoods of the scattering
  obstacles (cf., e.g., \cite[lemma~3.6]{KusSyl05}), this example fits
  into the framework of the previous sections. 

  We assume that the far field cannot be measured on the segment
  \begin{equation*}
    \Omega 
    \,=\, \{ \theta = (\cos t, \sin t) \in S^1 \;|\; 
    \pi/2 < t< \pi/2+\pi/3 \} \,,
  \end{equation*}
  i.e., $|\Omega| = \pi/3$.
  We first apply the least squares procedure and use the dashed circles
  shown in figure~\ref{fig:NumExa1} (left) as a priori information on
  the approximate source locations $B_{R_i}(c_i)$, $i=1,2,3$.
  More precisely, $c_1 = (24,-4)$, $c_2 = (-22,23)$, $c_3=(-15,-20)$ and
  $R_1=5$, $R_2=6$ and $R_3=4$.
  Accordingly we choose $N_1=7$, $N_2=9$ and $N_3=6$, and solve the
  linear system \eqref{eq:NumEx3}. 

  \begin{figure}
    \centering
    \includegraphics[height=5.0cm]{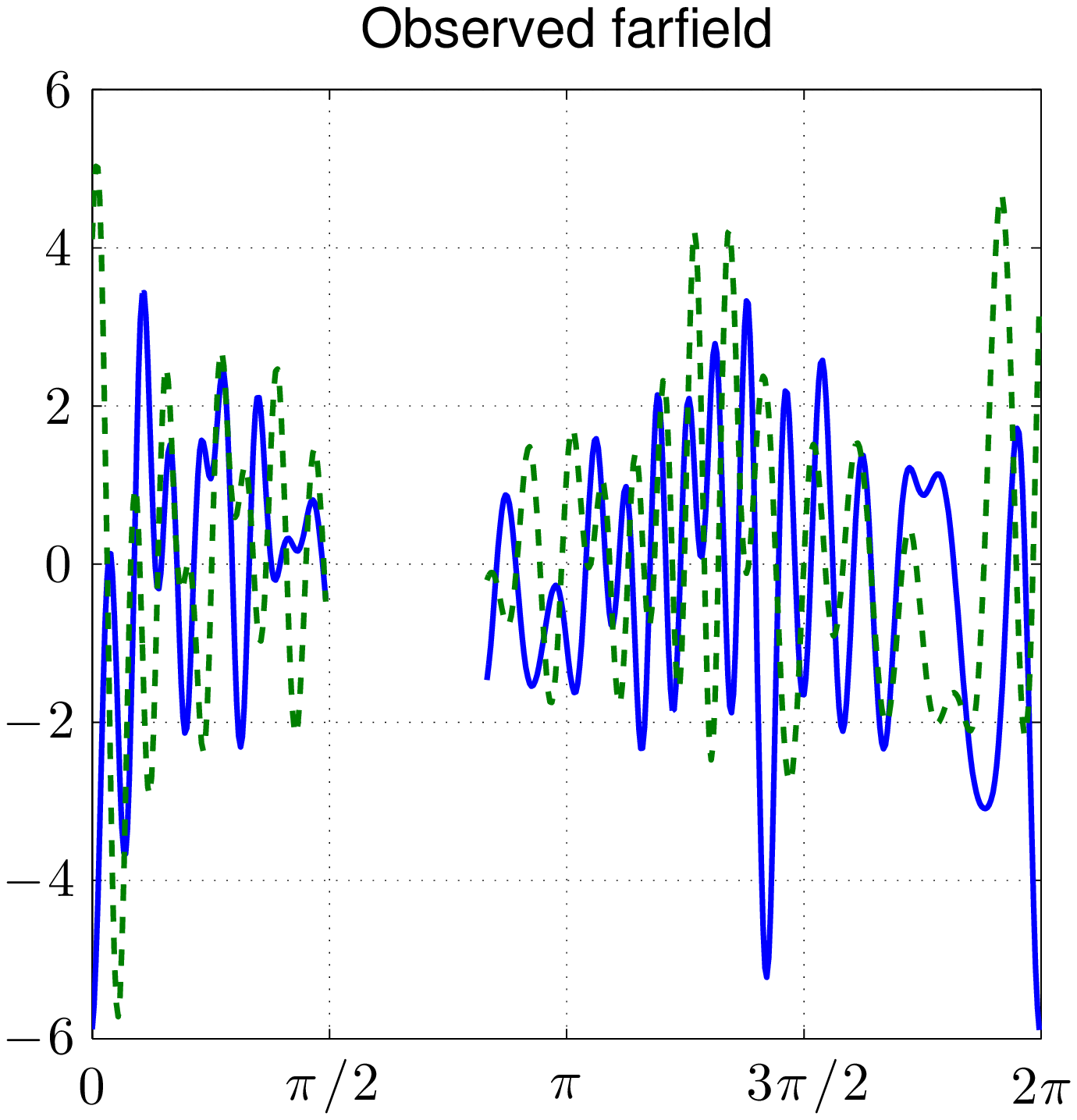}
    \quad
    \includegraphics[height=5.0cm]{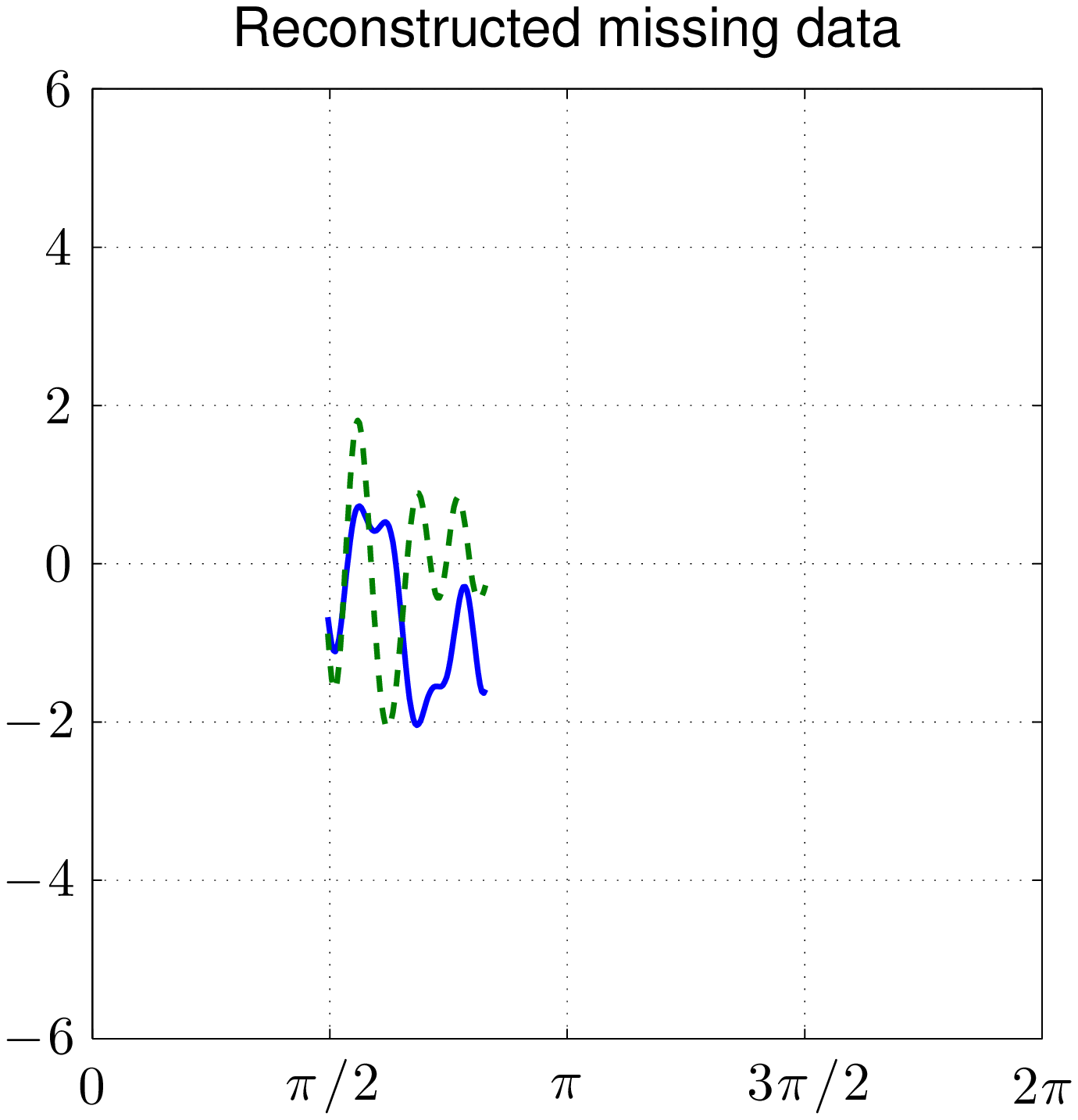}
    \quad
    \includegraphics[height=5.0cm]{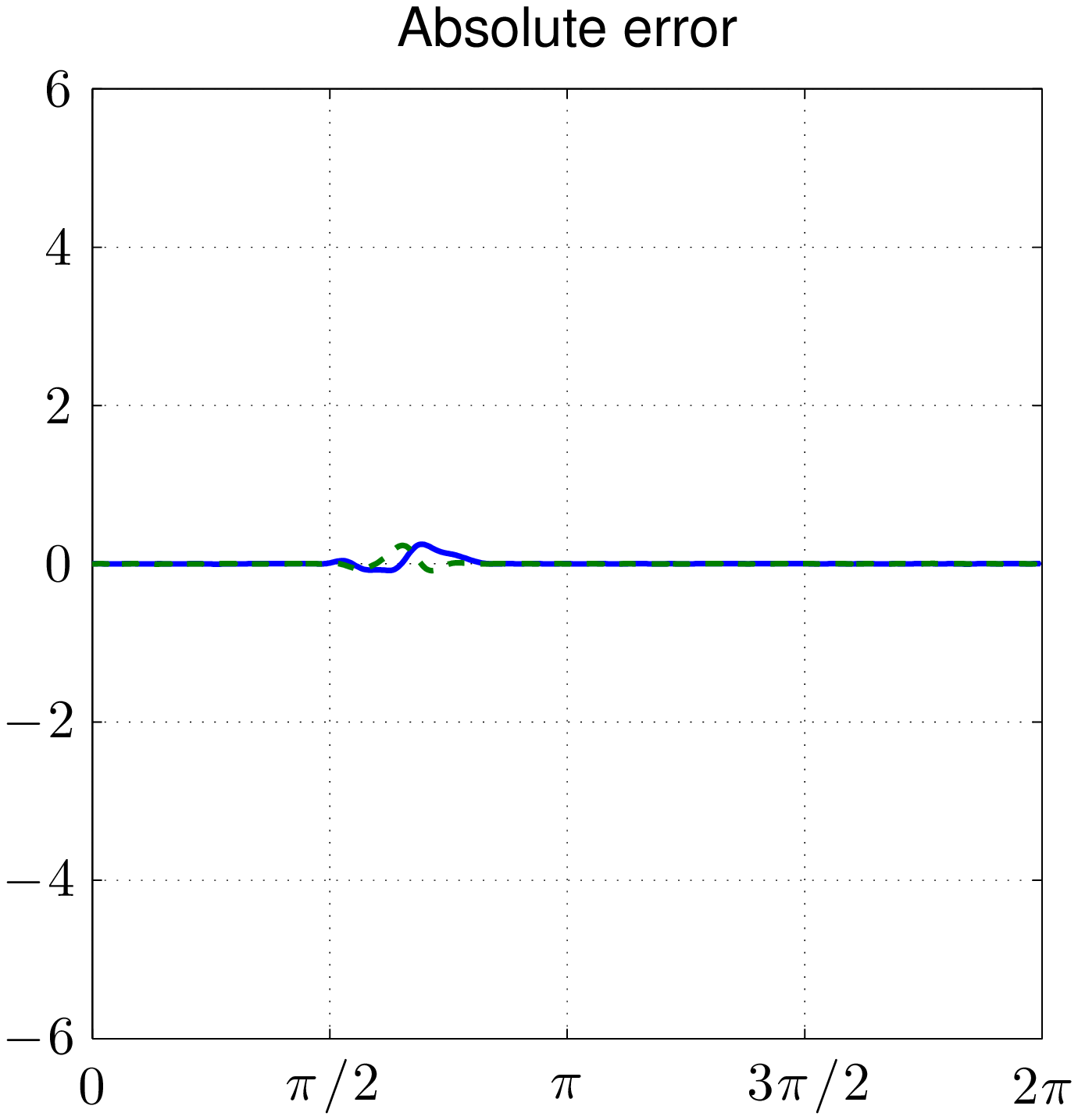}
    \caption{\small Reconstruction of the least squares scheme: 
      Observed far field $\gamma$ (left), reconstruction of the missing
      part $\alpha|_\Omega$ (middle), and difference between exact far
      field and reconstructed far field (right).} 
    \label{fig:NumExa2}
  \end{figure}
  Figure~\ref{fig:NumExa2} shows a plot of the observed data $\gamma$
  (left), of the reconstruction of the missing data segment obtained by
  the least squares algorithm and of the difference between the exact
  far field and the reconstructed far field.
  Again the solid line corresponds to the real part while the dashed
  line corresponds to the imaginary part.  
  The condition number of the matrix is $5.4 \times 10^4$.
  We note that the missing data component in this example is actually
  too large for the assumptions of
  theorem~\ref{th:L2CompleteAndSplitMultipleLS} to be satisfied.
  Nevertheless the least squares approach still gives good results.

  Applying the (fast) iterative soft shrinkage algorithm to this example
  (with regularization parameter $\mu = 10^{-3}$ in \eqref{eq:NumExa5})
  does not give a useful reconstruction.
  As indicated by the estimates in
  theorem~\ref{th:L1CompleteAndSplitMultiple} the \l1 approach seems to
  be a bit less stable.
  Hence we halve the missing data segment, consider in the following 
  \begin{equation*}
    \Omega 
    \,=\, \{ \theta = (\cos t, \sin t) \in S^1 \;|\; 
    \pi/2 < t< \pi/2+\pi/6 \} \,,
  \end{equation*}
  i.e., $|\Omega| = \pi/6$, and apply the \l1 reconstruction scheme to
  this data.
  \begin{figure}
    \centering
    \includegraphics[height=5.0cm]{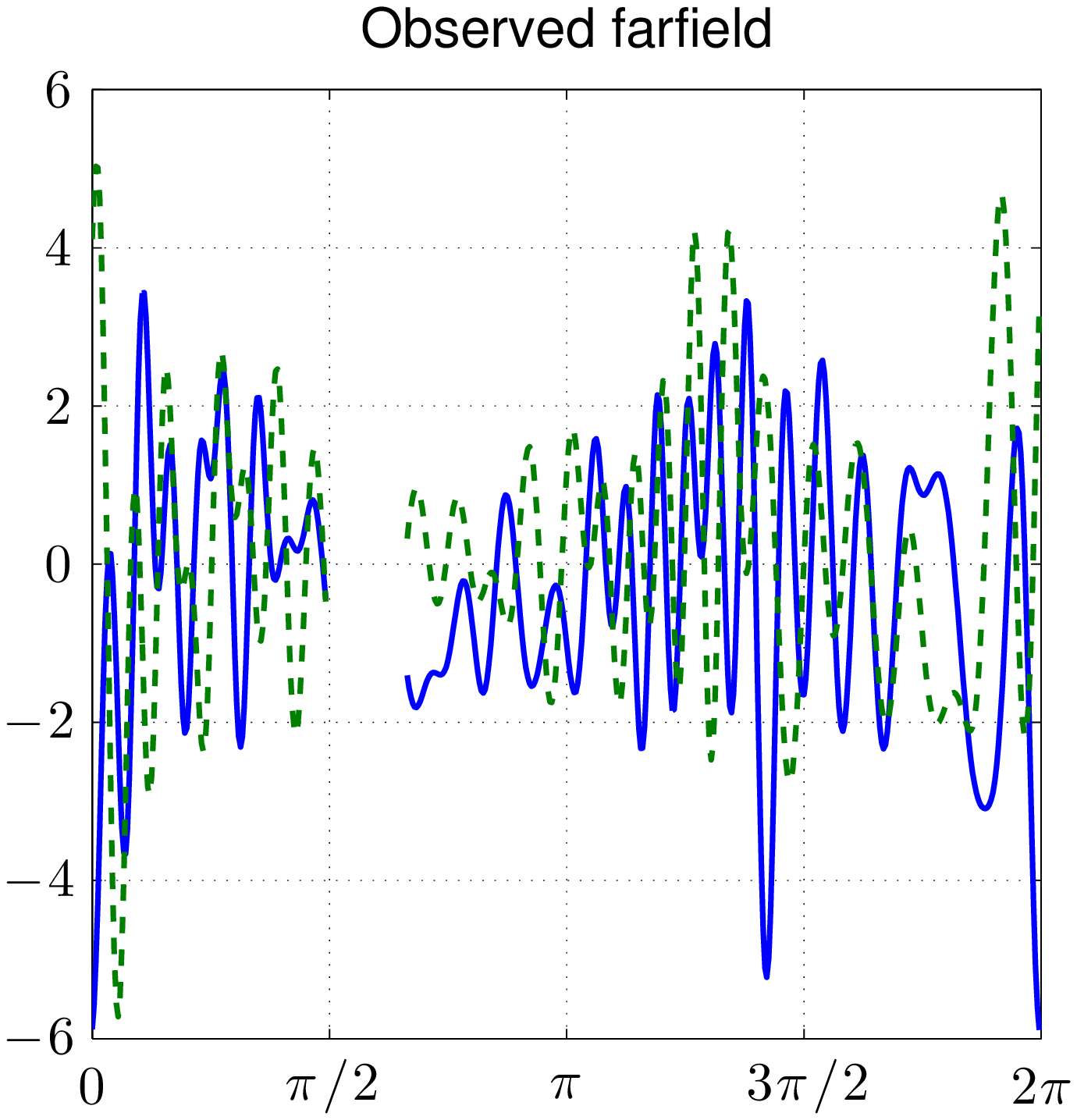}
    \quad
    \includegraphics[height=5.0cm]{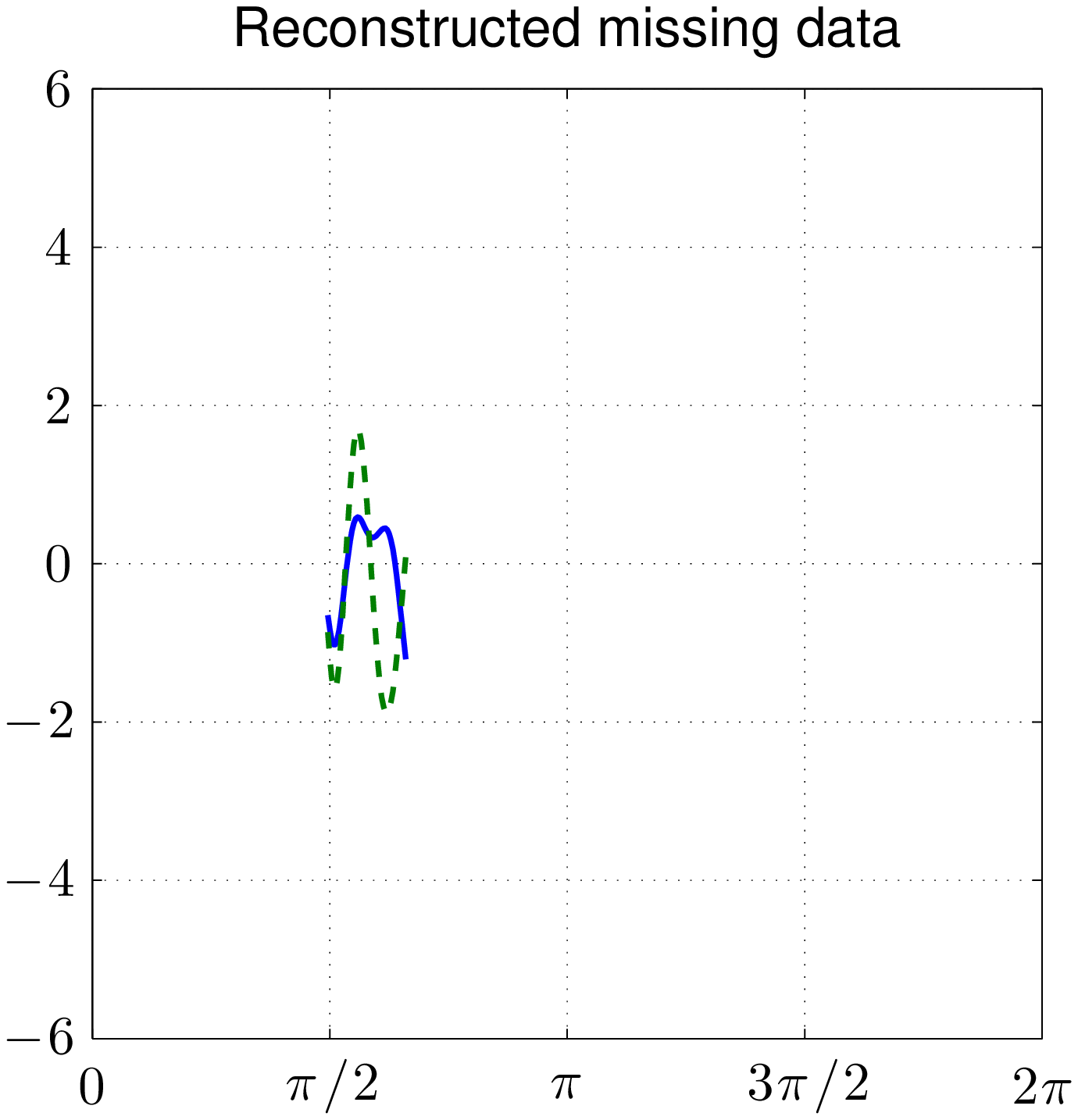}
    \quad
    \includegraphics[height=5.0cm]{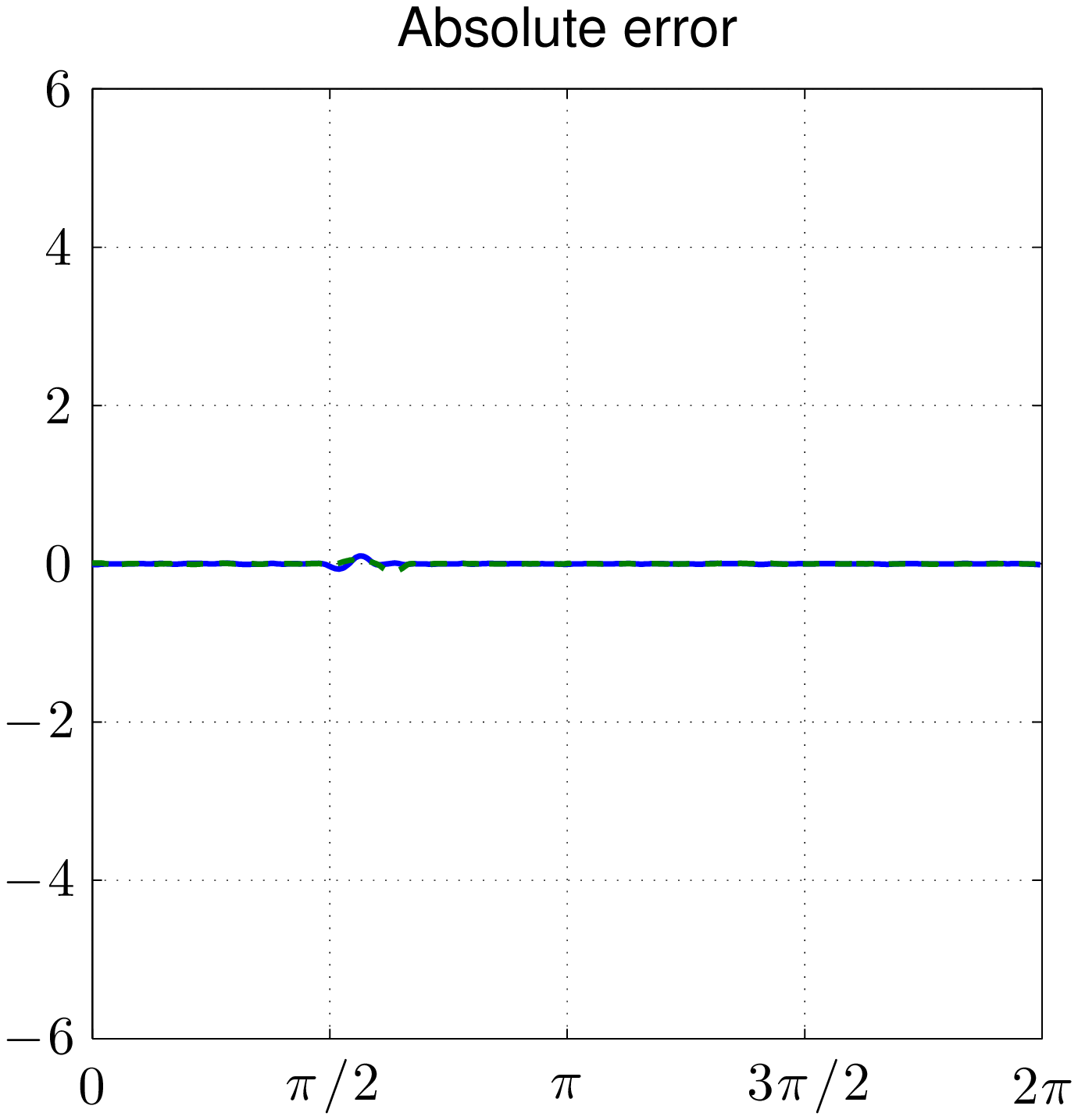}
    \caption{\small Reconstruction of the basis pursuit scheme: 
      Observed far field $\gamma$ (left), reconstruction of the missing
      part $\alpha|_\Omega$ (middle), and difference between exact far
      field and reconstructed far field (right).} 
    \label{fig:NumExa3}
  \end{figure}
  Figure~\ref{fig:NumExa3} shows a plot of the observed data $\gamma$
  (left), of the reconstruction of the missing data segment obtained by
  the fast iterative soft shrinkage algorithm (with $\mu = 10^{-3}$)
  after $10^3$ iterations (the initial guess is zero) and of the
  difference between the exact far field and the reconstructed far
  field. 

  The behavior of both algorithms in the presence of noise in the data
  depends crucially on the geometrical setup of the problem (i.e.\ on
  its conditioning). 
  The smaller the missing data segment is and the smaller the
  dimensions of the individual source components are relative to their
  distances, the more noise these algorithms can handle. 
  \hfill$\lozenge$
\end{example}

\section*{Conclusions}
We have considered the source problem for the two-dimensional
Helmholtz equation when the source is a superposition of finitely many
well-separated compactly supported source components.
We have presented stability estimates for numerical algorithms to
split the far field radiated by this source into the far fields
corresponding to the individual source components and to restore
missing data segments.
Analytic and numerical examples confirm the sharpness of these
estimates and illustrate the potential and limitations of the
numerical schemes. 

The most significant observations are:
\begin{itemize}
\item The conditioning of far field splitting and data completion
  depends on the dimensions of the source components, their relative
  distances with respect to wavelength and the size of the missing
  data segment.
  The results clearly suggest combining data completion with
  splitting whenever possible in order to improve the conditioning of
  the data completion problem.
\item The conditioning of far field splitting and data completion
  depends on wave length and deteriorates with increasing wave number.
  Therefore, in order to increase resolution one not only has to
  increase the wave number but also the dynamic range of the sensors
  used to measure the far field data.
\end{itemize}

\begin{appendix}
  \section*{Appendix}
  \section{Some properties of the squared singular values $\{s_n^2(R)\}$ 
    of $\Fcal_{B_R(0)}$}
  \label{app:sn2}
  In the following we collect some interesting properties of the squared
  singular values $\{s_n^2(R)\}$, as introduced in \eqref{eq:Defsn}, of
  the restricted Fourier transform $\Fcal_{B_R(0)}$ from
  \eqref{eq:DefRestFarfieldOperator}. 

  We first note that \cite[10.22.5]{OLBC10} implies that the squared
  singular values from \eqref{eq:Defsn} satisfy
  \begin{equation*}
    s_n^2(R) 
    \,=\, \pi R^2 \bigl(J_n^2(R)-J_{n-1}(R)J_{n+1}(R)\bigr) \,,
    \qquad n\in\Z \,,
  \end{equation*}
  and simple manipulations using recurrence relations for Bessel
  functions
  \begin{subequations}
    \label{eq:RecurrenceBessel}
    \begin{align}
      \frac{nJ_n(r)}r &\,=\, \frac{J_{n+1}(r) + J_{n-1}(r)}2 \,,
      &&n\in\Z \,,\\
      J_n'(r) &\,=\, \frac{J_{n-1}-J_{n+1}(r)}2 \,, 
      &&n\in\Z \,,
    \end{align}
  \end{subequations}
  (cf., e.g., \cite[10.6(i)]{OLBC10}) show that, for $n\in\Z$
  \begin{equation}
    \label{eq:sn2Explicit2}
    \begin{split}
      s_n^2(R) 
      &\,=\, \pi \bigl( n^2J_n^2(R) - R^2J_{n-1}(R)J_{n+1}(R) 
      + R^2J_n(R) - n^2J_n^2(R)\bigr)\\
      &\,=\, \pi \Bigl( \frac{R^2}{4}(J_{n-1}(R)+J_{n+1}(R))^2 
      - R^2J_{n-1}(R)J_{n+1}(R) + (R^2-n^2)J_n^2(R)\Bigr)\\
      &\,=\, \pi \Bigl( R^2 \Bigl(\frac{J_{n-1}(R)-J_{n+1}(R)}{2}\Bigr)^2
      + (R^2-n^2)J_n^2(R)\Bigr)\\
      &\,=\, \pi \bigl( (R J_n'(R))^2 + (R^2-n^2)J_n^2(R) \bigr) \,.
    \end{split}
  \end{equation}

  \begin{lemma}
    \label{lmm:Sumsn2}
    \begin{equation*}
      \sum_{n=-\infty}^\infty s_n^2(R)
      \,=\, \pi  R^2 \,.
    \end{equation*}
  \end{lemma}

  \begin{proof}
    Since
    \begin{equation*}
      \sum_{n=-\infty}^\infty J_n^2(r) \,=\, 1
    \end{equation*}
    (cf.\ \cite[10.23.3]{OLBC10}), the definition \eqref{eq:Defsn}
    yields
    \begin{equation*}
      \sum_{n=-\infty}^\infty s_n^2(R)
      \,=\, 2\pi\int_0^R \Bigl(\sum_{n=-\infty}^\infty J_n^2(r)\Bigr)r\dr
      \,=\, 2\pi\int_0^R r\dr
      \,=\, \pi R^2 \,.
    \end{equation*}
  \end{proof}

  The next lemma shows that odd and even squared singular values,
  $s_n^2(R)$, are monotonically decreasing for $n\geq0$ and monotonically
  increasing for $n\leq0$. 
  \begin{lemma}
    \label{lmm:Monotonicitysn2}
    \begin{equation*}
      s_{n-1}^2(R) - s_{n+1}^2(R) \,\geq\, 0 \,, \qquad n\geq 0 \,.
    \end{equation*}
  \end{lemma}

  \begin{proof}
    Using the recurrence relations \eqref{eq:RecurrenceBessel} we find
    that 
    \begin{equation*}
      J_{n-1}^2(r) - J_{n+1}^2(r) 
      \,=\, \frac{4n}r J_n(r) J_n'(r)
      \,=\, \frac{2n}r (J_n^2)'(r) \,.
    \end{equation*}
    Thus,
    \begin{equation*}
      \begin{split}
        s_{n-1}^2(R) - s_{n+1}^2(R) 
        &\,=\, 2\pi \Bigl( \int_0^R J_{n-1}^2(r) r \dr 
        - \int_0^R J_{n+1}^2(r) r \dr \Bigr)\\
        &\,=\, 2\pi \int_0^R 2n (J_n^2)'(r) \dr 
        \,=\, 4\pi n J_n^2(R) 
        \,\geq\, 0 \,.
      \end{split}
    \end{equation*}
  \end{proof}

  Integrating sharp estimates for $J_n(r)$ from \cite{Kra14}, we obtain
  upper bounds for $s_n^2(R)$ when $|n|\geq R>0$.
  Since $s_n^2(R) = s_{-n}^2(R)$, $n\in\Z$, it is sufficient to consider
  $n\geq R$.
  \begin{lemma}
    \label{lmm:Decaysn2}
    Suppose that $n \geq R > 0$.
    Then
    \begin{equation*}
      s_n^2(R) 
      \,\leq\, \frac{\pi
        2^{\frac23}n^{\frac23}}{3^{\frac43}\bigl(\Gamma(\frac23)\bigr)^2}
      \Bigl( \frac{n+\frac12}{n} \Bigr)^{n+1}
      \Bigl( \frac{R^2}{n^2} e^{1-\frac{R^2}{n^2}} \Bigr)^n 
      \frac{R^2}{n^2} \,.
    \end{equation*}
  \end{lemma}

  \begin{proof}
    From theorem~2 of \cite{Kra14} we obtain for $0<r\leq n$ that
    \begin{equation*}
      J_{n}^2(r)
      \,\le\, \frac{2^{\frac23}}{3^{\frac43}\bigl(\Gamma(\frac23)\bigr)^2}
      \frac{r^{2n}}{n^{2n+\frac23}}
      e^{\frac{n^{2}-r^{2}}{n+\frac12}} \,.
    \end{equation*}
    Substituting this into \eqref{eq:Defsn} yields
    \begin{equation*}
      \begin{split}
        s_n^2(R)
        &\,\leq\, 2\pi 
        \frac{2^{\frac23}}{3^{\frac43}\bigl(\Gamma(\frac23)\bigr)^2}
        \frac{e^{\frac{n^2}{n+\frac12}}}{n^{2n+\frac23}}
        \int_0^R r^{2n} e^{-\frac{r^2}{n+\frac12}} r\dr\\
        &\,=\, \pi 
        \frac{2^{\frac23}}{3^{\frac43}\bigl(\Gamma(\frac23)\bigr)^2}
        \frac{e^{\frac{n^2}{n+\frac12}}}{n^{2n+\frac23}} \Bigl(n+\frac12\Bigr)^{n+1}
        \int_0^{\frac{R^2}{n+\frac12}} t^{n} e^{-t} \dt\\
        &\,\leq\, \pi 
        \frac{2^{\frac23}}{3^{\frac43}\bigl(\Gamma(\frac23)\bigr)^2}
        \frac{e^{n}}{n^{2n+\frac23}} \Bigl(n+\frac12\Bigr)^{n+1}
        \int_0^{\frac{R^2}{n}} t^{n} e^{-t} \dt \,.
      \end{split}
    \end{equation*}
    Since $t^{n} e^{-t}$ is monotonically increasing for
    $0<t<\frac{R^2}{n}\leq n$, we see
    \begin{equation*}
      \begin{split}
        s_n^2(R)
        &\,=\, \pi 
        \frac{2^{\frac23}}{3^{\frac43}\bigl(\Gamma(\frac23)\bigr)^2}
        \frac{(n+\frac12)^{n+1}}{n^{2n+\frac23}} e^{n}
        \frac{R^2}{n} \frac{R^{2n}}{n^n} e^{-\frac{R^2}{n}}\\
        &\,=\, \pi
        \frac{2^{\frac23}}{3^{\frac43}\bigl(\Gamma(\frac23)\bigr)^2}
        n^{\frac23} \Bigl( \frac{n+\frac12}{n} \Bigr)^{n+1}
        \Bigl( \frac{R^2}{n^2} e^{1-\frac{R^2}{n^2}} \Bigr)^n 
        \frac{R^2}{n^2} \,.
      \end{split}
    \end{equation*}
  \end{proof}

  On the other hand, the squared singular values $s_{n}^{2}(R)$ are not
  small for $|n|<R$. 

  \begin{theorem}
    \label{th:BesselAsy}
    Suppose that $R>n\geq 0$, define $\alpha\in(0,\frac{\pi}{2})$ by 
    $\cos\alpha=\frac{n}{R}$, and therefore 
    $\sin\alpha=\sqrt{1-(\frac{n}{R})^{2}}$, and assume
    $\alpha>\varepsilon>0$. 
    Then
    \begin{align}
      \label{eq:EstBessel1}
      \biggl| J_n(R)
      -\sqrt{\frac{2}{\pi R\sin\alpha}}
      \cos\Bigl(R(\sin\alpha-\alpha\cos\alpha)
      -\frac{\pi}{4}\Bigr)\biggr|
      &\,\le\, \frac{C(\varepsilon)}{R} \,,\\\label{eq:EstBessel2}
      \biggl|J_n'(R)
      +\sqrt{\frac{2}{\pi R\sin\alpha}} \sin\alpha
      \sin\Bigl(R(\sin\alpha-\alpha\cos\alpha)
      -\frac{\pi}{4}\Bigr)\biggr|
      &\,\le\, \frac{C(\varepsilon)}{R} \,.
    \end{align}
    where the constant $C(\eps)$ depends on the lower bound $\eps$ but
    is otherwise independent of $n$ and $R$.
  \end{theorem}

  \begin{proof}
    By definition,
    \begin{equation}
      \label{eq:BesselIntDef}
      J_{n}(R)
      \,=\, \frac{1}{2\pi} 
      \int_{-\pi}^{\pi}e^{\rmi(R\sin{t}-n{t})}\dt
      \,=\, \frac{1}{2\pi}
      \int_{-\pi}^{\pi}e^{\rmi R(\sin{t}-{t}\cos\alpha)}\dt
      \,=\, \frac{1}{2\pi} 
      \int_{-\pi}^{\pi}e^{\rmi R\phi({t})}\dt
    \end{equation}
    with
    \begin{equation*}
      \phi({t}) \,=\, \sin{t}-{t}\cos\alpha \,, \quad
      \phi'({t}) \,=\, \cos{t}-\cos\alpha \,, \quad
      \phi''({t}) \,=\, -\sin{t} \,.
    \end{equation*}

    The phase function $\phi$ has stationary points at $\pm\alpha$, and
    $\phi''$ vanishes at $0$ and $\pi$. 
    We will apply stationary phase in a neighborhood of each stationary
    point. 
    The neighborhood must be small enough to guarantee that 
    $|\phi''({t})|$ is bounded from below there.
    Integration by parts will be used to estimate integral in regions
    where $\phi'$ is bounded below. 
    The hypothesis that  $\alpha>\varepsilon>0$ will guarantee
    that the union of these two regions covers the whole interval
    $(-\pi,\pi)$. 

    To separate the two regions, let 
    $a_{\eps}(\tau)=a\bigl(\frac{\tau}{\eps}\bigr)$, $\tau \in \R$, be a
    $C^{\infty}$ cutoff function satisfying 
    \begin{equation*}
      a_{\eps}(\tau) 
      \,=\,
      \begin{cases}
        1 & \text{if } |\tau|>2\eps\\
        0 & \text{if } |\tau|<\eps
      \end{cases} 
      \qquad \text{and} \qquad
      |a_{\eps}^{(j)}(\tau)|
      \,\le\,  \frac{C_{j}}{\eps^{j}}
    \end{equation*}
    with the $C_{j}>0$ independent of $\eps>0$. 
    Define $A_{\eps}({t}) := a_{\eps}(\phi'({t})) $, 
    ${t}\in(-\pi,\pi)$, then
    \begin{equation}
      \label{eq:StatPhase0}
      A_{\eps}({t}) 
      \,=\, \begin{cases}
        1 & \text{if } |\phi'({t})| > 2\eps\,,\\
        0 & \text{if } |\phi'({t})| < \eps\,. 
      \end{cases}
    \end{equation}

    Theorem~7.7.1 of \cite{Hor03} tells us that for any integer 
    $k\geq 0$
    \begin{equation}
      \label{eq:StatPhase1}
      \biggl|\int_{-\pi}^{\pi} e^{\rmi R\phi({t})}
      A_{\eps}({t}) \dt \biggr|
      \,\le\,
      \frac{C}{R^{k}}
      \sum_{j=0}^{k} \sup_{{t}\in\supp A_{\eps}}
      \biggl|\frac{A_{\eps}^{(j)}({t})}{(\phi'({t}))^{2k-j}}\biggr|
      \,\le\,
      \frac{C}{R^{k}\eps^{2k}} \sum_{j=0}^{k}C_{j} 
    \end{equation}
    with $C$ only depending on an upper bound for the higher order
    derivatives of $\phi$.
    For the second inequality we have used \eqref{eq:StatPhase0} and the
    fact that all higher derivatives of $\phi$ are bounded by 1.

    We will estimate the remainder of the integral using Theorem~7.7.5
    of \cite{Hor03}, which tells us that, if  $t_0$ is the unique
    stationary point of  $\phi$ in the support of a smooth function $B$,
    and $|\phi''({t})|>\delta>0$ on the support of  $B$, then  
    \begin{equation}
      \label{eq:StatPhase2}
      \biggl|\int_{-\pi}^\pi e^{\rmi R\phi({t})} B({t}) \dt
      - \sqrt{\frac{2\pi \rmi}{R\phi''(t_0)}} 
      e^{\rmi R\phi(t_0)} B(t_0)\biggr|
      \,\le\,
      \frac{C}{R} \sum_{j=0}^{2}
      \sup_{{t}\in\supp B}|B^{(j)}({t})|
    \end{equation}
    with $C>0$ depending only on the  lower bound $\delta$ for
    $|\phi''({t})|$ and an upper bound for higher derivatives of
    $\phi$ on the support of $B$. 
    We will set  $B=1-A_{\eps}({t})$, which is supported in  two
    intervals, one containing  $\alpha$ and the other containing
    $-\alpha$, so \eqref{eq:StatPhase2}, becomes
    \begin{equation}
      \label{eq:StatPhase3}
      \biggl|\int_{-\pi}^{\pi} e^{\rmi R\phi({t})}
      (1-A_{\eps}({t})) \dt
      - \sqrt{\frac{2\pi}{R\sin\alpha}}
      2 \cos\Bigr(R(\sin\alpha-\alpha\cos\alpha)
      -\frac{\pi}{4}\Bigr)\biggr|
      \,\le\,
      \frac{C}{R}\Sum_{j=0}^{2}\frac{C_{j}}{\eps^{j}}
    \end{equation}
    as long as $\eps$ is chosen so that
    $|\phi''({t})|=|\sin({t})|\ge\delta$ on the support of
    $1-A_{\eps}({t})$.

    The following lemma suggests a proper choice of $\eps$.
    \renewcommand{\qedsymbol}{}
  \end{proof}

  \begin{lemma}
    \label{lmm:StatPhase}
    Let ${t}$ and $\alpha$ belong to $\bigl[0,\frac{\pi}{2}\bigr]$ 
    then  
    \begin{equation*}
      |\cos{t}-\cos\alpha| 
      \,<\, \eps
    \end{equation*}
    implies that
    \begin{equation*}
      \sin{t}
      \,>\, \sin\alpha - \frac{2\eps}{\sin\alpha} \,.
    \end{equation*}
  \end{lemma}

  \begin{proof}
    Since
    \begin{equation*}
      \sin^{2}{t}-\sin^{2}\alpha 
      \,=\, \cos^{2}\alpha - \cos^{2}{t}
    \end{equation*}
    we deduce
    \begin{equation*}
      \sin{t}-\sin\alpha 
      \,=\, \frac{\cos\alpha+\cos{t}}{\sin{t}+\sin\alpha}
      (\cos\alpha-\cos{t}) \,.
    \end{equation*}
    Consequently
    \begin{equation*}
      |\sin{t}-\sin\alpha|
      \,\le\, \frac{2}{\sin\alpha}\eps \,.
    \end{equation*}
  \end{proof}

  \begin{proof}[End of proof of theorem~\ref{th:BesselAsy}]
    We choose $\eps=\frac{\sin^{2}\alpha}{4}$ and assume that
    $|\phi'(t)| = |\cos t - \cos\alpha| < \eps$, then
    lemma~\ref{lmm:StatPhase} gives 
    $\sin{t}>\frac{\sin\alpha}{2}$. 
    We use this value of $\eps$ in \eqref{eq:StatPhase0}, i.e.
    \begin{equation*}
      |\phi'({t})| 
      \,>\, 2\eps \quad \text{on } \supp A_\eps  
      \qquad \text{and} \qquad
      |\phi'({t})| 
      \,<\, \eps \quad \text{on } \supp (1-A_\eps) \,.
    \end{equation*}
    Accordingly,
    \begin{equation*}
      |\phi''({t})| 
      \,=\, |\sin{t}| 
      \,>\, \frac{\sin\alpha}{2} 
      \,=\, \sqrt{\eps} \,=:\, \delta \qquad
      \text{on } \supp (1-A_\eps)  \,.
    \end{equation*}
    Now, adding \eqref{eq:StatPhase1} and \eqref{eq:StatPhase3}
    establishes \eqref{eq:EstBessel1}. 

    The calculation for  \eqref{eq:EstBessel2} is analogous with
    \eqref{eq:BesselIntDef} replaced by 
    \begin{equation*}
      J_{n}'(R)
      \,=\,
      \frac{1}{2\pi}\int_{-\pi}^{\pi} \rmi\sin({t}) 
      e^{\rmi(R\sin{t}-n{t})} \dt 
      \,=\,
      \frac{1}{2\pi}\int_{-\pi}^{\pi} \rmi\sin({t}) 
      e^{\rmi R\phi({t})} \dt \,,
    \end{equation*}
    which has the same phase and hence the same stationary
    points. The only difference is that the term  $B(t_0)$
    in \eqref{eq:StatPhase2} at  $t_0=\pm\alpha$ will be
    $\pm \rmi\sin\alpha$ rather than 1.
  \end{proof}

  We now combine (\ref{eq:EstBessel1}) and (\ref{eq:EstBessel2}) with
  the equality \eqref{eq:sn2Explicit2}
  to obtain, for $\cos\alpha = \frac{n}{R}<1-\eps$ 
  \begin{equation*}
    \bigl| s_{n}^{2}(R)-2R\sin\alpha\bigr|
    \,=\, \bigl|s_{n}^{2}(R)-2\sqrt{R^{2}-n^{2}}\bigr|
    \,\le\, C(\eps)\sqrt{R} \,.
  \end{equation*}
  Since equation \eqref{eq:BesselIntDef} is only a valid definition 
  of the Bessel function  $J_{n}(R)$ when $n$ is an integer\footnote{The
    definition requires a contour integral when  $\nu$ is not an
    integer.}, we denote in the following by $\lceil \nu R\rceil$ is
  smallest integer that is greater than or equal to $\nu R$, so that we
  can state a convergence result. 

  \begin{corollary}
    \begin{equation*}
      \lim_{R\rightarrow\infty} \frac{s_{\lceil\nu R\rceil}^{2}(R)}{2R}
      \,=\,
      \begin{cases}
        \sqrt{1-\nu^{2}} & \nu \le 1
        \\
        0 & \nu \ge 1
      \end{cases}
    \end{equation*}
  \end{corollary}

  \section{Proof of estimate \eqref{eq:ProofUncertainty3-1}}
  \label{app:ProofKrasikov}
  Let $n\in\Z$ and $\mu := (2n+1)(2n+3)$.
  Theorem~2 of \cite{Kra06} establishes that for $r > \frac{\sqrt{\mu+\mu^{\frac32}}}{2}$,
  \begin{equation}
    \label{eq:KrasikovThm2}
    J_n^2(r) 
    \,\leq\, \frac{4(4r^2-(2n+1)(2n+5))}{\pi((4r^2-\mu)^{\frac32}-\mu)} \,.
  \end{equation}
  The following lemma shows that, under the assumptions of
  theorem~\ref{th:Uncertainty3}, the estimate \eqref{eq:KrasikovThm2} implies
  \eqref{eq:ProofUncertainty3-1}.

  \begin{lemma}
    Let $M,N \geq 1$ and $r>2(M+N+1)$, then
    \begin{equation*}
      \sup_{-(M+N)<n<(M+N)} J_{n}^2(r)
      \,\le\, \frac{b}{r} \qquad \text{with $b \approx 0.7595$}\,.
    \end{equation*}
  \end{lemma}

  \begin{proof}
    Since $J_{-n}^2(r)=J_n^2(r)$ we may assume w.l.o.g.\ that $n\geq0$.
    Let $\eta := \sqrt{(n+\frac12)(n+\frac32)}$.
    Then $\frac{\mu}4 = \eta^2 = n^2+2n+\frac34$, i.e.\
    \begin{equation}
      \label{eq:ProofKrasikov1}
      \frac{3}4 \,\leq\, \eta^2 \,\leq\, (n+1)^2
    \end{equation}
    and therefore our assumption $r>2(M+N+1)$ implies for $0\leq n<M+N$
    that 
    \begin{equation}
      \label{eq:ProofKrasikov2}
      r \,>\, 2(n+1) \,\geq\, 2 \eta \,.
    \end{equation}
    Accordingly,
    \begin{equation*}
      \frac12 \sqrt{\mu+\mu^{\frac32}}
      \,=\, \frac12 \sqrt{4\eta^2+(4\eta^2)^\frac32}
      \,=\, \eta \sqrt{1+\frac1{(4\eta^2)^{\frac13}}}
      \,\leq\, \eta \sqrt{1+\frac1{(4\bigl(\frac34\bigr)^2)^{\frac13}}}
      \,\leq\, \sqrt2 \eta
      \,\leq\, \frac{r}{\sqrt2} 
      \,\leq\, r \,.
    \end{equation*}
    This shows that the assumptions of theorem~2 of \cite{Kra06} are
    satisfied. 

    Next we consider \eqref{eq:KrasikovThm2} and further estimate its
    right hand side:
    \begin{equation*}
      \begin{split}
        J_n^2(r) 
        &\,\leq\, \frac{4(4r^2-(2n+1)(2n+5))}
        {\pi((4r^2-\mu)^{3/2}-\mu)}
        \,\leq\, \frac{4(4r^2-4\eta^2)}
        {\pi(8(r^2-\eta^2)^\frac32-4\eta^2)}\\
        &\,=\, \frac2\pi \frac1{(r^2-\eta^2)^\frac12} 
        \frac{1}{1-\frac12\frac{\eta^2}{(r^2-\eta^2)^{\frac32}}}
        \,=\, \frac2\pi \frac1r
        \frac1{\bigl(1-\bigl(\frac{\eta}{r}\bigr)^2\bigr)^\frac12} 
        \frac{1}{1-\frac12\frac{\eta^2}{(r^2-\eta^2)^{\frac32}}} \,.
      \end{split}
    \end{equation*}
    Since $r>2(M+N+1)\geq 6$, applying \eqref{eq:ProofKrasikov1} and
    \eqref{eq:ProofKrasikov2} yields
    \begin{equation*}
      \frac{\eta^2}{(r^2-\eta^2)^{\frac32}}
      \,=\, \frac1r \frac{\bigl(\frac{\eta}{r}\bigr)^2}
      {\bigl(1-\bigl(\frac{\eta}{r}\bigr)^2\bigr)^{\frac32}}
      \,\leq\, \frac1r \frac{\frac14}{\bigl(\frac34\bigr)^{\frac32}}
      \,=\, \frac1{3\sqrt{27}} \,,
    \end{equation*}
    whence
    \begin{equation*}
      J_n^2(r) 
      \,\leq\, \frac{2}{\pi} \frac1r 
      \Bigl(\frac43\Bigr)^{\frac12}
      \frac{1}{1-\frac12\frac1{3\sqrt{27}}}
      \,=\, \frac{b}{r} \qquad \text{with } b \approx 0.7595 \,.
    \end{equation*}
  \end{proof}

  \section{Some elementary inequalities}
  \label{app:SumEstimates}
  Here we prove some elementary inequalities that we haven't been able
  to find in the literature. 

  \begin{lemma}
    Let $a_1,\ldots,a_I \in \R$. 
    Then
    \begin{itemize}
    \item[(a)]
      \begin{equation}
        \label{eq:BasicIneq1}
        \sum_{i=1}^I \sum_{j\not=i}a_ia_j 
        \,\leq\, (I-1) \sum_{i=1}^I a_i^2 \,,
      \end{equation}
    \item[(b)]
      \begin{equation}
        \label{eq:BasicIneq2}
        \Bigl( \sum_{i=1}^I a_i \Bigr)^2
        \,\leq\, I \sum_{i=1}^I a_i^2 \,,
      \end{equation}
    \item[(c)]
      \begin{equation}
        \label{eq:BasicIneq3}
        \sum_{i=1}^I \sum_{j\not=i} a_ia_j 
        \,\leq\, \frac{I-1}{I} \Bigl(\sum_{i=1}^I a_i\Bigr)^2 \,.
      \end{equation}
    \end{itemize}
  \end{lemma}

  \begin{proof}
    \begin{itemize}
    \item[(a)]
      \begin{equation*}
        0 
        \,\leq\, \sum_{i=1}^I \sum_{j\not=i}(a_j-a_i)^2
        \,=\, \sum_{i=1}^I \sum_{j\not=i}(a_i^2-2a_ia_j+a_j^2)
        \,=\, 2(I-1) \sum_{i=1}^I a_i^2 
        - 2 \sum_{i=1}^I \sum_{j\not=i} a_ia_j \,.
      \end{equation*}

    \item[(b)] Using \eqref{eq:BasicIneq1} we find that
      \begin{equation*}
        \Bigl(\sum_{i=1}^I a_i\Bigr)^2
        \,=\, \sum_{i=1}^I a_i^2 
        + \sum_{i=1}^I\sum_{j\not=i}a_ia_j
        \,\leq\, I\sum_{i=1}^I a_i^2 \,.
      \end{equation*}
    \item[(c)] Proceeding as in (b) but applying \eqref{eq:BasicIneq1}
      the other way round yields 
      \begin{equation*}
        \Bigl(\sum_{i=1}^I a_i\Bigr)^2
        \,=\, \sum_{i=1}^I a_i^2 + \sum_{i=1}^I\sum_{j\not=i}a_ia_j
        \,\geq\, \frac{I}{I-1} \sum_{i=1}^I\sum_{j\not=i}a_ia_j \,.
      \end{equation*}
    \end{itemize}
  \end{proof}

\end{appendix}



\end{document}